\documentclass[12pt,reqno, a4paper]{amsart}
\PassOptionsToPackage{dvipsnames,svgnames}{xcolor}
\usepackage{amsmath, amssymb, amsfonts, tikz, enumerate, graphicx, textcomp, caption, amsthm, listings, verbatim, mathrsfs, url,datetime}
\usepackage[safe]{tipa}
\usepackage[sort,round,numbers]{natbib}
\usepackage[colorlinks=true,hyperindex, linkcolor=magenta, pagebackref=false, citecolor=cyan,pdfpagelabels]{hyperref}
\usepackage{xcolor}
\usepackage{array}
\usepackage[margin=3cm]{geometry}
\usetikzlibrary{calc, arrows}
\tikzset{curve/.style={settings={#1},to path={(\tikztostart)
    .. controls ($(\tikztostart)!\pv{pos}!(\tikztotarget)!\pv{height}!270:(\tikztotarget)$)
    and ($(\tikztostart)!1-\pv{pos}!(\tikztotarget)!\pv{height}!270:(\tikztotarget)$)
    .. (\tikztotarget)\tikztonodes}},
    settings/.code={\tikzset{quiver/.cd,#1}
        \def\pv##1{\pgfkeysvalueof{/tikz/quiver/##1}}},
    quiver/.cd,pos/.initial=0.35,height/.initial=0}

\tikzset{tail reversed/.code={\pgfsetarrowsstart{tikzcd to}}}
\tikzset{2tail/.code={\pgfsetarrowsstart{Implies[reversed]}}}
\tikzset{2tail reversed/.code={\pgfsetarrowsstart{Implies}}}
\tikzset{no body/.style={/tikz/dash pattern=on 0 off 1mm}}
\newcolumntype{L}{>{$}l<{$}}
\usetikzlibrary{calc, arrows}
\newcommand{\stackyq}{\text{ }/\kern-0.3cm /\text{ }}

\newtheorem{manualtheoreminner}{Theorem}
\newenvironment{manualtheorem}[1]{%
  \renewcommand\themanualtheoreminner{#1}%
  \manualtheoreminner
}{\endmanualtheoreminner}

\newtheorem{manualcorollaryinner}{Corollary}
\newenvironment{manualcorollary}[1]{%
  \renewcommand\themanualcorollaryinner{#1}%
  \manualcorollaryinner
}{\endmanualcorollaryinner}

\DeclareMathOperator{\Aut}{Aut}

\DeclareMathOperator{\Ind}{Ind}
\newcommand{\mr}{\mathrm}
\DeclareMathOperator{\Hom}{Hom}

\DeclareMathOperator{\Jac}{Jac}
\DeclareMathOperator{\Sym}{Sym}

\DeclareMathOperator{\Ext}{Ext}

\DeclareMathOperator{\QCoh}{QCoh}

\DeclareMathOperator{\Spf}{Spf}

\DeclareMathOperator{\Def}{Def}

\DeclareMathOperator{\Pic}{Pic}

\DeclareMathOperator{\Spec}{Spec}

\usepackage{tikz-cd}\tikzset{node distance=2cm, auto}
\newcommand{\SO}{\mathbb{S}}

\newcommand{\Z}{\mathbb{Z}}

\newcommand{\F}{\mathbb{F}}
\newcommand{\G}{\mathbb{G}}
\newcommand{\Q}{\mathbb{Q}}

\newcommand{\A}{\mathbb{A}}

\newcommand{\PP}{\mathbb{P}}

\newcommand{\OO}{\mathcal{O}}

\newcommand{\mc}{\mathcal}

\makeatletter
\gdef\rr@partsub{}
\newcommand\partsubtitle[1]{\gdef\rr@partsub{#1}}
\renewcommand\part[1]{%
  \par\vspace{2.2\linespacing plus \linespacing}%
  \refstepcounter{part}%
  \ifx\rr@partsub\empty\def\rr@tt{#1}\else\edef\rr@tt{\unexpanded{#1}: \unexpanded\expandafter{\rr@partsub}}\fi
  \addcontentsline{toc}{part}{\tocpart{\partname}{\thepart}{\rr@tt}}%
  \begingroup\centering\normalfont
  {\large\scshape\partname\ \thepart}\\[0.45\linespacing]
  {\LARGE\bfseries #1}\par
  \ifx\rr@partsub\empty\else\vspace{0.3\linespacing}{\Large\bfseries\rr@partsub}\par\fi
  \endgroup
  \gdef\rr@partsub{}
  \vspace{1.3\linespacing}\@afterindentfalse\@afterheading}
\makeatother

\newcommand{\bb}{\mathbb}
\newcommand{\Art}{\ms{Art}}

\newcommand{\ms}{\mathsf}
\newcommand{\mf}{\mathfrak}

\definecolor{codegray}{gray}{0.9}

\newtheorem{defn}{Definition}[section]
\newtheorem{theorem}[defn]{Theorem}
\newtheorem{lemma}[defn]{Lemma}
\newtheorem{cor}[defn]{Corollary}

\newtheorem{example}[defn]{Example}

\newtheorem{theorem*}{Theorem}
\newtheorem{conv}{Convention}
\newtheorem{proposition}[defn]{Proposition}
\newtheorem{construction}{Construction}

\newtheoremstyle{recall}{\topsep}{\topsep}{\itshape}{}{\bfseries}{.}{ }{\thmname{#1}\thmnumber{ #2}\thmnote{ #3}}
\definecolor{rinpink}{RGB}{214,30,130}
\newtheoremstyle{pinkq}{1.1\baselineskip plus .3\baselineskip minus .1\baselineskip}{1.1\baselineskip plus .3\baselineskip minus .1\baselineskip}{\addtolength{\leftskip}{2em}\addtolength{\rightskip}{2em}}{}{\bfseries\color{rinpink}}{.}{ }{\thmname{#1}\thmnote{ (#3)}}
\theoremstyle{remark}
\newtheorem{remark}[defn]{Remark}

\newtheorem*{notation}{Notation}

\theoremstyle{pinkq}
\newtheorem*{dreamq}{Dream Question}
\newtheorem*{question}{Question}
\newtheorem*{ranswer}{Rough Answer}
\newtheorem*{answer}{Answer}
\newcommand{\stmt}[1]{{\color{blue}#1}}
\theoremstyle{remark}

\usepackage{tcolorbox}

\usepackage{booktabs}
\providecommand{\mm}{\mathfrak{m}}

\providecommand{\bD}{\mathbb{D}}
\providecommand{\Gm}{\mathbb{G}_m}
\providecommand{\Ga}{\mathbb{G}_a}
\newtheorem{corollary}[defn]{Corollary}

\DeclareMathOperator{\ord}{ord}

\providecommand{\GL}{\mathrm{GL}}

\newcommand{\chic}{\chi_{\mathrm{c}}} 
\DeclareMathOperator{\rad}{rad}

\title[Tiny and Unique]{Tiny and Unique: \\ The Lubin--Tate Action at Height $p-1$ \\ from the Minimal Curve with Its Symmetries}
\author{Rin Ray}
\date{} 

\begin{document}

\begin{abstract}
Let $p$ be odd and $G':=\Z/p\rtimes\Z/(p-1)^2$ the maximal finite subgroup of the Morava stabilizer group at height $p-1$. Inverse Galois theory produces from $G'$ alone a curve $X$, the unique curve of minimal genus with $\Aut(X)\simeq G'$; its ramification, its field of definition and its equation are consequences of the group, not choices. We prove a $G'$-equivariant equivalence between the deformations of $X$ and Lubin--Tate space, so that the Lubin--Tate action of $G'$ is the action of $\Aut(X)$ on deformations of the curve. The proof is a coordinate-free Kodaira--Spencer argument reducing to a single character count. The action becomes explicit: $G'$ acts through $\F_p\rtimes\F_p^\times$ shifting and scaling $p+1$ points on $\PP^1$. From this we compute $H^*(G',\pi_*E_{p-1})$ and its Tate cohomology. One identity, $\pi^{p-1}=-p$, runs through every section.
\end{abstract}

\maketitle

\begin{center}
\includegraphics[width=13cm]{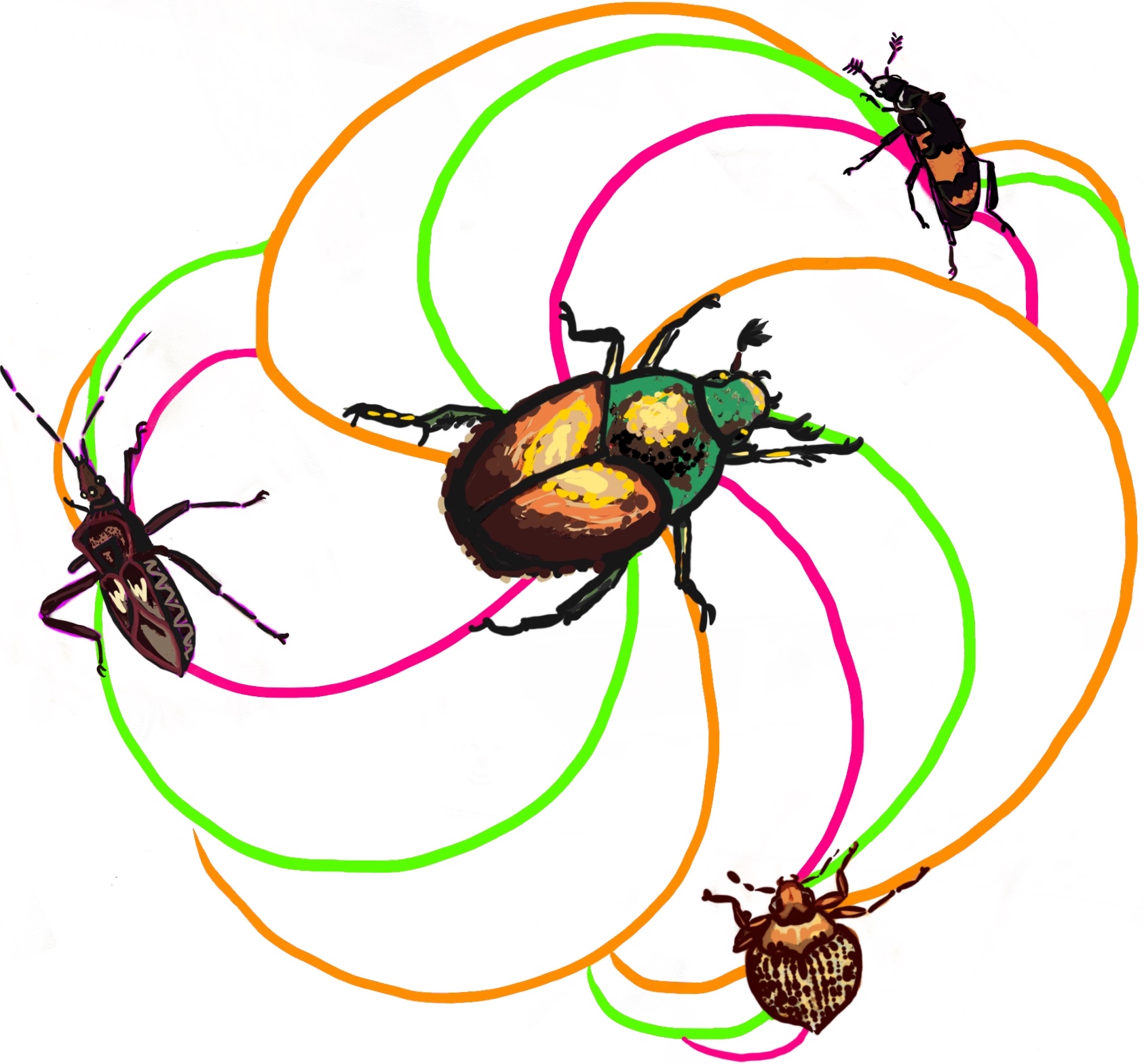}
\captionof{figure}{Moduli stack of ramification points (the beetles) of a cover. As the beetles walk, the ramification points deform.}
\end{center}

\newpage
\tableofcontents

\section*{Introduction}
\subsection{Motivation: the Lubin--Tate Action, and Dreams of a Perfect World}
\label{motivation}

Fix an odd prime $p$ and a formal group $H_h$ of height $h$ over a perfect field $k$ of characteristic $p$. Its automorphism group $\G_h$, the Morava stabilizer group, acts on the deformation space $\Def_{H_h}$ \cite{LT66}: this is the {\color{Bittersweet}{Lubin--Tate action}}. Five footpaths lead to it, through arithmetic, physics, chromatic homotopy, Shimura varieties and computation \cite[\S 0.1]{ray}. This paper adds a sixth, the one you can draw.

Rationally, the action of the full profinite group is understood \cite{BSSW}. What remains is $p$-torsion, and $p$-torsion lives in finite subgroups (Subsection \ref{whyfinite}). A finite group, unlike $\G_h$, can be the automorphism group of a curve, and the dream is a curve with decoration whose deformations \emph{are} Lubin--Tate space. At height $2$ that curve is elliptic and Serre--Tate makes the dream come true. Above height $2$:

\begin{question}
For which heights $h$ and primes $p$ does the action of the maximal finite subgroup $G'\subseteq\G_h$ on $\Def_{H_h}$ admit a geometric model, in the sense of \cite{ray}: a moduli problem of geometric objects together with a $G'$-equivariant equivalence onto $\Def_{H_h}$? And can it be built uniformly, for all interesting $h$ and $p$ at once?
\end{question}

We hunger for more than existence. A model found by inspection at each height is a catalogue; a model \emph{forced} by the group is a theorem. In a perfect world the curve is a consequence of the group, the action is a consequence of the curve, and the equivalence is a consequence of a single character count --- nothing chosen, everything ordained.

\begin{dreamq}
Is there a construction taking as input \emph{only} the abstract group $G'\subseteq\G_h$ and returning its Lubin--Tate action in geometric form:
\begin{itemize}
\item a moduli stack $\mc{M}$ of decorated curves with $G'$-action, canonically determined by $G'$;
\item on which the action is \emph{computable}: explicit generators acting on an explicit representing ring;
\item together with a $G'$-equivariant equivalence of moduli problems
$$\Phi\colon\mc{M}\xrightarrow{\ \simeq\ }\Def_{H_h},$$
the target carrying the canonical Lubin--Tate action?
\end{itemize}
\end{dreamq}

\begin{ranswer}
Yes at $h=p-1$, the first height carrying $p$-torsion, and there the dream version holds. The maximal finite subgroup is $G':=\Z/p\rtimes_{\chic}\Z/(p-1)^2$ with $\chic$ surjective onto $\F_p^\times$ --- the twisted love of multiplication and addition --- in a single conjugacy class \cite{bouj, Hewett} (recalled as \cite[Lemma 0.1]{ray}). Write $G:=C_{p-1}$ for the center of $G'$, and $\Def_{(X,G)}$ for the deformations of a curve $X$ together with its $G$-action \cite[Def.\ 7.6]{ray}. The whole paper is one chain, with the group as its only input:
\begin{gather*}
G'\ \Rightarrow\ X\ \Rightarrow\ \Def_{(X,G)}
\qquad\text{and then}\\[0.3\baselineskip]
\mr{Conf}(\PP^1,p+1)\ \overset{L}{\underset{\simeq}{\longleftarrow}}\ \Def_{(X,G)}\
\overset{\widehat{\mc{F}}^1}{\underset{\simeq}{\longrightarrow}}\
\Def^{FG^{\text{ dim }1}}_{\widehat{\mc{F}}^1(X,G)}.
\end{gather*}
The double arrows are inverse Galois theory: the curve $X$, its ramification, its field of definition, and finally its equation are consequences of $G'$, not choices (Theorem \ref{tinyandunique}). The functor $L$ makes the model computable, as a configuration space of branch points on which $G'$ acts by affine motions (Theorems \ref{theoremb1} and \ref{actup}). The functor $\widehat{\mc{F}}^1$ is a \emph{$G'$-equivariant equivalence of deformation stacks}: it identifies the model with Lubin--Tate space, action and all (Theorem \ref{equivalentexchange}). Composing the two arrows, we get: the Lubin--Tate action is the symmetric action on configurations of $p$ points on $\A^1$, restricted to $G'$ through $\F_p\rtimes\F_p^\times\subset\Sigma_p$ (Theorem \ref{theoremd}).
\end{ranswer}

\subsection{What is New}
\label{whatisnew}

The curve $X\colon y^p-y=x^{p-1}$ over $\F_{p^{p-1}}$ is the unique curve of minimal genus with automorphism group $G'\simeq C_p\rtimes C_{(p-1)^2}$ (Theorem \ref{tinyandunique}): \textit{the group is shown to determine the curve}, and the curve is never chosen.

Its deformations, as a curve with its central $C_{p-1}$ action, form a stack that we show is \textit{$G'$-equivariantly equivalent as deformation stacks} to Lubin--Tate space at height $p-1$, and the functor $\widehat{\mc F}^1$ exhibiting the equivalence is constructed here using idempotents for the first time (Theorem \ref{equivalentexchange}). This equivalence is not Serre--Tate (except at $p=3$, where $X$ is an elliptic curve and Theorem \ref{equivalentexchange} reduces to Serre--Tate). Serre--Tate identifies deformations of a curve with deformations of the $p$-divisible group of its Jacobian, which here has dimension $(p-1)(p-2)/2$; we identify deformations of a curve with decoration with deformations of a single one-dimensional formal group of height $p-1$, cut out of the Jacobian by an idempotent.

\textit{Nothing is lifted by hand}: the action of $\Aut(X)$ on the universal curve is supplied immediately by deformation theory \cite[Lemma 7.3]{ray}, not by an equation. This gets rid of the need to find an explicit universal lift of the universal curve and a lift of the action of $\Aut(X)$ to an action on the universal curve, both of which were required by previous work.

The paper also connects $\Aut(X)$ to $\Sigma_p$ at the level of geometric deformation stacks. \textit{Ramification data replaces level structure}: the branch locus of $X\to X/C_{p-1}$ is a configuration of $p$ points on the line, deformations of the curve with its group are deformations of that configuration \cite[Thm.\ 7.41]{ray}, and $\Aut(X)/C_{p-1}=\F_p\rtimes\F_p^\times$ is the affine group inside $\Sigma_p$ permuting them.

\begin{remark} The height-one instance is familiar in a different guise: the $\Sigma_3$-action on elliptic curves with full level-$2$ structure is usually reached through $\GL_2(\F_2)\simeq\Sigma_3$ rather than through its geometric source, the permutation of the three branch points of $E\to E/C_2$; from the level-$2$ standpoint, the global moduli of elliptic curves with level-$2$ structure is what makes the duality of $\mr{tmf}$ computable \cite{vesna}.\end{remark}

\subsection{Why This Finite Subgroup}
\label{whyfinite}

Two spectra carry this arithmetic into homotopy theory. The {\color{Bittersweet}{Lubin--Tate spectrum}} $E_h$ is the even periodic ring spectrum with
$$\pi_*E_h\;=\;W(k)[[u_1,\dots,u_{h-1}]][u^{\pm1}],\qquad|u|=-2,$$
whose formal group is the universal deformation of $H_h$: its homotopy \emph{is} the Lubin--Tate deformation ring, made periodic. By Goerss--Hopkins--Miller it is an $E_\infty$-ring on which $\G_h$ acts by $E_\infty$-maps \cite{GH04,Rez98}. {\color{Bittersweet}{Morava $K$-theory}} $K(h)$, with $\pi_*K(h)=\F_p[v_h^{\pm1}]$ and $|v_h|=2(p^h-1)$, is its residue field; it sees exactly height $h$, and $L_{K(h)}$ restricts the stable homotopy category to its height-$h$ layer.

Three facts make $G'$ the right target: it is where the torsion is, it is what the $K(p-1)$-local sphere is built from, and it is where the famous elements were found. We take them in turn.

\emph{Where the torsion is.} The two-tower method, built on the equivalence of the Lubin--Tate and Drinfeld towers \cite{Faltings}, sees all of $\G_h$, but only after inverting $p$ or one stratum at a time. After inverting $p$ it recovers Lazard's rational cohomology \cite{lazard}, as Barthel--Schlank--Stapleton--Weinstein show in computing the rational homotopy of the $K(h)$-local sphere and proving the chromatic splitting conjecture rationally \cite{BSSW}. In coheight one it computes the cohomology with coefficients in the Gross--Torii field, which is how the author, with Barthel, Mann, Schlank, Senger, Weinstein and Zhou, proved the weak chromatic splitting conjecture in coheight one \cite{BMRSSWZ}; the prequel \cite{ray} is a companion to that work.

What the towers leave behind is $p$-torsion, and there the dichotomy is sharp. If $p-1\nmid h$, then $\G_h$ has no $p$-torsion, the Adams--Novikov spectral sequence of the $K(h)$-local sphere has a horizontal vanishing line at $s=h^2+1$ \cite{henn}, and for $p$ large everything is algebraic \cite{pstra,BSS20}. If $p-1\mid h$, then $\G_h$ has elements of order $p$ \cite{Hewett,bouj} and its cohomology never stops; above the virtual cohomological dimension it is Farrell--Tate cohomology, and at $h=p-1$ Symonds shows that this belongs to $G'$ alone \cite{symonds}:
$$\hat H^*\big(\mathbb S_{p-1};(E_{p-1})_*\big)^{\mr{Gal}}\ \cong\ \hat H^*\big(G';(E_{p-1})_*\big)^{\mr{Gal}}\otimes\Lambda(\xi_0,\dots,\xi_{p-2}).$$
Above filtration $(p-1)^2$, every Adams--Novikov class of the $K(p-1)$-local sphere is a class of $G'$ times an exterior class $\xi_i$ coming from the torsion-free part of $\G_{p-1}$. Inverting $\beta$ passes from one to the other, and this is how Bobkova--Lachmann--Li--Lima--Stojanoska--Zhang find a vanishing line at $s=3(p-1)^2+1$ and bound the exotic Picard group \cite{BLLLSZ}.

\emph{What the sphere is built from.} Henn's finite resolutions assemble $L_{K(p-1)}S$ from finitely many spectra $E^{hF}$ with $F$ finite \cite{henn07}; at height $2$ these are the duality resolutions \cite{ghmr,bobkovagoerss}. Each $E_h^{hF}\to E_h$ is a faithful $K(h)$-local $F$-Galois extension \cite{rognes,devhop}, so the $E_h^{hF}$ are the intermediate fields of the $K(h)$-local sphere, with $KO=KU^{hC_2}$ and $K(2)$-local $\mr{tmf}$ the first two. They are rigid: their Tate spectra vanish \cite{heardtate}, they are Gross--Hopkins self-dual up to a computable shift \cite{BBGS}, and at height $p-1$ their Picard groups are cyclic, generated by the suspension \cite{HMS}.

\emph{Where the famous elements live.} Ravenel detected the odd-primary Arf invariant classes through a single element of order $p$ \cite{Rav}. Nave proved that the Smith--Toda complex $V(\tfrac{p+1}{2})$ does not exist, through $E_{p-1}^{hG'}$ \cite{nave}. Hill--Hopkins--Ravenel settled Kervaire invariant one at $p=2$ through a $C_8$ acting at height $4$ \cite{HHR}, and Hahn--Shi Real-oriented every Lubin--Tate theory through the central $C_2$ \cite{hahnshi}.

So $G'$ is not a toy version of the problem. It holds the problem's torsion, its building blocks and its celebrities, and it is the part a single curve can carry. At $h=p-1$ its Tate cohomology is periodic, and that periodicity is the shape of Theorem \ref{theoreme}.

\subsection{The One Identity}
\label{oneidentity}
The construction is organized around one identity of local number theory. In $\Z_p[\zeta_p]$, the element $\pi=\zeta_p-1$ satisfies
\[
\boxed{\ \pi^{p-1}=-p\cdot(\text{unit})\ }
\] because $[p](T)/T=((1+T)^p-1)/T$ is Eisenstein of degree $p-1$. The tame group $C_{(p-1)^2}\subset G'$ acts on a one-dimensional tangent line on both sides of the correspondence (the cotangent line of the curve at its wild point, and the tangent line $\pi_2E_{p-1}$ of the formal group), and on both sides the conjugation character of the wild $C_p$ is the $(-m)$-th power of that tangent character. On the curve, $m=p-1$ is forced by minimality; in $\Q_p(\zeta_p)$, $e=p-1$ is forced by the Eisenstein polynomial. Same exponent, two computations:

\begin{center}
\renewcommand{\arraystretch}{1.4}
\begin{tabular}{@{}p{0.45\textwidth}@{\hspace{1em}}p{0.5\textwidth}@{}}
\toprule
\textbf{the curve $X$ in characteristic $p$} &
\textbf{the field $\Q_p(\zeta_p)$} \\
\midrule
tangent character $\theta$ of $C_{(p-1)^2}$ at the wild point &
tangent character $\tau\mapsto\rho$ of $C_{(p-1)^2}\subset\mc{O}_D^\times$
on $\pi_2E_{p-1}$ \\
conjugation character $\chic$ of $C_{(p-1)^2}$ on $C_p$ &
cyclotomic character $\tau\mapsto\chi_{\mathrm{cyc}}(\tau)$ on $\mu_p$ \\
ramification break $m$ of $X\to X/C_p$ &
ramification index $e$ of $\Q_p(\zeta_p)/\Q_p$ \\
$\chic=\theta^{-m}$ \hfill (Lemma \ref{lem:char}) &
$\chi_{\mathrm{cyc}}=\rho^{-(p-1)}$ \hfill \cite[Lemma 6.3]{ray} \\
$m=r-1$ for $r$ colliding branch points &
$\deg\big([p](T)/T\big)=p-1$ \\
\midrule
\multicolumn{2}{@{}l}{\emph{and, following \cite{higherlevels}, the same identity as a collision and as a period:}} \\
$\sigma$-fixed locus of $\Def_{(X,G)}$: all $p$ affine branch points collide &
$\sigma$-fixed locus of $\Def_F$: $\Spf W[\zeta_p]$, $\pi^{p-1} = -p$ Eisenstein \\
scalar of the Kodaira--Spencer derivative $d\widehat{\mc F}^1$, a unit (Lemma \ref{lem:count}) &
scalar of the Gross--Hopkins derivative, a unit in the integral frame \\
\bottomrule
\end{tabular}
\end{center}

{\stmt{The break of the curve, the exponent of the semidirect product, the degree of the field of definition, and the height are all $p-1$, for one reason.}}

\begin{remark} The constant $\pi$ with $\pi^{p-1}=-p$, which recurs throughout this paper, is the $p$-adic counterpart of $2\pi i$ for the additive character: the Artin--Schreier sheaf attached to $f$ is the characteristic-$p$ analogue of
the exponential $\mathcal D$-module $e^{f}$, and its trace function is the finite-field version of the oscillatory integral $\int e^{f}$ of a
Landau--Ginzburg model with superpotential $f$. \end{remark}

\subsection{Main Results}
\label{results}

Notation is fixed in Subsection \ref{conventions}, near the end of this introduction. First, the group determines the curve.

\begin{manualtheorem}{\ref{tinyandunique}}(Tiny and Unique) \stmt{
Let $p$ be an odd prime and $k$ a perfect field containing $\F_{p^{p-1}}$.
Among smooth projective geometrically connected curves over $k$ with a
faithful $G'$-action, the minimal genus is $g=\frac{(p-1)(p-2)}{2}$,
attained precisely by the $k$-forms of
\[
X\colon\quad y^p-y\ =\ z^{p-1}.
\]
On the minimal stratum the following are consequences, not hypotheses:

\begin{enumerate} 
\item $X/P\cong\PP^1_k$ for $P$ the Sylow $p$-subgroup;
\item $X\to X/P$ is branched at one point, with break $m=p-1$;
\item $X$ has $p$-rank $0$;
\item $\Aut(X)=G'$ for $p\ge5$, and $\Aut(X,O)=G'$ at $p=3$.
\item If $(k^\times)^{p-1}=k^\times$, in particular if $k=\bar k$, the curve is
\textit{unique} up to $k$-isomorphism.
\end{enumerate}}
\end{manualtheorem}

Next, the curve determines the model $\Def_{(X,G)}$: deformations of $X$ together with its action of the central subgroup $G:=C_{p-1}$, that is, lifts of $X$ to Artinian $W(k)$-algebras to which the $G$-action lifts too \cite[Def.\ 7.6]{ray}. Unlike Lubin--Tate space, this one you can watch move, with one deformation variable for each branch point, up to translating and rescaling the line: \stmt{$G'$ acts on the moduli of the curve by shifting and scaling the $p$ affine branch points.}

\begin{manualtheorem}{\ref{actup}}(Act Up) \stmt{There are $G'$-equivariant identifications
$$\Def_{(X,G)}\ \overset{L}{\underset{\mathrm{B.1}}{\simeq}}\ \Def_{(\A^1,p)}\ =\ \mr{Conf}(\A^1,p)^\wedge_{\F_p}\ \underset{\mathrm{B.2}}{\simeq}\ \Spf\big(\Sym(\overline\rho)[N^{-1}]^\wedge_I\big)\stackyq\G_m.$$}
Here $\mr{Conf}(\A^1,p)^\wedge_{\F_p}$ is the formal completion of the configuration space of $p$ points on the line at the configuration $\F_p=\{0,1,\dots,p-1\}$, the affine branch points of $X\to X/G$: its points are the deformations of those $p$ points. On the right, $\overline\rho$, $N$, $I$ are as in Subsection \ref{conventions}, $G'$ acts through $\F_p\rtimes\F_p^\times\subset\Sigma_p$, and the central $C_{p-1}$ acts only through the grading.
\end{manualtheorem}

\begin{center}
\captionsetup{justification=centering}
\begin{minipage}[t]{0.50\textwidth}\centering
\parbox[c][5.2cm][c]{\linewidth}{\centering\includegraphics[width=\linewidth,height=5.2cm,keepaspectratio]{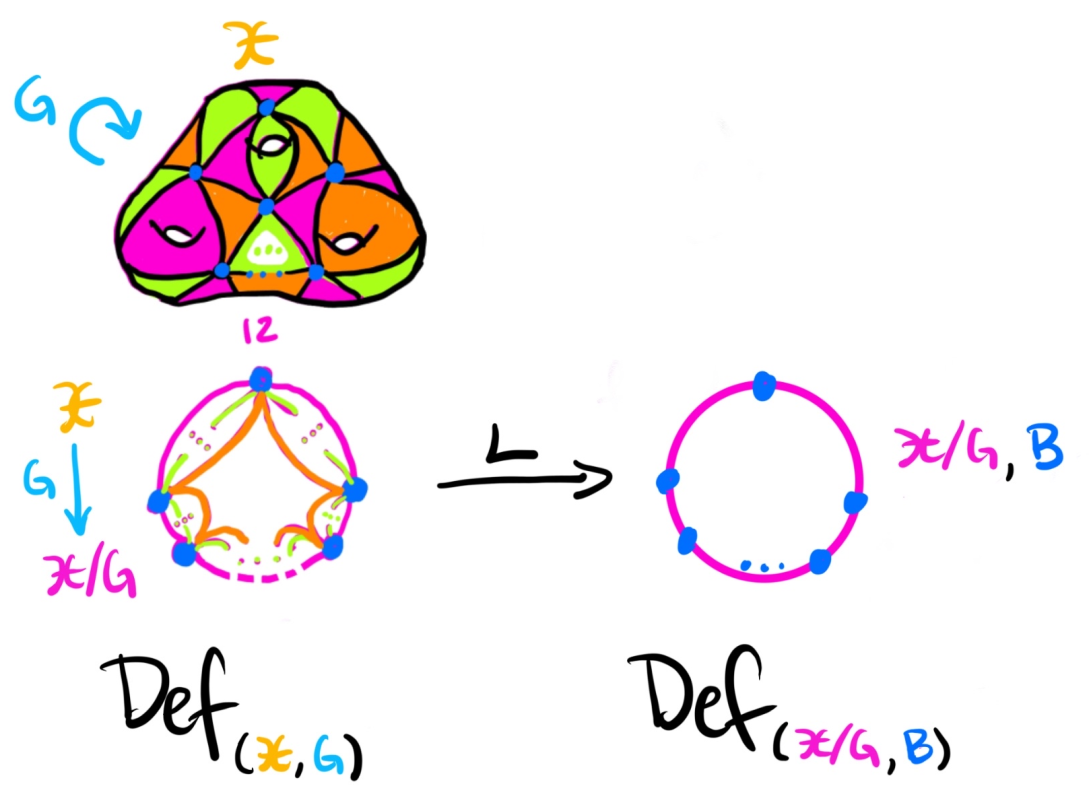}}
\captionof{figure}{Theorem \ref{theoremb1}: the functor $L$ sends a curve with its $G$-action to its quotient together with the branch points.}
\end{minipage}\hfill
\begin{minipage}[t]{0.46\textwidth}\centering
\parbox[c][5.2cm][c]{\linewidth}{\centering\includegraphics[width=\linewidth,height=5.2cm,keepaspectratio]{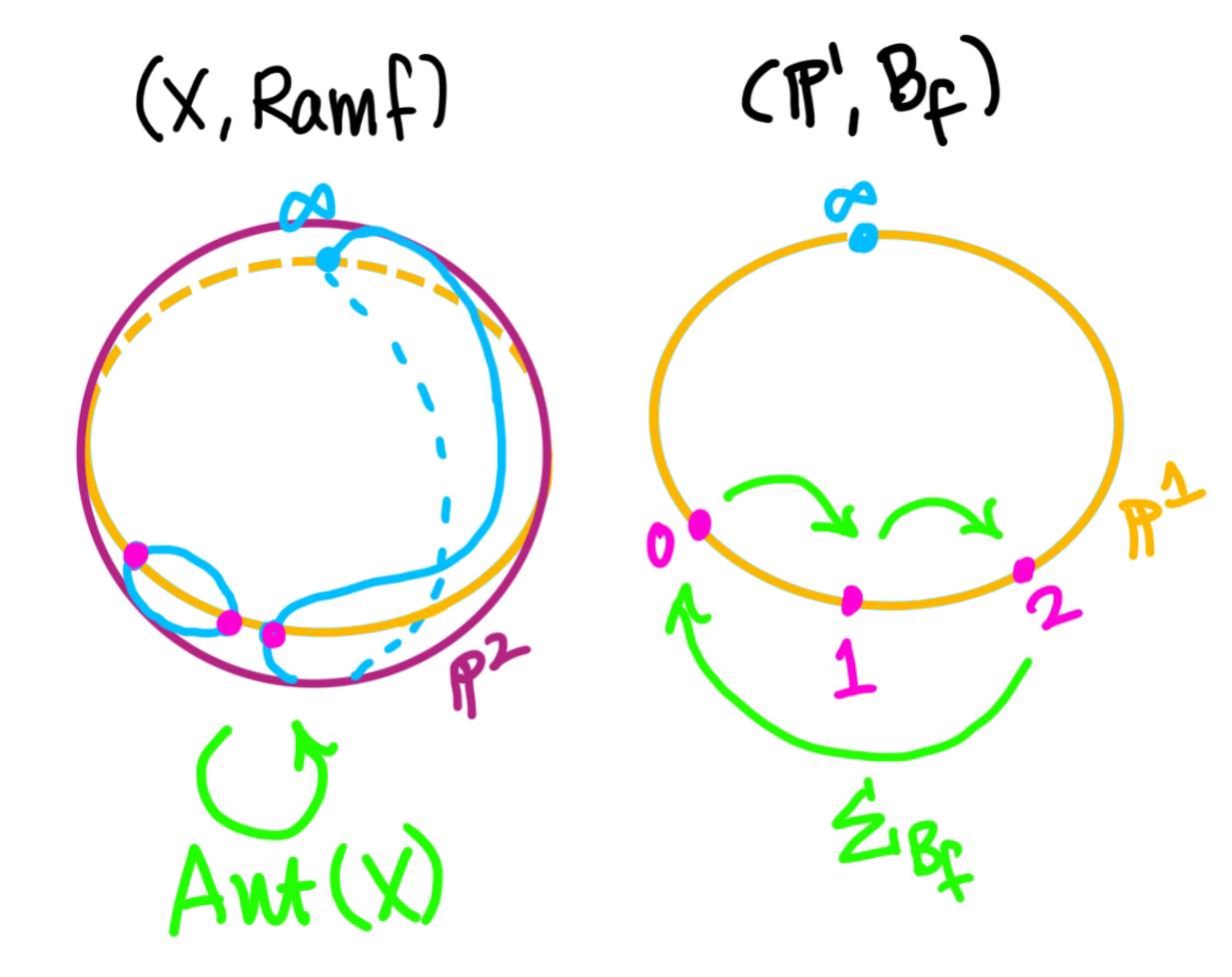}}
\captionof{figure}{Theorem \ref{theoremb2}: $\Aut(X)$ acts on configuration space by shifting and scaling the branch points on $\PP^1$.}
\end{minipage}
\end{center}

Shifting and scaling is the entire wild action of $C_p$, encoded
as combinatorics of ramification; the mechanism is explained in Subsection
\ref{symmetricsultan} (for the special fiber) and \cite[Thm.\ 7.41]{ray} (for the deformation stack).

Now the heart of the paper, and the theorem the title promises. Nothing in it is chosen, not even the height: \stmt{the automorphisms of the curve impose the height $p-1$ of the formal group}, since through Manin's formula they fix the slopes of Frobenius on the Jacobian (Theorem \ref{splittydooda}). We use $G'$ to construct the $G'$-equivariant functor $\widehat{\mc{F}}^1$ by taking the formal group of the Jacobian and projecting onto the idempotent summand cut by the central tame character (Theorem \ref{splittydooda} and Definition \ref{def:functor}). The functor is always surjective, and its injectivity is the minimal-genus theorem seen from the formal group.

\begin{tcolorbox}[colframe=rinpink,colback=white,boxrule=0.9pt,arc=2pt,left=6pt,right=6pt,top=4pt,bottom=4pt,title={\bfseries Main Theorem: Equivalent Exchange},colbacktitle=rinpink,coltitle=white]
\begin{manualtheorem}{\ref{equivalentexchange}}(Equivalent Exchange)
{\stmt{Let $G:=C_{p-1}\subseteq Z(G')$ and $X$ as above.
\begin{enumerate}
\item[(a)] The functor
$$\widehat{\mc{F}}^1\colon\Def_{(X,G)}^\star\longrightarrow
\big(\Def_{\widehat{\mc{F}}^1(X,G)}^{FG^{\text{ dim }1}}\big)^\star$$
is a $G'$-equivariant equivalence of formal moduli problems, where
$\widehat{\mc{F}}^1(X,G)=F$ is a one-dimensional formal group of height $p-1$
and the $G'$-action on the target is the canonical Lubin--Tate action.
Consequently
$$\Def_{(X,G)}^\star\stackyq G'\ \simeq\
\big(\Def_{\widehat{\mc{F}}^1(X,G)}^{FG^{\text{ dim }1}}\big)^\star
\stackyq G'.$$
\item[(b)] (Graded Equivalent Exchange) The functor
$$\widehat{\mc{F}}^1_{\mf{Lie}} \colon \Def_{(X, G)}^{\mr{Curve}_{G}^{\mr{Lie }\chi^1}} \longrightarrow \Def_{F}^{\mc{M}_{\mr{fg}_1}^{\mr{Lie} \simeq \mr{triv}}}$$
is a $G'$-equivariant equivalence of formal moduli problems, a map of $\G_m$-torsors over \emph{(a)}, where the $G'$-action on the target is the graded Lubin--Tate action; so the identification of \emph{(a)} is one of graded rings, with $\pi_*E_{p-1}$ on the right.
\end{enumerate}}}
\end{manualtheorem}
\end{tcolorbox}

\begin{center}\includegraphics[width=7cm]{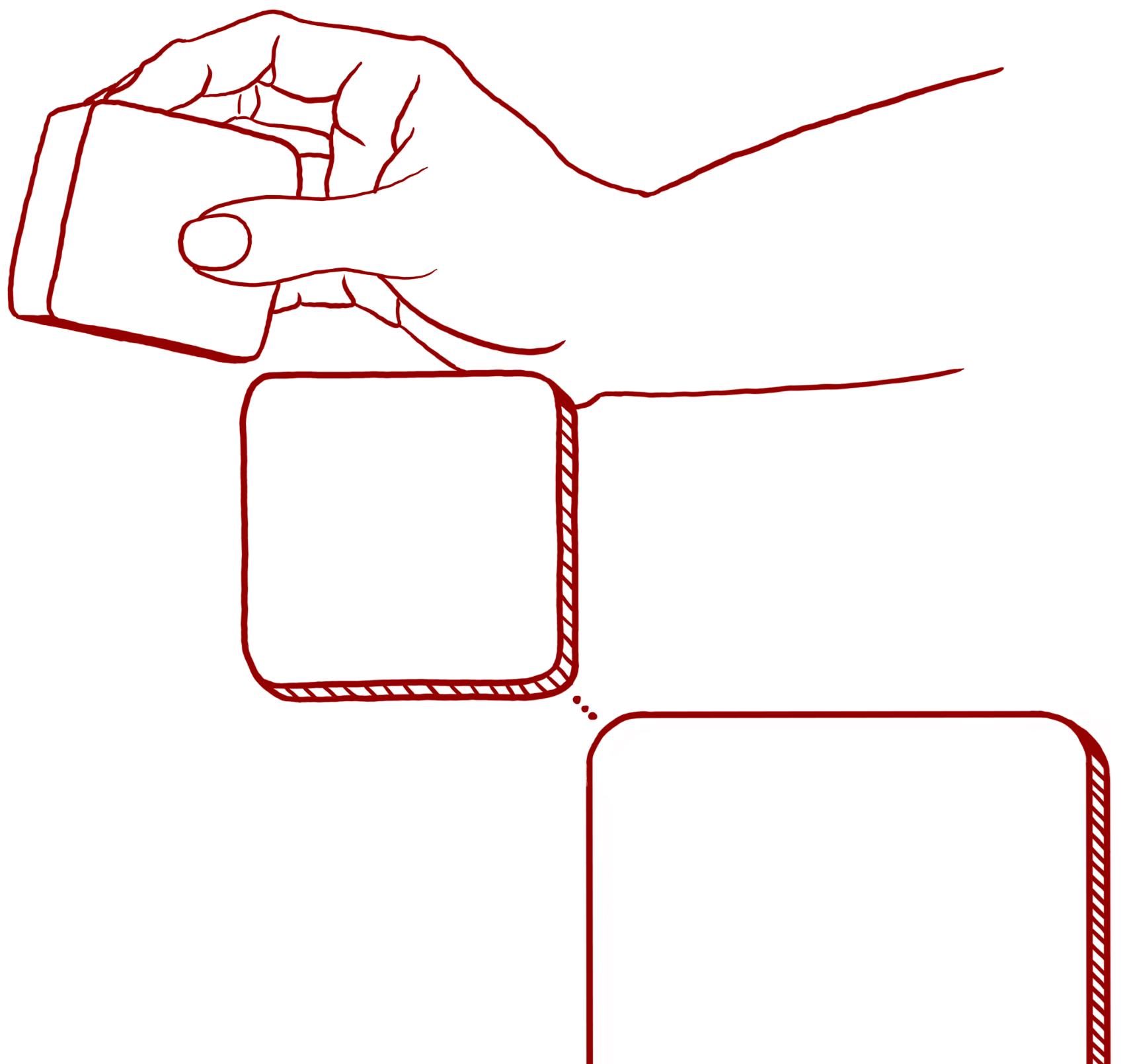}
\captionof{figure}{The $p$-divisible group of the Jacobian, split into lower-dimensional summands (Theorem \ref{splittydooda}). The functor $\widehat{\mc F}^1$ takes the one-dimensional summand off the top.}\end{center}

\begin{remark}
Since $\Def_F^\star\stackyq\G_h\simeq\widehat{\mc{H}}(h)$ is the geometric form of
the Devinatz--Hopkins description, Theorem \ref{equivalentexchange}
replaces a sub-Galois extension of $L_{K(h)}\SO$ by a moduli problem of
curves.
\end{remark}

Finally, we put the model to work. A geometric model is worth having only if it can be used, and this one computes: it hands us $\pi_*E_{p-1}$ as a $G'$-representation, and its Tate cohomology.

\begin{manualtheorem}{\ref{theoremd}}(Up Shot) \stmt{Let
$\widetilde\Lambda:=\Sym(\overline\rho)$. As $G'$-representations,
$$(E_{p-1})_*\ \simeq\ \widetilde\Lambda[N^{-1}]^\wedge_I\ =:\ \Lambda,$$
and
$$H^*(G',E_*)\ \simeq\ H^*(G',\widetilde\Lambda)[\Delta^{-1}]^\wedge_I\
\simeq\ H^*(G',\widetilde\Lambda)[\Delta^{-1}],$$
the disappearance of the completion being Lemma \ref{idealhop}.}
\end{manualtheorem}

\begin{manualtheorem}{\ref{theoreme}}\stmt{The Tate cohomology of
the $G'$-module $(E_{p-1})_*$ is that of $\widetilde\Lambda[\Delta^{\pm1}]$:
$$\hat H^*\big(G',\widetilde\Lambda[\Delta^{\pm1}]\big)\ \simeq\
\bb{W}/p[\alpha,\beta^{\pm1},\Delta^{\pm1}]/\alpha^2,$$
with $|\alpha|=(1,2(p-1))$, $|\beta|=(2,2p(p-1))$,
$|\Delta|=(0,2p(p-1)^2)$.}
\end{manualtheorem}

The ring was computed by Symonds \cite{symonds}, and the corresponding $E_2$-term by Hopkins and Miller, as recorded by Nave \cite{nave}; its history, and what is new in our route to it, are recalled in Subsection \ref{relatedwork}.

\subsection{Strategy}
\label{strategy}

The paper is an algebraic geometry paper whose theorem is about homotopy theory. It is written to be intelligible to topologists, arithmetic geometers, and algebraic geometers alike, and the author has tried to make it a pleasure for all three.

These theorems were found by looking at height $2$ in a new way, dear readers: at the elliptic curve's two-torsion not as torsion but as ramification. That shift in perspective is what let the construction climb, and with every step up, something fell apart. Each break exposed a subtlety the elliptic curve had kept invisible. The general construction is what remains once every subtlety is in view, and the way to understand it is to watch it fall apart and be rebuilt, in the order it happened.

\subsubsection{Height 2 as the Model Case: Elliptic Curves as Configuration Spaces}
\label{height2}

Elliptic curves model the Lubin--Tate action of finite subgroups at heights $1$ and $2$ \cite{forms,moore2,chromsplit}. What follows is not an analogy but a specialization of Section \ref{hp-1}, worked out in detail in \cite[\S 6.2]{ray}: the curve $X\colon y^p-y=x^{p-1}$ of the general construction has genus $\frac{(p-1)(p-2)}{2}$, so at $p=3$ it \emph{is} the supersingular elliptic curve $E\colon y^3-y=x^2$, with $\Aut(E,O)\simeq C_3\rtimes C_4$. Everything in this subsection is the $p=3$ member of the family built in the rest of the paper, and the reader should keep it as a running example.

\begin{remark}[why $p=3$]\label{rem:whyp3} The prime is not incidental. Height $2$ at $p=3$ is the first case of the dichotomy of Subsection \ref{whyfinite}, \emph{Why This Finite Subgroup}: since $p-1=2$ divides the height, the stabilizer group $\G_2$ itself contains elements of order $3$, and they sit in $G'\simeq C_3\rtimes C_4$. This is torsion in the group of symmetries, not torsion points on the curve; the element of order $3$ is the automorphism $y\mapsto y+1$ of $E$, and it is what makes $K(2)$-local homotopy at $p=3$ exotic.
\end{remark}

\begin{itemize}
\item Minimality of genus with automorphism group $G'\simeq\Z/3\rtimes\Z/4$ forces $E$ to be supersingular (formal group of height $2$), and quotienting by the two cyclic factors gives a roof of Galois covers
$$\PP^1\xleftarrow{\ C_2\ }E\xrightarrow{\ C_3\ }\PP^1.$$
The central $C_2$ is $[-1]$, and $E\to E/[-1]=\PP^1$ is the degree-two map ramified exactly at the $2$-torsion $E[2]$.
\item Here is the observation the whole paper runs on. A full level-$2$ structure on $E$ is an ordering of $E[2]$, and $E[2]$ is the ramification locus of $E\to\PP^1$: \stmt{\emph{level structure is ramification data}}. So the moduli of elliptic curves with level-$2$ structure is the configuration space of the branch points: the Legendre line $x^2=y(y-1)(y-\lambda)$, with branch divisor $\{0,1,\lambda,\infty\}$, and three actions coincide: \stmt{$\GL_2(\F_2)$ on level structures, $\Sigma_3$ on the three finite branch points, and $\Aut(E,O)/C_2$ on the representing ring of $\Def_{(E,C_2)}$}. The isomorphisms behind this coincidence are
$$\Aut(E,O)/[-1]\;=\;(C_3\rtimes C_4)/C_2\;\simeq\;C_3\rtimes C_2\;\simeq\;\Sigma_3\;\simeq\;\GL_2(\F_2),$$
and the last one is an accident of small numbers: it is what lets the branch-point action be read as a linear action on $E[2]\simeq\F_2^2$. For $p>3$ there is no group law, no torsion, and no $\GL_2$; what survives is $\Aut(X)/C_{p-1}=\F_p\rtimes\F_p^\times$, the affine group, acting on the $p$ branch points as a proper subgroup of $\Sigma_p$. The ramification locus is still there, and that is what we order.
\item Since $C_2$ is central, $\Def_{(E,C_2)}$ carries an action of the full $G'\simeq\Aut(E,O)=\Aut(E,C_2)$. By Serre--Tate \cite{st}, deformations of $E$ coincide with deformations of its formal group.
\end{itemize}

From this: the automorphisms of the elliptic curve acting on its deformation space model the Lubin--Tate action restricted to $G'$,
\begin{gather*}
\Def^\star_{(E,C_2)}\stackyq G'\ \simeq\ \Def^\star_{\mc{F}(E)}\stackyq G',\qquad\text{hence}\\[0.3\baselineskip]
H^*\big(G',\OO(\Def^\star_{(E,C_2)})\big)\simeq H^*\big(G',\pi_*(E_2)\big).
\end{gather*}

\begin{remark} The reader might expect the $C_4$-cover rather than $C_2$; but the quotient of $E$ by $C_4$ is not Galois. The choice of the \emph{central} tame subgroup is not a convenience; it is the selection principle that generalizes.
\end{remark}

\begin{remark} The picture has a global form, and it has been used before. The Legendre line is $\mc{M}(2)$, the moduli of elliptic curves with full level-$2$ structure, and $\mr{TMF}(2)$ is its spectrum of topological modular forms; Stojanoska \cite{vesna} works with $\mr{TMF}(2)$ as functions on the Legendre line carrying the $\Sigma_3$-action, recovers $\mr{TMF}$ as the $\Sigma_3$-homotopy fixed points, and uses exactly this to prove the Anderson self-duality of $\mr{Tmf}$ --- the $C_3\subseteq\Sigma_3$, our wild part, being where the $3$-torsion comes from. The global form at height $p-1$ is Subsection \ref{sec:pastandfuture}.
\end{remark}

\subsubsection{What Breaks Past Height 2, and the Twisted Love That Repairs It}
\label{perfectworld}\label{hp-1}

As we journey past the height-$2$ horizon, let us take a moment to acknowledge the shift in complexity in constructing a geometric model for the Lubin--Tate action. Serre--Tate is gone: no one-dimensional variety has formal group of height $h>2$, by Cartier duality. So the model must be a curve whose \emph{Jacobian} contains a one-dimensional height-$h$ formal group compatibly in families (the variety must not split as an algebraic group, yet \emph{must} split on the level of $p$-divisible groups), and the moduli of the curve must be equivalent to deformations of that one puny summand. It is a lot to ask of a curve. That is Parts \ref{sec:pdiv} and \ref{sec:KSpart}.

Level structure is gone too: higher-genus curves have no group structure, hence no torsion to order. Curves remain the right vehicle for one reason: they are the varieties whose moduli one can still control after giving up group structure: \emph{ramification data replaces the level structure one gave up}. That is Parts \ref{curveprops} and \ref{sec:configpart}.

At $h=p-1$ the curve delivers, and it is the one height at which the model is literally a \emph{family of curves}. There the construction from a perfect world exists, birthed by the twisted love of multiplication and addition, $\F_p\rtimes\F_p^\times$. Offered the group $G'\subset\G_h$ as the seed, we sow the model, from the curve $X$ down to the functor inducing the equivalence, following the guide of the elliptic curve.

\subsubsection*{Decorate the Center}

Why does the wild subgroup $C_p$, whose action is the entire difficulty, appear nowhere in the decoration of the curve we deform? Because it cannot: decorating $C_p$ would make its action on the deformation problem \emph{trivial} (inner automorphisms act trivially on deformations of the decorated object), the opposite of the wildly nontrivial Lubin--Tate action. The $C_p$ enters only through its shadow: it normalizes the tame decoration, hence descends to the configuration of branch points, where it acts by the cyclic shift. \emph{Decorate the center; let the wild part act through the branch divisor}: this inversion is the mechanism of the paper.

One calibration of what the model is \emph{not}. Sekiguchi--Suwa theory
\cite{ss,mrt} produces, over $\Z_p[\zeta_p]$, a group scheme interpolating
between $\mu_p$ and $\Z/p$ (the unified Kummer--Artin--Schreier--Witt
isogeny) whose extremal parameter satisfies exactly $\lambda^{p-1}=-p$.
But this is the arithmetic of a single point, not a family: its parameters
move the local wild structure vertically, while Lubin--Tate space needs
$p-2$ horizontal moduli, which come from the configuration of branch points
and from nothing else. Sekiguchi--Suwa is the \emph{period} of the story,
the identity made flesh; the configuration space is its \emph{moduli}.

\subsubsection{The Proof of Theorem \ref{equivalentexchange} : Kodaira--Spencer}
\label{sec:KSapproach}

The equivalence is a coordinate-free Kodaira--Spencer argument in three
parts; Section \ref{sec:KS} opens with a roadmap of it. (i) \emph{The
functor:} $\widehat{\mc F}^1$ takes the $\psi$-eigenpart of the
$p$-divisible group of the Jacobian, for $\psi$ the character of the
central $G$ on the extremal differential $\omega_1=dy/x$; its special
fiber is one-dimensional of height $p-1$, and $G'$ embeds in its
automorphism group. (ii) \emph{The derivative:} by Grothendieck--Messing
and Katz--Oda, the derivative of $\widehat{\mc F}^1$ at the closed point
is cup product with the Hodge line,
$H^1(T_X)^G\to\Hom(k\omega_1,H^1(\OO_X)^\psi)$, which is surjective with
kernel $H^0(\OO_D)^\psi$ for $D=\mr{div}(\omega_1)$. The whole proof is
therefore one character count on that divisor, and the count comes out
to zero because the curve is minimal: \emph{injectivity of the
Kodaira--Spencer map is the minimal-genus theorem seen from the formal
group.} (iii) Both sides are smooth of dimension $p-2$, so a tangent
isomorphism is an equivalence, and the transported action is the
Lubin--Tate action because $G'$ sits in $\G_{p-1}$ in a single conjugacy
class. An earlier version of this paper proved the theorem the other way
round, from an explicit universal curve and a recognition principle; that
proof needs an equation, this one needs none, and this one is the one
that generalizes (Subsection \ref{rem:generalizes}).

\subsection{Prior and Current Work}
\label{relatedwork}

Here is whose shoulders we stand on, and where we step off them.

\subsubsection*{As Part of a Series}
This paper is the second of a trilogy that builds geometric models of the Lubin--Tate action from decorated curves, one for each height $p^{k-1}(p-1)$ at which the Morava stabilizer group acquires new $p$-power torsion. The first, \emph{Modeling Group Actions on Stacks (Especially the Lubin--Tate Action)} \cite{ray}, sets up geometric models of group actions on deformation stacks, compares coherent cohomology of quotient stacks with continuous group cohomology, and develops the deformation theory of curves with a finite group action. We use its results throughout, citing them by its numbering (v2). The third, \emph{Only the Gauge Grows: The Lubin--Tate Action at Height $p^{k-1}(p-1)$ via Surgeries on Curves} \cite{higherlevels}, climbs to the heights $p^{k-1}(p-1)$ for $k\ge2$; its Theorem B computes the tangent of Lubin--Tate space as a $G'$-module for every $k$, and its Conjecture I is the $k\ge2$ form of Theorem \ref{equivalentexchange}.

\subsubsection*{On Whose Shoulders}
The order-$p$ subgroup of $\G_{p-1}$ was first exploited by Ravenel \cite{Rav}. He obtained it from the embedding $\Q_p(\zeta_p)\hookrightarrow D_{p-1}$ into the $p$-adic division algebra of Hasse invariant $1/(p-1)$. Using nothing about the action except the existence of an element in $\G_{h-1}$ of order $p$, he detected the odd-primary Arf invariant classes $b_i$ through
a homomorphism from the Adams--Novikov $E_2$-term to
$H^*(\Z/p;\F_{p^{p-1}})$ which, as he notes, he could not construct explicitly. The geometric model here writes that homomorphism down (Theorem \ref{theoremd}) and identifies its target: $b_i$ maps to $\beta\,\Delta^{(p^i-1)/(p-1)}$ up to a unit (Remark \ref{rem:ravenel}), and the exponent is the $a_i/p$ of his inductive differential.

A stack deforming the equation of $X$, never formally stated, is implicit in Gorbounov--Mahowald \cite{gm}. Their map to Lubin--Tate space is likewise not Serre--Tate; it is built by hand, through Honda's theory of formal
group laws, rather than from idempotents, and they use it to prove an analogue of Theorem \ref{theoremd} directly rather than to deduce it from an equivalence of moduli problems. Hill's global stack for $\mr{EO}_{p-1}$ \cite{hillphd} is built on a Hopf algebroid presented by the same equation $x^p-x=y^{p-1}$; Section \ref{sec:pastandfuture} shows that an ordered, marked version of Hill's stack has $\Def_{(X,C_{p-1})}$ as a local neighborhood.

Ravenel generalized the Gorbounov--Mahowald construction to higher heights in order to study the $\Z/p$-action there \cite{Ravenel, Ravenelii}. In \cite{Ravenel}, he showed that the Jacobian of $X_f\colon y^{p^f-1}=x^p-x$ contains a one-dimensional formal summand of height $(p-1)f$, the analogue of Theorem \ref{splittydooda}. His follow up paper \cite{Ravenelii} constructed by hand an explicit deformation of $X_f$ over a polynomial ring $L$ in $(p-1)f-1$ variables. He conjectured that the Jacobian of this deformation carries the universal one-dimensional formal group of that height; and showed how to isolate a one-dimensional formal group by lifting, through a linear approximation, from a formal group over
$L/(e_1\cdots e_{(p-1)f-1})^2$ to one over $L$. The lifting of the $\Z/p$-action was not addressed.

Finally, the ring of Theorem \ref{theoreme} has a history. Hopkins and Miller computed the $E_2$-term at height $p-1$, unpublished but recorded by Nave \cite{nave} and in Hill's thesis \cite{hillphd}, and Nave used it against the Smith--Toda complexes. Symonds computed the Tate cohomology of $G'$ with coefficients in $(E_{p-1})_*$, and from it that of the whole stabilizer group \cite{symonds}. We reach the same ring from the other side, by reading it off the configuration-space model of Theorem \ref{actup}, and with the action of the full $C_{(p-1)^2}$ on the grading derived rather than assumed (Lemma \ref{lem:gradedaction}).

\begin{remark}[Shimura varieties]\label{rem:shimura}
The stacks of topological automorphic forms \cite{taf,taf2,tafhill} are
built on PEL Shimura varieties for $U(1,h-1)$. A PEL Shimura variety $\mc{S}$ for $U(1,h-1)$ parametrizes abelian varieties with extra structure whose $p$-divisible groups are $h$ copies of one one-dimensional height-$h$ formal group, giving for each $s\in\mc{S}(k)$ a projection to that formal group. Thus, for that each supersingular point $s \in S$, we are given $\Def_s\simeq\Def_F$
\cite{carayol90,rappoportzink}. There is no clear relation between Jacobian of the universal curve over $\Def_{(X,G)}$ and such a Shimura variety at $h=p-1$: the bases have the same dimension, but the fibers have dimensions $g=(p-1)(p-2)/2$ and $h^2=(p-1)^2$. The Shimura variety naturally attached to $\Def_{(X,G)}$ would be a moduli of abelian varieties arising from monodromy datum called $\mr{Sh}(\mu_{(p-1)^2},\mf f)$ in the sense of \cite{deligne1986monodromy}; however, a construction from ramification data along the lines of \cite{li2018newton} does not yet cover this case. 
\end{remark}

\begin{remark} The global object $\mr{Curve}_G$ in which we deform $\Def_{(X, G)}$ is the smooth locus of a Hurwitz stack (Subsection \ref{sec:pastandfuture}).
\end{remark} 

\subsection{Conventions and Notation}
\label{conventions}

Fixed for the whole paper, so the reader need not reconstruct it.

\begin{itemize}
\item $p$ is an odd prime and $h=p-1$ unless stated otherwise; $k$ has
characteristic $p$ and contains $\F_{p^{p-1}}$; where a statement needs
$k=\bar k$ we say so.
\item $X$ is the minimal genus curve with $\Aut_k(X)\simeq G'$, equation
$y^p-y=x^{p-1}$, genus $\frac{(p-1)(p-2)}2$. Minimality is Theorem
\ref{tinyandunique}, not a convention.
\item $G':=\Aut_k(X)\simeq C_p\rtimes C_{(p-1)^2}$ is the \emph{large}
group, the maximal finite subgroup of $\G_h$; $G:=C_{p-1}\subseteq Z(G')$
is the \emph{small}, tame one used to decorate $X$. The prime on $G'$ is
inherited from the literature and indicates no derived construction.
\item $/$ is both a group quotient and the algebraic quotient of a scheme
by a finite group \cite[Def.\ 7.26]{ray}; $\stackyq$ is the stacky
quotient.
\item $\Def_Y^\star$: deformations of $Y$ with morphisms reducing to the
identity; $\Def_Y^G$: morphisms reducing to elements of $G$;
$\Def_Y:=\Def_Y^{\Aut(Y)}$. $(-)^\wedge$ is completion, at the identity
section for a group scheme and at $I$ when written $(-)^\wedge_I$.
\end{itemize}

Representation-theoretic notation for Theorems \ref{actup},
\ref{theoremd}, \ref{theoreme} ($\lambda$, $\overline\rho$, $N$ are as in \cite[Defs.\ 7.44, 7.46]{ray}, recalled from \cite[\S 7.11]{ray}):

\begin{itemize}
\item \emph{Group-action
clarification, used silently throughout:} $G'=C_p\rtimes_{\chic} C_{(p-1)^2}$
acts on configurations through $C_p\rtimes\Aut(C_p)$; the central kernel
$C_{p-1}$ acts trivially on the configuration and only through the grading,
on the invariant differential $u$. This is the graded content of Theorem
\ref{gradedequivalentexchange}.
\item $x_0:=[1]\in\overline\rho$ and $x_i:=\sigma^i(x_0)$ for $\sigma$ generating $C_n$; and $\delta_j:=(1-\sigma^j)(x_0)=x_0-x_j$, the \emph{displacements}, how far each beetle has wandered from home. The $\delta_j$ vanish at the base configuration, and
$I:=(p,\delta_1,\dots,\delta_{n-1})$ is the $G'$-equivariant ideal of the
characteristic locus, the equivariant replacement for
$(p,x_0,\dots,x_{n-1})$, which is not $C_p$-stable.
\item $N:=\prod_i\sigma^i(x_0)=x_0\cdots x_{n-1}$. In centered coordinates
the $x_i$ are differences of branch points along the $C_p$-orbit, so
inverting $N$ inverts the adjacent-collision locus, which by transitivity
is the whole collision locus: $N$ and the discriminant cut out the same
closed set, and $N$ is the equivariant generator of convenient degree.
$\Delta:=N^{-(p-1)^2}$ carries the inverse exponent deliberately: with
$|x_i|=-2$, this makes $|\Delta|=(0,+2p(p-1)^2)$, the periodicity class of
Theorem \ref{theoreme}; at $p=3$, $|\Delta|=24$, the modular $\Delta$,
as it must be.
\end{itemize}

\subsection{Outline}
\label{outline}

Part \ref{curveprops} is the inverse Galois theory: the curve $X$, its ramification, its field of definition, and finally its equation are consequences of $G'$, not choices (Theorem \ref{tinyandunique}). The stacky background and definitions we work with are those of the prequel \cite{ray}, cited where they are used.

Part \ref{sec:configpart} lets the curve move: its deformations become a configuration space of branch points, and the representing ring is computed with its $G'$-action (Theorems \ref{theoremb1}, \ref{theoremb2}, \ref{actup}).

Part \ref{sec:pdiv} splits the $p$-divisible group of the Jacobian and isolates one block of height $p-1$ (Theorem \ref{splittydooda}).

Part \ref{sec:KSpart} builds $\widehat{\mc F}^1$ and proves Theorem \ref{equivalentexchange}, ungraded and graded.

Part \ref{sec:cohomology} puts the model to work, computing $\pi_*E_{p-1}$ as a $G'$-representation (Theorem \ref{theoremd}) and its Tate cohomology (Theorem \ref{theoreme}).

\section*{Acknowledgements}

This paper is the second of three papers on expeditions during my PhD thesis, and I would like to thank Paul Goerss for advising me during graduate school and allowing me the great honour of being his last student. During that time I was supported as a fellow by the NSF GRFP under Grant Number DGE 1842165. As a postdoc, where I finished this paper, I was funded by the Deutsche Forschungsgemeinschaft (DFG, German Research Foundation) – Project-ID 427320536–SFB1442, as well as under Germany’s Excellence Strategy EXC 2044 390685587, Mathematics M\"unster: Dynamics-Geometry-Structure.

I would like to thank Sasha Shmakov, Sasha Petrov, Luc Illusie, Noah Riggenbach, Louzi Shi, Emanuel Reinecke, Barry Mazur, Maxime Ramzi, Tom Perutka, Anish Chedalavada, Sofia Marlasca Aparicio, Andrea Bianchi, Aristides Kontogeorgis, Rachel Preis, Julia Hartmann, and David Harbater for helpful conversations. I'd like to thank Jonathan Lubin for teaching me about formal groups when I was still an engineer. I'd like to thank an anonymous reviewer for suggesting to add detail to comparison between prior and current work. I'd like to thank Thomas Nikolaus, Christopher Deninger, Mike Hill, and Andres Mejia for encouraging and pushing me to overcome my perfectionism and finally share this story with the world.

\part{Uniqueness in Inverse Galois Theory, Canonical Designer Curves with Decoration}
\label{curveprops}

\section*{Overview}
Throughout, a curve over a field $k$ means a smooth projective geometrically connected curve over $k$. Given a finite group $G$ and a field $k$, inverse Galois theory asks, and sometimes answers:

\begin{question} Is there a curve $X$ over $k$ with a faithful action of $G$ by $k$-automorphisms? Equivalently, is there a Galois extension of function fields $k(X)/k(X)^G$ with group $G$ in which $k$ is algebraically closed? \end{question}

\begin{answer} If $k$ is algebraically closed, yes, for every $G$ and in every characteristic, and one may even take $X/G\simeq\PP^1_k$. In characteristic $0$ this is Riemann's existence theorem: $G$ is a quotient of $\pi_1(\PP^1\smallsetminus\{n\text{ points}\})$ once $n-1$ is at least the number of generators of $G$. In characteristic $p$ it follows from Abhyankar's conjecture, proved by Raynaud and Harbater \cite{raynaudabh,harbaterabh}: $G$ is a quotient of $\pi_1^{\mathrm{\acute et}}(\PP^1_k\smallsetminus\{n\text{ points}\})$ as soon as $G/p(G)$ is generated by $n-1$ elements, where $p(G)$ is the subgroup generated by the elements of $p$-power order. Over $k=\Q$ with $X/G\simeq\PP^1_\Q$ demanded, this is the regular inverse Galois problem, which is open. \end{answer}

Inverse Galois theory rarely asks whether the realization is unique, and in general it is not: if $X$ carries a faithful $G$-action, so does every $G$-equivariant cover of it, at larger genus. The question with teeth is at minimal genus, and we ask it for the one group this paper is about, $G' := C_p \rtimes_{\chic} C_{(p-1)^2}$ with $\chic$ surjective onto $\F_p^\times$, the maximal finite subgroup of the Morava stabilizer group at height $p-1$.

\begin{question} Let $k$ be a perfect field of characteristic $p$ containing $\F_{p^{p-1}}$. What is the minimal genus of a curve over $k$ with a faithful $G'$-action, and is the curve attaining it unique up to $k$-isomorphism? \end{question}

\begin{answer} The minimal genus is $\frac{(p-1)(p-2)}{2}$, and it is attained exactly by the $k$-forms of $y^p-y=z^{p-1}$. These forms are classified by $k^\times/(k^\times)^{p-1}$, so the curve is unique as soon as every element of $k$ has a $(p-1)$-st root, in particular for $k=\bar k$ (Theorem \ref{thm:minimal}, Lemma \ref{lem:descent}). \end{answer}

\noindent And the mechanism is pretty: $G'$ carries a character which must be surjective, the ramification break is the exponent through which that character factors, and there is exactly one break small enough to do the job.
Minimality of genus does no work beyond pinning down the shape of the cover; after that, the arithmetic of $(p-1)^2$ takes over and writes the equation for us. This is Theorem \ref{thm:minimal}: we did not choose $y^p-y=x^{p-1}$. It was ordained by the group.

This Part runs as follows. Section \ref{sec:uniqueness} opens with the language of covers, algebraic quotients and ramification divisors, for the ease of interested topologists, and then proves uniqueness. Its engine is Lemma \ref{lem:char} (two characters, one break), the curve-side avatar of $\pi^{p-1}=-p$. Section \ref{sec:roof} then computes the ramification of the roof of Galois covers $\PP^1\xleftarrow{C_{p-1}}X\xrightarrow{C_p}\PP^1$ that the curve carries (Lemma \ref{rammymammy}), and Section \ref{symmetricsultan} shows how $G'/G=:\Aut^f(X)$ acts on the branch points of $f\colon X\to X/G$, giving
$$\Aut^f(X)\hookrightarrow\Sigma_p.$$
That embedding is the decoration we deform in Part \ref{sec:configpart}: the wild group $C_p$ is never decorated (doing so would trivialize its action) and enters only through its permutation of the branch points of the tame
quotient.

\begin{conv}\label{conv:curveprops}
Throughout this Part, $p$ is an odd prime, and
$$G'\simeq C_p\rtimes_{\chic} C_{(p-1)^2},\qquad
\chic\colon C_{(p-1)^2}\twoheadrightarrow\F_p^\times,$$
is the maximal finite subgroup of $\G_{p-1}$ with nontrivial $p$-torsion
\cite[Lemma 0.1]{ray}. For a curve $X$ with a faithful $G'$-action we write $P$
for the Sylow $p$-subgroup, $H\simeq C_{(p-1)^2}$ for a complement
(Schur--Zassenhaus, the orders being coprime), $Y:=X/P$ of genus $g_Y$,
$B\subset Y$ the branch locus of $q\colon X\to Y$ with $|B|=r$ and breaks
$m_b$, and
$$T\ :=\ \frac{(p-1)(p-2)}{2}$$
for the target genus.
\end{conv}

\section{Theorem A, Tiny and Unique: the Curve from the Group}
\label{sec:uniqueness}\label{sec:cycliccovers}
For background on the language of covers and ramification, see \cite[\S 7.9]{ray}.

We prove that $y^p-y=z^{p-1}$ is the unique curve of minimal genus admitting a
faithful action of $G'=\Z/p\rtimes_{\chic}\Z/(p-1)^2$ with $\chic$ surjective,
uniformly for all odd $p$ including $p=3$. No hypothesis on the ramification
or on the quotient $X/P$ is imposed: the break $m=p-1$, the rationality of
$X/P$, the vanishing of the $p$-rank, and the minimal field of definition
$\F_{p^{p-1}}$ are consequences of the group.

\begin{manualtheorem}{A}\label{tinyandunique}\label{thm:minimal}\stmt{
Let $p$ be an odd prime and $k$ a perfect field containing $\F_{p^{p-1}}$.
Among smooth projective geometrically connected curves over $k$ admitting a
faithful $G'$-action, the minimal genus is
$$g\ =\ T\ =\ \frac{(p-1)(p-2)}{2},$$
attained precisely by the $k$-forms of
$$X\colon\quad y^p-y\ =\ z^{p-1}.$$
On the minimal stratum the following are consequences, not hypotheses:
\begin{enumerate}
\item $X/P\cong\PP^1_k$;
\item $X\to X/P$ is branched at exactly one point, with break $m=p-1$;
\item $X$ has $p$-rank $0$;
\item for $p\ge5$, $\Aut(X)=G'$; for $p=3$, $\Aut(X,O)=G'$ for the origin
$O$ of the elliptic curve $X$.
\end{enumerate}
If moreover $(k^\times)^{p-1}=k^\times$, in particular if $k=\bar k$, the curve is unique up to $k$-isomorphism.}
\end{manualtheorem}

\begin{remark}[$p=3$]\label{rem:p3} We use the phrasing ``faithful action" to speak of the $p=3$ and $p \geq 5$ case in the same breath. Since $\Aut(E)$ is infinite for an elliptic curve, we consider instead $\Aut(E, O)$ in this case (When $T=1$ and $X=E\colon y^3-y=x^2$ is the supersingular elliptic curve $j=0$, with  of order $12$). The phrasing ``faithful action'' is what makes the statement uniform, since $\Aut(E)$ itself contains the translations. The $p$-rank computation independently confirms uniqueness: characteristic $3$
has one supersingular $j$-invariant.
\end{remark}

\begin{remark}[minimality is essential]
Without it the group determines nothing: for $p=5$ the curves
$y^5-y=x^{28}+cx^{12}$, $c\ne0$, carry $\Z/5\rtimes\Z/16$ with $\chic$
surjective, have genus $54$, and move in a family. The eigen-equation
\eqref{eq:eigen} below forces only $\operatorname{supp}(f)\subset m+(p-1)^2\Z$,
a single monomial exactly when $m<(p-1)^2$. The \emph{square} matters for the
same reason: with a complement of order $p-1$ the minimal stratum is a
$(p-3)$-dimensional family, not a point.
\end{remark}

\begin{remark}[minimal field of definition]\label{rem:field}
Since $p^n\equiv1+n(p-1)\pmod{(p-1)^2}$, one has $\ord_{(p-1)^2}(p)=p-1$: the
smallest field containing $\mu_{(p-1)^2}$ is $\F_{p^{p-1}}$, and by Lemma
\ref{lem:char} the tangent character embeds $\Z/(p-1)^2$ into $k^\times$. The
base field of this paper is therefore forced by $G'$; it is, coherently, the
residue field of the maximal unramified subfield of the division algebra of
invariant $1/(p-1)$.
\end{remark}

\subsection{Two Characters, One Break}

Two characters are in play, for unrelated reasons. The
{\color{Bittersweet}{conjugation character}} $\chic\colon H\to\F_p^\times$,
$c\sigma c^{-1}=\sigma^{\chic(c)}$, is pure group theory, canonical because
$\Aut(\Z/p)=\F_p^\times$ independently of a generator. The
{\color{Bittersweet}{tangent character}} $\theta\colon H\to k^\times$, the
scalar by which $H$ acts on a cotangent line $\mm_Q/\mm_Q^2$, is pure linear
algebra. The break is the bridge, and the shape of the bridge is what lets a
group-theoretic hypothesis pin a ramification invariant: the tame part is
fully visible on the tangent line, while the wild part is invisible there
($\theta(\sigma)=1$ always) and shows its first datum only in degree $m+1$.
Invariantly, $H$ acts on
$R_m:=\Hom(\mm_Q/\mm_Q^2,\ \mm_Q^{m+1}/\mm_Q^{m+2})$ by $\theta^m$, and
$\sigma^j\mapsto[\,t\mapsto(\sigma^j)^*t-t\,]$ is an injective
$H$-equivariant map $P\hookrightarrow R_m$.

\begin{lemma}[two characters, one break]\label{lem:char}\label{2c1b}
\stmt{Let $H\simeq C_{(p-1)^2}$ be the tame complement of Convention \ref{conv:curveprops} (not its central subgroup $C_{p-1}$, on which $\chic$ is trivial), $b\in B$ with $Q\in X$ above it, and $H_b\subseteq H$ the stabilizer. Then $H_b$ acts on the tangent line at $Q$ by an injective
character $\theta$, and
$$\chic\big|_{H_b}=\theta^{-m_b},\qquad\text{hence}\qquad
\big|\chic(H_b)\big|=\frac{|H_b|}{\gcd(|H_b|,m_b)}.$$}
\end{lemma}
\begin{proof}
$P\trianglelefteq G'$ forces every element of $G'$ to preserve the fiber over
$b$, so $H_b$ acts on $\mm_Q/\mm_Q^2$. Since $p\nmid|H_b|$ we may average: this
produces a uniformizer $t$ with $c^*t=\theta(c)\,t$ exactly. If $\theta(c)=1$
then $c$ fixes a topological generator of $\widehat\OO_{X,Q}$ and is trivial,
so $\theta$ is injective. Now the wild part shows its hand: $\sigma^*t=t+a\,t^{m_b+1}+\cdots$ with $a\ne0$,
and $(\sigma^j)^*t=t+ja\,t^{m_b+1}+\cdots$: the coefficient of $t^{m_b+1}$
reads off the exponent. Conjugating,
$$(c\sigma c^{-1})^*t\ =\ t+a\,\theta(c)^{-m_b}\,t^{m_b+1}+\cdots,$$
which must be $\sigma^{\chic(c)}$. Two characters, one coefficient, no choice. The index formula follows since $\theta(H_b)$
is cyclic.
\end{proof}

This lemma is the curve-side avatar of $\pi^{p-1}=-p$: at $m=p-1$ it reads
$\chic=\theta^{-(p-1)}$, the mod-$p$ shadow of the identity in $\Z_p[\zeta_p]$,
with tangent character $\leftrightarrow$ Lubin--Tate character, conjugation
character $\leftrightarrow$ cyclotomic character, and break
$\leftrightarrow$ $e(\Q_p(\zeta_p)/\Q_p)$. We return to this in Remark
\ref{rem:shadow} once the break has been forced.

For the proof of uniqueness, we will need a size estimate on tame cyclic automorphism groups. Every fairy tale has a giant who is too big for the house he wants to live in. Here the giant is a cyclic group of order $(p-1)^2$, the houses are curves, and the estimate says which houses are too small to let him in.

\begin{lemma}[Wiman--Harvey]\label{lem:wiman}
Let $H$ be cyclic of order $N$ acting faithfully on a smooth projective curve
of genus $g_Y\ge1$ over $k=\bar k$, with $p\nmid N$. Then $N\le4g_Y+2$.
\end{lemma}
\begin{proof}[Reference] The proof of Wiman's bound \cite{wiman,harvey} minimizes the orbifold Euler characteristic $\mathcal E=2g_{Y/H}-2+\sum_j(1-1/n_j)$ subject to $2g_Y-2=N\mathcal E$ and uses only \emph{tame} Riemann--Hurwitz, so it holds verbatim in characteristic $p$ when $p\nmid N$, with no lifting argument.
\end{proof}

\subsection{Proof of Uniqueness}

The proof has three steps, each with its own intuition, so we separate them:
(A) minimality forces $g_Y=0$; (B) it forces a single branch point with break
$p-1$; (C) the eigen-equation then writes the equation. Steps A and B are mostly counting. Counting is the part of the evening nobody photographs --- getting dressed, finding the keys --- but there is no party without it, and the party is Step C, where the equation shows up wearing something wonderful. Let's go.

\begin{lemma}[Step A]\label{prop:step1}
If $g_Y\ge1$ then $g_X>T$.
\end{lemma}

The idea: $H$ is a \emph{large} cyclic group, of order $\sim p^2$. A curve of
positive genus has no room for it, so $g_Y$ must itself be of order $p^2$,
and Riemann--Hurwitz \cite[Lemma 7.29]{ray} multiplies that expense by $p$, overshooting $T\sim p^2/2$
by a factor of $p$. There is exactly one configuration in which the
amplification fails ($g_Y=1$ with $q$ unramified), and it is excluded not by a genus count but by the group law of an elliptic curve. Three cases, two of them bookkeeping.

\begin{proof}
\emph{Case $g_Y\ge2$.} First the quotients that are too generous with their genus. By Lemma \ref{lem:wiman}, $g_Y\ge((p-1)^2-2)/4$.
Discarding the non-negative ramification sum in Riemann--Hurwitz \cite[Lemma 7.29]{ray},
$$2g-2\ \ge\ 2p(g_Y-1)\ \ge\ \tfrac p2\big((p-1)^2-6\big),$$
and $\tfrac p2((p-1)^2-6)>p(p-3)\iff p^2-4p+1>0$, true for $p\ge4$. For $p=3$
directly: $2g-2\ge2p(g_Y-1)\ge6>0=2T-2$. Too expensive.

\emph{Case $g_Y=1$, $r\ge1$.} Next the tori with a little ramification, dressed for a party they cannot afford. Every point stabilizer in $H$ has order at most
$6$ (Lemma \ref{lem:wiman} at $g_Y=1$), so every $H$-orbit in $Y$ has at least
$(p-1)^2/6$ points; $B$ is $H$-stable and nonempty, so $r\ge(p-1)^2/6$ and
$$2g-2\ \ge\ -2p+2(p-1)\cdot\tfrac{(p-1)^2}{6}\ =\ -2p+\tfrac{(p-1)^3}{3},$$
which exceeds $p(p-3)$ for $p\ge5$ since $(p-1)^3\ge3p(p-3)+6p$ there. For
$p=3$, Riemann--Hurwitz \cite[Lemma 7.29]{ray} with $r\ge1$ gives $2g-2\ge0+2\cdot3=6>0$. Also too expensive.

\emph{Case $g_Y=1$, $r=0$.} This is the one case where counting fails us, so we have to look. Riemann--Hurwitz \cite[Lemma 7.29]{ray} gives $g-1=p(g_Y-1)=0$, so
$X\cong Y$ is an elliptic curve and $P$ acts freely, by translation $t_a$ by a
point $a$ of exact order $p$; in particular $X$ is ordinary. $H$ has a fixed
point (the ramification points of $X\to X/H$, nonempty since $g_{X/H}=0$), so
$H$ injects into $\Aut(X,O)$ for a suitable origin. For
$\varphi\in\Aut(X,O)$, $\varphi t_a\varphi^{-1}=t_{\varphi(a)}$, so $\chic$
factors through the action on $X[p]\cong\Z/p$ and
$|\chic(H)|\le|\Aut(X,O)|$. For $p\ge5$: $|H|=(p-1)^2\le|\Aut(X,O)|\le6$ is
impossible. For $p=3$: we need $H\cong\Z/4$ inside $\Aut(X,O)$, but $X$ is
ordinary, and in characteristic $3$ ordinary means $j\ne0$, where
$\Aut(X,O)=\{\pm1\}$. Contradiction in both cases: an elliptic curve simply has no room for a group of order $(p-1)^2$ twisting its translations.
\end{proof}

\begin{remark}
This is the only lemma of the section whose proof differs at $p=3$,
and not by accident: at $p=3$ the target genus \emph{is} $1$, so genus-one
competitors cannot be excluded by a genus count and must be excluded by the arithmetic of elliptic curves: the supersingularity dichotomy. Small primes are like small towns: everyone knows everyone, and the rules that hold everywhere else come with local exceptions. For $p\ge5$ the argument is purely numerical.
\end{remark}

\begin{lemma}[Step B]\label{prop:step2}
The minimum of $g$ is $T$, attained exactly when $r=1$ and $m=p-1$.
\end{lemma}
\begin{proof}
By Lemma \ref{prop:step1}, $Y\cong\PP^1$, and $H=\mu_{(p-1)^2}$ acts as
$z\mapsto\alpha z$: two fixed points, every other orbit free of size
$(p-1)^2$. So the branch locus $B$ has very few places it is allowed to sit, and we visit each, behind four doors.
\begin{itemize}
\item \emph{$r=0$.} Impossible, since $\pi_1^{\text{\'et}}(\PP^1)=1$: the line has no unramified covers to offer.
\item \emph{$r\ge(p-1)^2$, that is, $B$ meets a free orbit.} Riemann--Hurwitz \cite[Lemma 7.29]{ray} gives $2g-2\ge-2p+2(p-1)^3>p(p-3)$ for all $p\ge3$ (at $p=3$: $16-6=10>0$). Too expensive.
\item \emph{$r=1$.} And, as in every good story, the third door is the right one. The stabilizer of the branch point is all of $H$ and $\chic$ is surjective, so Lemma \ref{lem:char} gives
\[p-1\ =\ \frac{(p-1)^2}{\gcd((p-1)^2,m)}\iff\gcd\big((p-1)^2,m\big)=p-1.\]
Writing $m=(p-1)\kappa$ this says $\gcd(p-1,\kappa)=1$; with $p\nmid m$ the least admissible value is $\kappa=1$, and the break is forced:
\[\boxed{\,m=p-1\,}\]
and then $2g-2=-2p+p(p-1)=p(p-3)$, so $g=T$.
\item \emph{$r=2$, that is, $B$ is the pair of fixed points.} Both stabilizers equal $H$, so each break is $\ge p-1$ by the previous case, and $2g-2\ge-2p+2p(p-1)=2p(p-2)>p(p-3)$. Twice the ramification, twice the genus; one branch point it is.
\end{itemize}
\end{proof}

\begin{lemma}[Step C]\label{prop:step3}
On the minimal stratum, $X\cong\{y^p-y=z^{p-1}\}$.
\end{lemma}
\begin{proof}
The equation has been waiting offstage this whole time; now it walks into the light. Since $Y\cong\PP^1$, Artin--Schreier theory gives $k(X)=k(z)(y)$ with
$y^p-y=f\in k(z)$, and $f$ may be taken \emph{reduced}: if $v_{P_0}(f)=-pn$
with uniformizer $u$ and leading coefficient $c$, subtracting
$\wp(c^{1/p}u^{-n})$ lowers the pole order ($k$ is perfect). For reduced $f$
the branch points are the poles and the break at a pole is its order. By
Lemma \ref{prop:step2} there is one pole, of order $p-1$; move it to
$z=\infty$, so $f\in k[z]$ has degree $p-1$. Reducedness forbids only the
constant term, which lies in $\wp(k)=k$ and may be deleted. Thus
$f=\sum_{i=1}^{p-1}a_iz^i$ with $a_{p-1}\ne0$. That is the whole wild side of the story: a polynomial of degree $p-1$ and no constant term. Now the tame side, which will kill every coefficient but one.

Since $P$ is normal, $H$ acts on $k(X)^P=k(z)$, faithfully because
$H\cap P=1$, fixing $z=\infty$; so $H\subseteq\Ga\rtimes\Gm$. The projection
$H\to\Gm$ is injective, its kernel lying in the $p$-torsion group $\Ga$ while
$p\nmid|H|$. A generator $z\mapsto\alpha_0z+\beta_0$ has $\alpha_0\ne1$, hence
a fixed point $z_0=\beta_0/(1-\alpha_0)$; translating $z\mapsto z-z_0$ (which
fixes $\infty$ and preserves the shape of $f$, deleting the new constant
term) we obtain $H=\{z\mapsto\alpha z:\alpha\in\mu_{(p-1)^2}\}$.

Let $c\in H$ with $c^*z=\alpha z$. Then $\wp(c^*y)=f(\alpha z)$ and
$k(z)(c^*y)=k(X)$, so by the Artin--Schreier correspondence the classes of
$f(\alpha z)$ and $f$ generate the same subgroup of $k(z)/\wp k(z)$: there are
$\lambda\in\F_p^\times$ and $h\in k(z)$ with
\begin{equation}\label{eq:eigen}
f(\alpha z)\ =\ \lambda\,f(z)+\wp(h).
\end{equation}
A pole of $h$ of order $s\ge1$ would give $\wp(h)$ a pole of order $ps\ge p$;
but $\wp(h)$ is a polynomial of degree $\le p-1$ by \eqref{eq:eigen}, so
$h\in k$ and $\wp(h)$ is constant. Comparing coefficients of $z^i$, $i\ge1$:
$a_i\alpha^i=\lambda a_i$. Taking $i=p-1$ gives $\lambda=\alpha^{p-1}$. If
$a_i\ne0$ for some $1\le i\le p-2$ then $\alpha^{i-(p-1)}=1$ for
\emph{every} $\alpha\in\mu_{(p-1)^2}$, so $(p-1)^2\mid(p-1)-i$ with
$0<(p-1)-i\le p-2<(p-1)^2$: contradiction. Hence $f=a_{p-1}z^{p-1}$, and
rescaling $z$ by a $(p-1)$-st root of $a_{p-1}$ gives $f=z^{p-1}$. The group has written the equation; we only held the pen, and it is a very good pen.
\end{proof}

\begin{lemma}[descent to $k$]\label{lem:descent}
Let $k \supseteq \F_{p^{p-1}}$ be perfect and let $X/k$ lie on the minimal
stratum. Then $X/P \simeq \PP^1_k$ over $k$, and $X$ is $G'$-equivariantly
$k$-isomorphic to
$$X_a\colon\quad y^p - y \;=\; a\,z^{p-1}, \qquad a \in k^\times,$$
with $\sigma\colon y \mapsto y+1$ and
$c_\alpha\colon (z,y) \mapsto (\alpha z,\ \alpha^{p-1}y)$ for
$\alpha \in \mu_{(p-1)^2}$. Two of these, $X_a$ and $X_{a'}$, are
$G'$-equivariantly $k$-isomorphic if and only if
$a' \equiv a \bmod (k^\times)^{p-1}$. So the $k$-forms of $(X,G')$ are a
torsor under $k^\times/(k^\times)^{p-1} = H^1(k,\mu_{p-1})$, and $X$ is unique
up to $k$-isomorphism as soon as $(k^\times)^{p-1} = k^\times$, in particular
for $k = \bar k$. No Artin--Schreier invariant appears: the tame action forces
the constant term of the equation into $\wp(k)$.
\end{lemma}
\begin{proof}
In the old stories, two sisters can be told apart in their own village and are mistaken for each other everywhere else. Forms are curves like that: distinct over $k$, the same curve over $\bar k$. The branch point of $X \to X/P$ is unique (Lemma \ref{prop:step2}), hence
Galois-fixed, hence $k$-rational, and a genus-zero curve with a $k$-point is
$\PP^1_k$. Now run the proof of Lemma \ref{prop:step3} over $k$ and watch where it needs $\bar k$. Every step is
Galois-equivariant and uses only $\mu_{(p-1)^2}\subseteq k$, except the two
normalizations at the end; without them it gives $y^p-y=f(z)$ with
$f=a\,z^{p-1}+b$, $a\in k^\times$, $b\in k$, and $H=\mu_{(p-1)^2}$ acting on
$z$ by scaling after a $k$-rational translation (the fixed point
$z_0=\beta_0/(1-\alpha_0)$ lies in $k$).

\emph{The constant term dies.} Let $c\in H$ act by $c^*z=\alpha z$ with
$\alpha$ a generator of $\mu_{(p-1)^2}$, and write $c^*y=\lambda y+h$ with
$\lambda\in\F_p^\times$, $h\in k(z)$, as in \eqref{eq:eigen}. Comparing
$\wp(c^*y)=f(\alpha z)$ with $\lambda\,\wp(y)+\wp(h)=\lambda f(z)+\wp(h)$,
$$(\alpha^{p-1}-\lambda)\,a\,z^{p-1}\;+\;(1-\lambda)\,b\;=\;\wp(h).$$
The left side is a polynomial of degree $\le p-1$, so $h$ has no poles and
lies in $k$; then $\lambda=\alpha^{p-1}$ and $(1-\lambda)b=\wp(h)\in\wp(k)$.
Since $\alpha$ generates $\mu_{(p-1)^2}$, $\lambda=\alpha^{p-1}$ generates
$\F_p^\times$, so $1-\lambda\in\F_p^\times$ is a unit; as $\wp$ is
$\F_p$-linear, $b\in\wp(k)$. Replacing $y$ by $y-h'$ with $\wp(h')=b$,
$h'\in k$, deletes it. The Artin--Schreier invariant never had a chance. This is $X_a$, and the displayed $c_\alpha$ is the lift
of $z\mapsto\alpha z$ with $c_\alpha^*y=\alpha^{p-1}y$: indeed
$\wp(\alpha^{p-1}y)=\alpha^{p-1}(y^p-y)=a(\alpha z)^{p-1}$.

\emph{Isomorphisms.} Let $\varphi\colon X_a\to X_{a'}$ be a $G'$-equivariant
$k$-isomorphism. It commutes with $P$, so it descends to
$\bar\varphi\in\Aut(\PP^1_k)$ commuting with every scaling $z\mapsto\alpha z$,
$\alpha\in\mu_{(p-1)^2}$. As $\mu_{(p-1)^2}$ contains an element of order
$\ge3$, the centralizer in $\mr{PGL}_2(k)$ is the diagonal torus, so
$\bar\varphi\colon z\mapsto\gamma z$ with $\gamma\in k^\times$. Since
$\varphi$ commutes with $\sigma$, the function $\varphi^*y'-y$ is
$\sigma$-invariant, hence lies in $k(z)$: write $\varphi^*y'=y+g$. Then
$$a\,z^{p-1}+\wp(g)\;=\;\wp(\varphi^*y')\;=\;\varphi^*\big(a'z'^{\,p-1}\big)
\;=\;a'\gamma^{p-1}z^{p-1},$$
so $\wp(g)=(a'\gamma^{p-1}-a)z^{p-1}$ forces $g\in k$ as before, then
$\wp(g)=0$ and $a=a'\gamma^{p-1}$: the classes of $a$ and $a'$ in
$k^\times/(k^\times)^{p-1}$ agree. Conversely, if $a=a'\gamma^{p-1}$ then
$(z,y)\mapsto(\gamma z,y)$ is such an isomorphism, commuting with $\sigma$
and with each $c_\alpha$ because scalings commute.

This agrees with Galois cohomology. The $k$-forms of the pair $(X_1,G')$ are
classified by $H^1(k,\Aut(X_1,G'))$, where $\Aut(X_1,G')$ is the centralizer
of $G'$ in $\Aut(X_1)$ (which is $G'$ for $p\ge5$ by Theorem
\ref{thm:minimal}(4), and $\Aut(X_1,O)=G'$ at $p=3$), namely
$Z(G')=\ker\chic\cong\mu_{p-1}$, acting by $z\mapsto\zeta z$. Since
$\mu_{p-1}=\F_p^\times\subseteq k$, Kummer theory gives
$H^1(k,\mu_{p-1})=k^\times/(k^\times)^{p-1}$, the invariant $a$. The wild
part $P$ is not central, which is why $k/\wp(k)$ does not appear.
\end{proof}

Now that we have all of these lemmas, let us reap the benefits.

\begin{proof}[Proof of Theorem \ref{thm:minimal}] Lemma \ref{prop:step1} is (1), Lemma \ref{prop:step2} is (2) and the value of the minimal genus, and Lemma \ref{prop:step3} is the equation, with uniqueness over $k$ under the stated hypotheses by Lemma \ref{lem:descent}. For (3), Deuring--Shafarevich \cite{subrao} for the $\Z/p$-cover totally ramified over one point gives $\gamma_X-1=p(0-1)+(p-1)=-1$, so the $p$-rank is $0$. For (4): if $p\ge5$, $\Aut(X)=G'$ by Valentini--Madan \cite{valmadan} (the big-action regime is $m\equiv1\bmod p$; here $m\equiv-1$); if $p=3$ the curve is the elliptic curve $y^3-y=z^2$, supersingular by (3), and characteristic $3$ has a unique supersingular $j$-invariant, $j=0$, with $\Aut(X,O)$ the dicyclic group of order $12$, which is $\Z/3\rtimes\Z/4$.
\end{proof}

\begin{remark}[more general $k$]
Steps A and B need only $\mu_{(p-1)^2}\subset k^\times$; geometric
connectedness excludes the constant-field-extension covers of $\PP^1_k$, and
the unique branch point is Galois-stable, hence rational, so ``genus $0$''
does mean $\PP^1_k$. Step C uses $\bar k$ only once, to normalize $a_{p-1}$:
the constant term is deleted by the tame action, not by the field (Lemma
\ref{lem:descent}). Over a general perfect $k\supseteq\F_{p^{p-1}}$ the
minimal stratum is therefore the set of forms
$$y^p-y=a\,z^{p-1},\qquad a\in k^\times/(k^\times)^{p-1},$$
all isomorphic over an extension of degree dividing $p-1$.
\end{remark}

\begin{remark}[the arithmetic shadow: the many faces of $p-1$]\label{rem:shadow}
The break is now forced, and Lemma \ref{lem:char} at $m=p-1$ reads
$\chic=\theta^{-(p-1)}$: the characteristic-$p$ avatar of $\pi^{p-1}=-p$ in
$\Z_p[\zeta_p]$, with $p-1$ the ramification index of the level-one
Lubin--Tate extension of $\Q_p$. The same $p-1$ appears as the break of $q$,
the exponent of the semidirect product, the degree of the field of definition
(Remark \ref{rem:field}, the same statement counted multiplicatively), and,
in the intended application, the height --- for one reason. Geometrically: an
Artin--Schreier break of $p-1$ is what one sees in the special fiber when a
$\mu_p$-torsor of $p$ branch points collides along $\pi^{p-1}=-p$; this is the
last row of the table in Subsection \ref{oneidentity}, and Part
\ref{sec:configpart} is where the collision is deformed.
\end{remark}

\section{The Roof and its Ramification}
\label{sec:roof}

The group handed us a curve; the curve hands us a roof. Theorem \ref{thm:minimal} comes with two Galois covers of the line,
$$\PP^1\xleftarrow{\ C_{p-1}\ }X\xrightarrow{\ C_p\ }\PP^1,$$
and we now write its automorphisms and ramification explicitly. Theorem
\ref{thm:minimal}(4) identifies the automorphism group abstractly; the
following lemma exhibits the generators we compute with for the rest of the
paper.

\begin{lemma}\label{autgroup}\stmt{
The curve $X\colon y^p-y=x^{p-1}$ over $\F_{p^{p-1}}$ has automorphism group
$G'=\Z/p\rtimes\Z/(p-1)^2$ generated by
$$\sigma\colon(x,y)\mapsto(x,y+1),\qquad
\tau\colon(x,y)\mapsto(\zeta^px,\ \zeta^{p-1}y),$$
with $\zeta$ a primitive $(p-1)^2$-th root of unity.}
\end{lemma}
\begin{proof}
$\sigma$ is an automorphism because $y^p-y=\prod_{i\in\F_p}(y-i)$ is
translation-invariant in characteristic $p$; the element
$\eta=\zeta^{p-1}$ of order $p-1$ acts by $x\mapsto\eta x$ fixing $y$, so $X$
is both an Artin--Schreier and a Kummer cover and $\Aut(X)\supseteq
C_p\times C_{p-1}$. That is the boring part. The interesting part is $\tau$: applying it to the
equation, $\tau$ is an automorphism iff $\zeta^{p(p-1)}=\zeta^{p-1}$, i.e.\
iff $p(p-1)\equiv p-1\pmod{(p-1)^2}$, which holds since
$p(p-1)-(p-1)=(p-1)^2$. Its interaction with both $x$ and $y$ is what makes
$\Aut(X)$ a semidirect product rather than a product:
$\tau\sigma\tau^{-1}=\sigma^{\zeta^{p-1}}$, i.e.\ $\chic(\tau)=\eta$, which is
$\theta^{-(p-1)}$ for the tangent character $\theta(\tau)=\zeta^{-p}$: Lemma \ref{lem:char} in coordinates. That these generate all of $\Aut(X)$ is
Theorem \ref{thm:minimal}(4).
\end{proof}

\begin{lemma}\label{rammymammy}\stmt{
Let $X\colon y^p-y=x^{p-1}$ over $k=\F_{p^{p-1}}$, with projections
$p_x,p_y\colon X\to\PP^1$.
\begin{itemize}
\item $p_x$ is $\Z/p$-Galois and totally wildly ramified at the single point
$[1:0:0]=\infty$, with break $p-1$:\ \ $\mr{Ram}(p_x)=p\,[\infty]$.
\item $p_y$ is $\Z/(p-1)$-Galois and tamely totally ramified over
$\F_p\cup\{\infty\}\subset\PP^1$:
$$\mr{Ram}(p_y)=\sum_{i\in\F_p}(p-2)\,[0:i:1]+(p-2)\,[1:0:0]
=\sum_{P\in\F_p\cup\{\infty\}}(p-2)\,[P].$$
\end{itemize}}
\end{lemma}
\begin{figure}[h]\begin{center}\includegraphics[width=8cm]{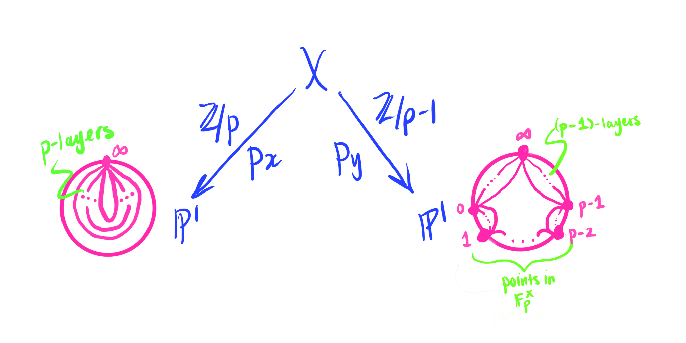}\end{center}\caption{The roof of Galois covers $\PP^1 \xleftarrow{\ \Z/p\ } X \xrightarrow{\ \Z/(p-1)\ } \PP^1$. The wild cover $p_x$ is totally ramified over one point, with $p$ sheets meeting there like the layers of an onion; the tame cover $p_y$ is totally ramified over the $p+1$ points $\F_p \cup \{\infty\}$, with $p-1$ layers at each.}\end{figure}
\begin{proof}
$C_p$ acts by $y\mapsto y+1$ fixing $x$, so $p_x$ is a $C_p$-cover. It is the
projection from $[0:1:0]\notin X$, hence of degree $p$ by \cite[Lemma 7.30]{ray}; the point at infinity $[1:0:0]$ of $X\subset\PP^2$ is
fixed by $C_p$, and it is the unique branch point with break $p-1$ by Theorem
\ref{thm:minimal}(2). $C_{p-1}$ acts by $x\mapsto\eta x$ fixing $y$, so $p_y$
is a $C_{p-1}$-cover; it is the projection from $[1:0:0]\in X$, hence of
degree $p-1$. Its ramification points are the fixed points of
$x\mapsto\eta x$, namely $x=0$ (the points $[0:i:1]$, $i\in\F_p$, since
$y^p-y=\prod(y-i)$) and $[1:0:0]$; the stabilizer at each is all of
$C_{p-1}$, and for a tame map the ramification coefficient is $|G_x|-1=p-2$. One point for the wild cover, $p+1$ for the tame one: the onion and the beetles of the figures.
\end{proof}

\section{Representations of the Symmetric Group via Ramified Covers}
\label{symmetricsultan}

Now the branch points start carrying the group. Riemann's existence theorem tells us that degree-$d$ covers of $Y$ are
equivalent to representations
$\pi_1(Y\smallsetminus\{y_1,\dots,y_n\})\to\Sigma_d$: a cover is captured by
its local monodromy. Here we use the tame $G$-cover $f\colon X\to X/G$,
$G=C_{p-1}$, to encode the remaining degree-$p$ Artin--Schreier cover through
the action of $G'/G$ on the branch points $(y_1,\dots,y_n)$ of $f$, the
mechanism promised in the Overview.

\begin{defn}
For a curve $X$ with a $G$-action and quotient $f\colon X\to X/G$, the
{\color{Bittersweet}{restricted automorphism group}} is the quotient
$\Aut^f(X)$ of $\Aut(X)$ by $\Aut(f\colon X\to X/G)$:
$$\Aut(f\colon X\to X/G)\hookrightarrow\Aut(X)\twoheadrightarrow\Aut^f(X).$$
\end{defn}

\begin{example}
The subgroup $C_{p-1}\hookrightarrow C_{(p-1)^2}$, $\zeta\mapsto\zeta^{p-1}$,
is exactly $\Aut(f\colon X\to\PP^1)$: it is the central $C_{p-1}$ of $G'$.
\begin{center}\begin{tikzcd}
\Aut(f\colon X\to X/G)\arrow[r,hook] & \Aut(X)\arrow[r,two heads] & \Aut^f(X)\\
C_{p-1}\arrow[r,hook] & C_p\rtimes C_{(p-1)^2}\arrow[r,two heads] & C_p\rtimes C_{p-1}
\end{tikzcd}\end{center}
This is the group-action clarification of the Conventions: the acting group
of the paper is $G'$, its action on the branch configuration factors through
$C_p\rtimes\Aut(C_p)$, and the central $C_{p-1}$ acts on configurations
trivially; it will act only through the grading.
\end{example}

For the rest of the section let $B_{\mr{aff}}$ be the branch points of
$f\colon X\to\PP^1$ other than the point at infinity.

\begin{lemma}\label{lem:branchfaithful}
$\Aut(X)$ preserves the set $B_{\mr{aff}}$, and an element of $\Aut^f(X)$
fixing $B_{\mr{aff}}$ pointwise is the identity.
\end{lemma}
\begin{proof}
By Lemma \ref{rammymammy}, $B_{\mr{aff}}=\{[i:1]:i\in\F_p\}$. The generator
of $\Z/p$ acts by $[i:1]\mapsto[i+1:1]$; the generator $\tau$ of
$\Z/(p-1)^2$ acts by $[i:1]\mapsto[\eta i:1]$ with $\eta$ a primitive
$(p-1)$-st root of unity in $\F_p$, fixing $[0:1]$ and permuting the rest in
a $(p-1)$-cycle. Every nontrivial element of $\Aut^f(X)\simeq C_p\rtimes
C_{p-1}$ therefore moves some point of $B_{\mr{aff}}$: shift, scale, or both, and nothing stands still.
\end{proof}

\begin{example}
Here it is with actual numbers, the kind of check one does on a diner napkin at two in the morning. Abbreviate $[i:1]\in\PP^1_{\F_5}$ as $[i]$; $2$ is a primitive $4$th root of
unity in $\F_5$, and acting by it sends the ordered branch points
$\{[0],[1],[2],[3],[4]\}$ to $\{[0],[2],[4],[1],[3]\}$ - this is a $4$-cycle in
$\Sigma_5$ fixing $[0]$.
\end{example}

\begin{cor}[symmetric subtlety]\label{symmetric_subtlety}\stmt{
With $\Sigma_{B_{\mr{aff}}}\simeq\Sigma_p$ the symmetric group on the $p$
affine branch points, there is an injective homomorphism
$$\Aut^f(X)\hookrightarrow\Sigma_{B_{\mr{aff}}}=\Sigma_p.$$}
\end{cor}
\begin{proof}
Lemma \ref{lem:branchfaithful} says the kernel is trivial, and that is all injectivity ever asks for.
\end{proof}

\begin{remark}
The proof was inspired by a theorem in complex geometry: for $W(X)$ the Weierstrass points of $X$, there is a faithful map $\Aut(X)\to\Sigma_{W(X)}$. Gorgeous stuff. Keeping the point at infinity would give the weaker embedding into $\Sigma_{p+1}$.
\end{remark}

This concludes the analysis of the curve $X$ built from $G'$, with its
natural decoration by $G$. In Part \ref{sec:configpart} we deform it.

\part{Curves as Configurations, Configurations as Representations}
\label{sec:configpart}
\section*{Overview}

Part \ref{curveprops} took only our group $G'$ and handed us one curve, without a choice made anywhere. Now we let that curve move, and watch who moves it: the aim of this Part is the $\Aut(X,G)$-action on $\Def_{(X,G)}$, in two stages. Everyone moves differently once the music starts. $C_p$ shifts, $C_{(p-1)^2}$ scales, and the center stands by the wall doing nothing at all, which turns out to be the whole point. The dance is not only local. The same action lives on the global configuration space $\mr{Conf}(\A^1,p)$ of $p$ points on the line, of which $\Def_{(X,G)}$ is the formal neighborhood at the branch points of our curve; so, through Theorem \ref{equivalentexchange}, the Lubin--Tate action of $G'$ extends to an action on all of configuration space. 

We show in \cite[Thm.\ 7.41]{ray} that deformations of any curve $X$ with a tame totally ramified action of $G$ are equivalent to deformations of its quotient $X/G$ and the branch points of $q:X \to X/G$. Further, the equivalence is $\Aut(X, G)/G$ equivariant. In our example, the $G'/G$-action on our deformation space $\Def_{(X, G)}$ can be understood as a restriction of the permutation action of the group $G'$ on the branch points of the map $q: X \to X/G.$ 

\begin{manualtheorem}{\ref{theoremb1}} \stmt{There is a $G'/G$-equivariant isomorphism of moduli problems $$L \colon \Def_{(X, C_{p-1})} \simeq \Def_{(\bb{P}^1, p+1)} \simeq  \Def_{(\bb{A}^1, p)}.$$} \end{manualtheorem}

The second phase is to understand the action of $G'/G$ on $\Def_{(\bb{A}^1, p)}$, which we do by passing to a global moduli of configuration spaces of $n$ points on $\A^1$. We also connect this analysis to past work on other global stacks related to $\Def_{(X, G)}$ in Subsection \ref{sec:pastandfuture}.

\begin{center}\includegraphics[width=6.5cm]{being_vessel/insertion.jpg}
\captionof{figure}{The curve $X$ with its ramification divisor (left), and its tame quotient $\PP^1$ with the branch divisor $B_f$ (right). The functor $L$ carries the action of $\Aut(X)$ on the left to the permutation of the branch points on the right.}\end{center}

\begin{manualtheorem}{\ref{theoremb2}} \stmt{The stack $\mr{Conf}(\bb{A}^1, n)$ is representable as a moduli problem, and the ring representing it is as a $\Sigma_n$-module $$R \simeq \Sym(\lambda)[N^{-1}] \stackyq \G_m,$$ where $\lambda$ is the standard $n-1$-dimensional representation of $\Sigma_n$.} \end{manualtheorem}

Combining Theorems \ref{theoremb1} and \ref{theoremb2}, we get the $G'/G$-action on $\Def_{(X, G)}$ stated fully in Theorem \ref{actup}. Step back and look at what has happened, because it is wild. \stmt{The action of $C_p$ (wildly ramified, and the whole difficulty of the Lubin--Tate action) has become a permutation of $p$ points on $\A^1$: the shift $t\mapsto t+1$ of the branch points $\F_p$, with $C_{(p-1)^2}$ scaling them.} Nothing about it is hidden in the curve anymore; it is all on the line, where one can watch the beetles crawl. 

\section{Theorem B.1: Curves as Configurations of Branch Points}
\label{sec:theoremb1}
Before we get started on showing Theorem \ref{theoremb1}, let us fully justify the choice of $C_{p-1}$ in our deformation problem. 

Inverse Galois theory behaves will with respect to deformation problems. Recall that $G$-Galois covers of curves with source $X$ are equivalent to $G$-actions on $X$. Under this equivalence, ramification points correspond to points with nontrivial stabilizer group. Our curve $X$ in characteristic $p$ defined solely by having automorphism group $G'$ has a convenient tame scapegoat, $G := C_{p-1}$. The points of $X$ with nontrivial stabilizer under the action of $G'$ and under the action of its subgroup $G = C_{p-1}$ coincide. We use this to our advantage, work with $p_y: X \to X/C_{p-1}$ and encode the remaining $G'/G$ action through its action on the $p$ affine branch points $$G'/G \hookrightarrow \Sigma_{B_{\mr{aff}}}=\Sigma_p.$$ Thus recording of the action on branch points exactly encodes the corresponding $C_p$-cover.

\emph{Tools from the prequel.} We use the equivariant pushforward $q^G_* \colon \QCoh^G(X) \to \QCoh(X/G)$ and the exact sequence $H^1(Y, q^G_*T_X) \to \Ext^1_G(\Omega_X, \OO_X) \to H^0(Y, R^1q^G_*T_X)$ of \cite[Thm. 7.24]{ray}; see \cite[\S 7.6]{ray} for the definitions. For a tame, totally ramified action we use \cite[Thm.\ 7.41]{ray}: $L \colon \Def_{(X,G)} \simeq \Def_{(X/G; B)}$, with $\dim H^1(Y, q^G_*T_X) = 3g_Y - 3 + |B|$ \cite[Cor.\ 7.40]{ray}.

\begin{theorem} 
For $X$ the Artin--Schreier curve and $G = C_{p-1}$ described in Section \ref{curveprops}, $\Def_{(X, G)}$ is representable and represented by a ring $R \simeq W(k)[[t_1, ..., t_{p-2}]]$ with $\Aut(X)$-action induced by \cite[Def. 2.2]{ray}.
\end{theorem}

\begin{proof} This is a special case of \cite[Thm. 7.24]{ray}. As before, we are putting together 

$$H^1(Y, q_*^GT_X) \to \Ext^1_G(\Omega_X, \OO_X) \to H^0(Y, R^1q_*^GT_X).$$

Let's start with the image. Since $R^1q_*^GT_X$ is a skyscraper sheaf $H^0(Y, R^1q_*^GT_X) \simeq R^1q_*^GT_X.$ By \cite[Lemma 7.23]{ray}, $H^0(Y, R^1q_*^GT_X) \simeq 0.$ Thus, $R_\sigma \simeq W(k)$.

By \cite[Cor.\ 7.40]{ray}, $\dim H^1(Y, q_*^GT_X) = 3g_{\PP^1} - 3 + (p+1) = p-2$, since $|B| = p+1$ by Lemma \ref{rammymammy}. 

The local-to-global spectral sequence also shows \cite[3.3.3]{bm}, \cite[3.3]{wewers} that $$\Ext^2_G(\mc{L}_{X/k}, \OO_X) \simeq H^1(Y, \mc{E}xt^1_G(\mc{L}_{X/k}, \OO_X)).$$ However, $\mc{E}xt^1_G(\mc{L}_{X/k}, \OO_X)$ is a skyscraper sheaf by \cite[Lemma 7.23]{ray}, whose cohomology vanishes above $H^0$. Hence, $\Ext^2_G(\mc{L}_{X/k}, \OO_X)$ vanishes and $\Def_{(X, G)}$ is unobstructed.

Since $G$ is in the center of $\Aut(X)$, $\Aut(X, G) \simeq \Aut(X)$ as shown in \cite[Ex. 7.7]{ray}. Thus, the automatic action of $\Aut(X, G)$ given by \cite[Def. 2.2]{ray} is in fact an action of $\Aut(X)$.
\end{proof}

\begin{manualtheorem}{B.1} \label{theoremb1} For $X$ the Artin--Schreier curve over a field $k$ with a tame action by a group $C = C_{p-1}$ described in section \ref{curveprops}, let $B$ be the set of branch points of the quotient map $q_*: X \to X/G$ described in Lemma \ref{rammymammy}. Then, there is an $\Aut(X)/G$-equivariant equivalence of formal moduli problems
$$\Def_{(X, G)} \simeq \Def_{(\PP^1; p+1)} \simeq \Def_{(\A^1, p)}$$
\end{manualtheorem} 

\begin{proof} The isomorphism of formal moduli problems follows as an example of \cite[Thm.\ 7.41]{ray}, since $(X/G, B) \simeq (\PP^1, p+1)$ by Lemma \ref{rammymammy}. The isomorphism $\Def_{(\PP^1; p+1)} \simeq \Def_{(\A^1, p)}$ follows by fixing one of the points to be the point at infinity, thus restricting to the affine branch points $|B_{\mr{aff}}| = p$.

We now justify its enhancement to an equivariant equivalence. Our deformation problems are discrete by \cite[Thm. 7.17]{ray}, which implies we have an equivalence of automorphisms of the universal object and automorphisms of the object being deformed by \cite[Lemma 7.3]{ray}. 

The action $\Aut(X, G) \simeq \Aut(X)$ since $G$ is in the center of $\Aut(X).$ To check that $L: \Def_{(X, G)} \to \Def_{(X/G; b_j)}$, it suffices to check that $\Aut(X,G)/G \to \Sigma_{|B_{aff}|}$ is an injection. This is shown in Lemma \ref{symmetric_subtlety}. Since it's an injection, it induces an isomorphism. 

\end{proof}

\section{Theorems B.2 and B, Act Up: the Configuration Space as a Representation}
We use the configuration stack $\mr{Conf}(\A^1, n)$, the standard representation $\lambda$ of $\Sigma_n$, and the reduced regular representation $\overline{\rho}$ of $C_n$, together with the identification of $\mr{Conf}(\A^1,n)$ with $\Sym(\lambda)[N^{-1}] \stackyq \G_m$ as a $\Sigma_n$-representation \cite[Thm.\ 7.45]{ray}, \cite[Cor.\ 7.47]{ray}.

\begin{manualtheorem}{B.2} \label{theoremb2} \stmt{
In $\text{char } k = p$, consider $\F_p$ in $\A^1 \simeq \Spec k[t]$, that is, the set of points $P := \{g^i(p)\}$ where $g$ is the generator of $\Z/p$ which acts by $t \mapsto t+1$, then $$\Def{(\A^1; n)} \simeq \mr{Conf}(\A^1, n)^{\wedge}_P$$ is representable by the graded ring $$R \simeq \Sym(\lambda)[N^{-1}]^{\wedge}_{I} \stackyq \G_m $$ as a $\Sigma_n$-representation. Here $\lambda$ is the standard representation
$$\lambda := W(k)\{y_0, \dots, y_{n-1}\}/(y_1 + \cdots + y_{n-1}), \qquad |y_i| = -2,$$
and
$$N=\prod_{i = 0}^{n-1} \sigma^i(y_0), \qquad I = (p, \delta_1,\dots,\delta_{n-1}),\quad \delta_j:=(1-\sigma^j)(y_0),$$
where $\sigma$ generates $C_n$; the $\delta_j$ are the displacements of the branch points.}
\end{manualtheorem}

\begin{proof} This is immediate from \cite[Cor.\ 7.47]{ray} of \cite[Thm.\ 7.45]{ray}.
\end{proof}

\begin{manualtheorem}{B} \label{actup} \stmt{Let $R$ be as in Theorem \ref{theoremb2},
$\Def_{(X, G)}^\star$ is represented by $R$ as a $Aut(X)/G$-representation.}
\end{manualtheorem}

\begin{proof} This immediately follows combining Theorem \ref{theoremb2} with Theorem \ref{theoremb1}, the latter of which states that $$\Def_{(X, G)} \simeq \Def_{(\A^1, p)}$$ and this is an equivariant equivalence with respect to $\Aut(X)/G \hookrightarrow \Sigma_p$. 
\end{proof}

\begin{remark}[the induced module and its gauge] In the language of \cite{higherlevels}: $\lambda$ restricted to $C_p \rtimes C_{p-1} = \F_p \rtimes \F_p^\times$ is $\overline{\rho} = P(\theta)/\rad^{p-2}P(\theta)$, where $P(\theta) = \Ind_{C_{p-1}}^{\F_p \rtimes \F_p^\times}\theta = k[\F_p]\otimes\theta$ is the regular representation of the branch orbit twisted by the tangent character, and the sum-zero condition $y_0 + \cdots + y_{p-1} = 0$ together with the quotient by $\G_a$ is the quotient by the \emph{gauge}, the affine functions on the orbit. So Theorem \ref{actup} is the $k=1$ case of \cite[Thm.\ B]{higherlevels}, where only the gauge grows: $\rad^{p-2}$ becomes $\rad^{h-1}$. And $\mr{Conf}(\A^1,p)$ is the stack of surgeries of pole order zero at the branch orbit \cite[\S 7.11]{ray}, \cite[\S 2]{higherlevels}. \end{remark}

\begin{remark} Once we have $R$ as an $\Aut(X)/G$ representation, the full $\Aut(X)$-action is determined by the grading character (Lemma \ref{lem:gradedaction}). \end{remark}

\subsection{Comparison with Global Stacks: Marked Points as Level Structure}
\label{sec:pastandfuture}

The picture of $\Aut(X)$ acting on $\Def_{(X,G)}$ through the branch points
has a global form: a moduli stack of $(p-1)$-elliptic curves with ordered
marked points, the analogue of the Legendre family
$x^2=y(y-1)(y-\lambda)$, i.e.\ of $X(2)$, the moduli of elliptic curves
with level $2$ structure \cite[Prop.\ 7.1]{vesna}. For $p>3$ the curves
carry no group structure, and the ordered branch points take the place of
the level structure.

\begin{defn}\label{Fordtruck}(ford truck) Let $F^{\mr{ord}}$ be the moduli
functor on $\Z[\frac1{p-1}]$-schemes whose objects over $R$ are smooth
families $x^{p-1}=(y-y_0)\cdots(y-y_{p-1})$ with $y_0,\dots,y_{p-1}\in R$
pairwise distinct, marked by the ordered points $P_i=[0:y_i:1]$, and whose
morphisms are the coordinate changes $x\mapsto\mu^{-p}x$,
$y\mapsto\mu^{-(p-1)}y+c$ preserving this form.
\end{defn}

\begin{theorem}\label{itsafuckingstack} The stack $\mc M_{g,C_{p-1}}$ of
smooth curves of genus $g=(p-1)(p-2)/2$ with a $C_{p-1}$-cover of $\PP^1$ is
equivalent to $F^{\mr{ord}}$, and $F^{\mr{ord}}$ is represented by
$(\A^p-\Delta)\stackyq(\G_m\rtimes\G_a)$, where $\G_a$ translates the $y_i$
and $\G_m$ acts by $\mu^{p-1}$; in the coordinates $r_i=y_i-y_0$ this is
$(\A^{p-1}-\{0\})\stackyq\G_m\simeq\PP^{p-2}$.
\end{theorem}
\begin{proof} As in \cite[Prop.\ 7.1]{vesna}. A $C_{p-1}$-cover of $\PP^1$
is a Kummer extension $x^{p-1}=f(y)$ \cite[Lemma 7.27]{ray}, and
Riemann--Hurwitz forces $\deg f=p$ with $p+1$ branch points, one at
infinity; over an \'etale extension in which $f$ splits the curve has the
form of Definition \ref{Fordtruck}, and conversely every such curve carries
the action $x\mapsto\zeta x$ with the $P_i$ as its fixed points. The
morphisms of $F^{\mr{ord}}$ are exactly the changes of variables of $\PP^2$
preserving the split form, which is $\G_m\rtimes\G_a$, and the discriminant
is inverted because the $y_i$ are distinct. Quotienting by $\G_a$ first,
$y_i\mapsto r_i=y_i-y_0$, leaves $\A^{p-1}-\{0\}$ with $\G_m$ acting by
$\mu^{p-1}$.
\end{proof}

\begin{remark} The Hopf algebroid of $(p-1)$-elliptic curves without the
ordering is the one used by Hill \cite{hillphd} to compute
$\pi_*EO_{p-1}$. The difference from $\mr{Conf}(\A^1,p)$ of \cite[Thm.\ 7.45]{ray} is only the $\G_m$-action, a stacky memory of the
$C_{p-1}$-action on the branch locus. The stack $\mc M_{g,C_{p-1}}$ is also
the $C_{p-1}$-fixed locus of the moduli of curves with level structure on
their Jacobians \cite[Thm.\ 5.2, 5.4]{bm}, and its completion at $X$ is
$\Def_{(X,G)}$ \cite[Lemma 5.2.1]{bm}.
\end{remark}

\begin{lemma}\label{pastandfuture}\stmt{(past and future) Let $R$ be
the set of points of $X$ with nontrivial stabilizer under $G=C_{p-1}$. Then
$(F^{\mr{ord}})^\wedge_{(X,R)}\simeq\Def_{(X,G)}$.}\end{lemma}
\begin{proof} Send $(\mf X,G)$ to $(\mf X,\mc R)$, $\mc R$ its points with
nontrivial stabilizer. All curves in $\Def_{(X,G)}$ have the genus of $X$
(Chevalley--Weil, \cite[Prop.\ 3.1]{champs}), so essential surjectivity and
full faithfulness follow from Theorem \ref{itsafuckingstack}.\end{proof}

\begin{remark} These global problems sit in the smooth locus of a Hurwitz stack, whose compactification is the natural home for a connective $\mr{eo}_{p-1}$: a curve-theoretic alternative to $\mr{taf}$ and to the Hopf algebroid of \cite{hillphd}. We leave its spectrum to another place and time. A longer comparison with related global stacks is in the first arXiv version of this paper (arXiv:2509.23428v1). \end{remark}

\part{Controlling the p-Divisible Group of the Jacobian of our Curve}

\section*{Overview}
\label{sec:pdiv}

Part \ref{curveprops} showed that the automorphisms determine the curve. This Part shows that they determine much more: the same automorphisms fix the dimension and the height of every piece of the formal group of its Jacobian. We consider the curve $X$ with affine equation $y^p-y = x^{p-1}$ over a characteristic $p$ field $k$ which contains $\F_{p^{p-1}}$, and devote this Part to the formal group of its Jacobian --- a large abelian variety with one small summand we care about. The center $G := C_{p-1}$ of $G' := \Aut(X)$ splits it by idempotents; the eigenvalues of $G'$ on differentials fix the dimension of each summand; and through Manin's formula the automorphisms fix the slopes of Frobenius, hence the heights. Nothing about the Jacobian is chosen either. 

We use the idempotent decomposition of a $p$-complete additive category with an action of a group $G$ of order prime to $p$, and its equivariance for the centralizer of $G$ \cite[\S 8.2]{ray}, \cite[Lemma 8.1]{ray}.

The Jacobian has dimension $(p-1)(p-2)/2$ and we want one dimension of it. This is the oldest story there is: the party is enormous and you came for one person. Fix $X$ to be the Artin--Schreier curve over $k=\F_{p^{p-1}}$ with affine form $y^p-y = x^{p-1}.$ 

\begin{manualtheorem}{\ref{splittydooda}} \stmt{(splitty dooda) Let $A$ be the p-divisible group of $\Jac_k(X)$, then $A$ has an $C_{p-1}$ idempotent integral decomposition
$$A \simeq \bigoplus_{\chi \in \Hom(C_{p-1}, \G_m)} A^{\chi^i}.$$
into $p-2$ summands of dimension $1, 2, ..., p-1$, all of height $p-1$, where $A^{\chi^i}$ is of dimension $i$. There is no component corresponding to the trivial character.}
\end{manualtheorem}

\begin{manualcorollary}{\ref{hsplitting}} \stmt{(h-splitting) The $p$-divisible group $A^{\chi^1}$ is connective and corresponds to a formal group of dimension 1 and height $p-1$.}
\end{manualcorollary}

\begin{center} \begin{figure}\includegraphics[width=8cm]{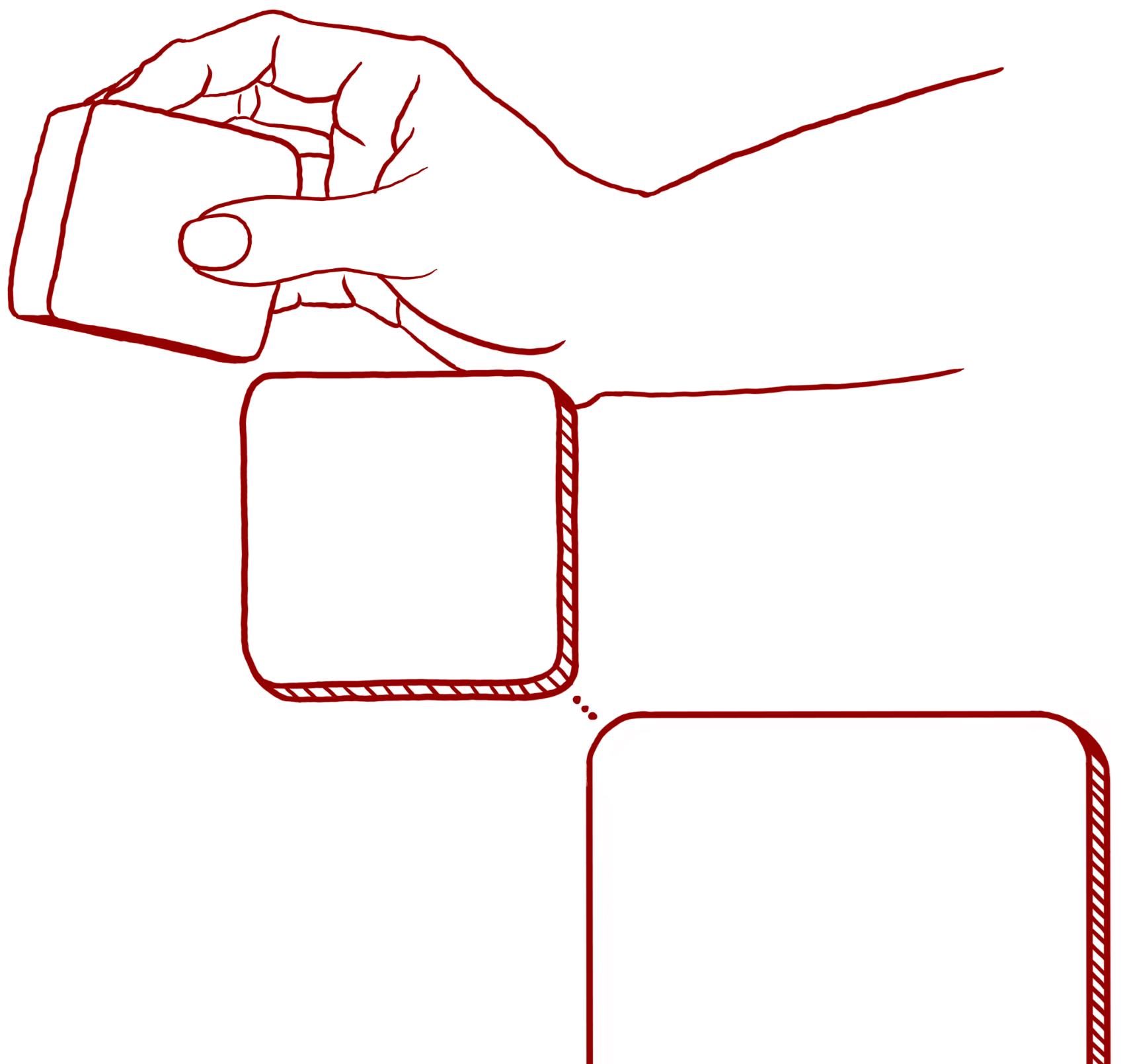} \caption{These cubes are the Jordan blocks of the decomposition of our $p$-divisible group into lower dimensional summands in Theorem \ref{splittydooda} and Corollary \ref{hsplitting}. Our functor selects the one dimensional summand off the top (Definition \ref{def:functor})}. \end{figure} \end{center}

We abbreviate $\Jac(X)$ as $\mc J$. The plan: the central $G=C_{p-1}$ splits
$\mc J[p^\infty]$ by its idempotents \cite[\S 8.2]{ray}; the
dimension of each summand is read off the $G$-action on
$H^0(X,\Omega^1_X)=T^*_e\mc J$ (Proposition \ref{diffblocks}); and the
height of each summand is read off the slopes of Frobenius, which Manin's
formula computes from the automorphisms of $X$ (Theorem \ref{slope}). Theorem
\ref{splittydooda} assembles the two through the Hodge filtration on
$H^1_{\mr{cris}}/p=H^1_{\mr{dR}}$, and Corollary \ref{hsplitting} names the
one-dimensional summand.

\section{Dimensions from Holomorphic Differentials}

We fix the following notation for the rest of this Part. Let $\Jac(X)$ be the Jacobian of a projective curve $X$ with affine form $y^p-y = x^{p-1}$ over $R \in \Art_k$, and let $\Jac(X)_e^\wedge$ be its formal group.

\begin{conv}[group element names]\label{conv:characters1}
$G=C_{p-1}=\langle\tau^{p-1}\rangle$ is the central tame subgroup, acting by
$x\mapsto\eta x$, $y\mapsto y$, $\eta=\zeta^{p-1}$ where $\zeta$ is a $(p-1)^2$root of unity. \end{conv}

\begin{proposition} \label{diffblocks} (decomposition) 
\begin{enumerate} 
\item The holomorphic differentials $\Omega := \Omega_X^1$ have the basis $$\{x^iy^j \omega \colon 0 \leq i + j \leq d-3\}$$ where $\omega := \frac{dy}{P_x} = \frac{dx}{P_y}$ where $P_x, P_y$ are the derivatives of $P$ with respect to $x$ and $y$ respectively.
\item Let us consider characters of $G$, $\chi \in \Hom(G, \G_m).$ The group $G$ acts on $\Omega$ by $\chi^{-(i+1)}$, $$\dim H^0(\Omega^1)^{\chi^j}=j\quad(1\le j\le p-2),\qquad H^0(\Omega^1)^{\chi^0}=0,$$
\item The $\chi^1$-block is the line $k\omega_1$, $\omega_1=dy/x$. The generator $\tau$ of $C_{(p-1)^2}$ acts on $\omega_1$ by $\zeta^{-1}$, a primitive $(p-1)^2$-th root of unity, and $\sigma$ fixes $\omega_1$.
\end{enumerate}
\end{proposition}

\begin{proof}
(1) A genus of smooth projective plane curve of degree $d$ is $(d-1)(d-2)/2$, so we need to construct that many linearly independent holomorphic differentials. For a smooth plane curve $X$ given by equation $P(x,y) = 0$ we have $P_x dx + P_y dy = 0$ (by differentiating $P = 0$), where $P_x, P_y$ are the derivatives of $P$ with respect to $x$ and $y$ respectively. Since $X$ is smooth the vector $(P_x, P_y)$ doesn't vanish anywhere on $X$. Hence we can define a holomorphic nowhere (on $X$) vanishing 1-form $\omega = \frac{dy}{P_x} = - \frac{dx}{P_y}$. We may also choose any non-zero scalar multiple of $\omega$ as our $\omega$.

Let us examine pairs $(i, j)$ such that $0 \leq i + j \leq d-3$. This is choosing ordered pairs of numbers from the set $\{0,1,...,d-2\}$, in other words, from a set of $d-1$ numbers. There are then ${d-1 \choose 2} = \frac{(d-1)(d-2)}{2}$ of such $(i, j)$. 

We thus get a set of precisely $(d-1)(d-2)/2$ holomorphic forms $x^i y^j \omega$ on $X$, and they are linearly independent, hence a basis: the canonical sheaf of a smooth plane curve of degree $d$ is $\OO_X(d-3)$ \cite[II, Example 8.20.3]{hartshorne}, so $H^0(X,\Omega^1_X)$ is the restriction of the forms of degree $d-3$ on $\PP^2$, and the monomials $x^iy^j$ with $i+j\le d-3$ restrict injectively because $X$ lies on no curve of degree less than $d$.

(2) Under $x \mapsto \eta x$, using the notation of convention \ref{conv:characters1}, $\omega := \frac{dy}{x^{p-2}}$ is of weight $\chi^{-(p-2)}(\eta) = \eta^{-(p-2)},$ and the weight of $x^iy^j\omega$ is $\eta^{i-(p-2)}$. Now we do some combinatorics to show the dimension breakup. There are $d-2$ pairs such that $i=0$, i.e., $(0,1), (0,2), ..., (0, d-2)$. Further, there are $d-3$ pairs such that $i=1$, and in general, $(d-2)-k$ pairs such that $i=k$. Ending with 1 pair such that $i=d-3$, i.e., $(d-3, d-2)$. Let's set $a = i-(p-2)$. Thus, according to the pairs computation we just did, $\dim(\chi^a) = a$ for $1 \leq a \leq p-2$, and nothing for $a=0.$

(3) This is a direct consequence of (2). There is only one element such that $i=p-3$ ($a=1$), which is $x^{p-3}\frac{dy}{x^{p-2}} = \frac{dy}{x}$. Finally, $\tau^*(dy/x) = \frac{\zeta^{(p-1)}dy}{\zeta^p x} = \zeta^{-1}\frac{dy}{x}.$ The action of $\sigma$ is trivial, since it has order $p$ and $k^\times$ has no nontrivial $p$-th roots of unity. Hence, $\sigma$ fixes $x$ and $dy$.
\end{proof}

\begin{remark} This is the
Lubin--Tate character row of the table in Subsection \ref{oneidentity}: the wild group acts on the block through $\Z_p[\zeta_p]$, and $\pi=\zeta_p-1$ is its uniformizer, with $$\pi^{p-1}=-p\cdot(\text{unit}),$$ because $[p](T)/T=((1+T)^p-1)/T$ is Eisenstein of degree $p-1$. This is the identity of Subsection \ref{oneidentity} in its third costume: the height of $F$ is $[\Q_p(\zeta_p):\Q_p]=p-1$ \emph{because} the ramification index of $\Q_p(\zeta_p)/\Q_p$ is the exponent in $\pi^{p-1}=-p$, and the slope $1/(p-1)$ of Manin's formula (Lemma \ref{slope}) is $1/e$. \end{remark}

\section{Isogeny Decomposition Using Eigenvalues of Frobenius}
\label{eigenvalues}
We now describe its slope decomposition, which is a description of $H^1_{cris}$ up to isogeny, i.e. $H^1_{cris}[1/p]$. These are the eigenvalues of Frobenius. 

Let us consider $\chi$ to be a multiplicative character, and $\phi$ to be an additive character. An additive character is a character $\phi$ of the additive group $\F_q$ so $\phi(a+b) = \phi(a)\phi(b)$ for all $a$ and $b$. A multiplicative character is a character $\chi$ of the unit group $\F_q^\times$ extended to be $0$ on nonunits in order to be defined on the whole ring, while preserving the rule $\chi(a)\chi(b) = \chi(ab).$ For a multiplicative character $\chi$ on $\F_q$, its Gauss sum is $G(\chi) = \sum_{a \in \F_q} \chi(a) \zeta^{\mr{tr}(a)}$ where $\mr{Tr}: \F_q \to \F_p = \Z/p\Z$ is the trace mapping, and $\zeta$ is a a $p$th root of unity. 

Let $t \in \F_{p}^\times$ and lets choose a lift $\tilde{t}$ to $W(k)$. We define $\phi_i(t) = \zeta^{i \tilde{t}}$, and $\chi_j(t)= - \tilde{t}^{-j}$.
\begin{theorem} (Thm 4.1 \cite{manin}, \cite{katz}) \label{slope}  The eigenvalues of Frobenius on the curve $X: y^p- y = x^{p-1}$ are sums of the following form. 

$$\tau(\phi_i, \chi_j)=\sum_{t \in \F_p^{\times}} \phi_i(t)\chi_j(t).$$

where $1 \leq i \leq p-1$ and $1 \leq j \leq p-2$. 
\end{theorem}

\begin{remark} This is originally a theorem of Davenport-Hasse. A beautiful proof of this by Coleman \cite{coleman} constructs a correspondence associated to $\Aut(X)$ which is equivalent to Frobenius as a correspondence. This determines the eigenvalues of Frobenius as an operator acting on $H^1_{\mathrm{\acute{e}t}}(X)$, which give the slopes of our $p$-divisible group and the height decomposition of our formal group up to isogeny. In this way the automorphisms of the curve $X$ entirely determine its eigenvalues. \end{remark}

\begin{cor} \label{sloped} The slopes of the p-divisible group associated to the curve $X$ are $\{ 1/(p-1), 2/(p-1), ..., (p-2)/(p-1) \}$. \end{cor}

\begin{proof} We use the eigenvalues above. The key observation of Stickelberger is that for $\lambda = 1-\zeta$, 
$$\tau(\phi_i, \chi_j)=-j^{-1}\lambda^j \mod \lambda^{j+1}.$$

Since $v_p(\lambda) = 1/(p-1)$, this means that
$v_p(\tau(\phi_i, \chi_j)) = \frac{j}{p-1}.$ So we get $(p-1)$ copies of each $1 \leq j \leq (p-2).$ Each of these eigenvalues has multiplicity $p-1.$
\end{proof}

We now show how to go from this slope decomposition up to isogeny to an integral decomposition. 

\section{Theorem C.1, Splitty Dooda: $H^1_{cris}(X)$ Gives the Formal Group}

\label{bigboysection}
We have collected our lemmas. We are now ready to prove Theorem \ref{splittydooda} and most importantly Corollary \ref{hsplitting}.

\begin{manualtheorem}{C.1} \label{splittydooda} \stmt{(splitty dooda) Let $A$ be the p-divisible group of $\Jac_k(X)$, then $A$ has an $C_{p-1}$ idempotent integral decomposition
$$A \simeq \bigoplus_{\chi \in \Hom(C_{p-1}, \G_m)} A^{\chi^i}.$$
into $p-2$ summands of dimension $1, 2, ..., p-1$, all of height $p-1$, where $A^{\chi^i}$ is of dimension $i$. There is no component corresponding to the trivial character.}
\end{manualtheorem}

\begin{proof} Theorem \ref{tinyandunique}(3) says that $X$ has $p$-rank 0, hence $\Jac(X)[p]$ (and thus $\Jac(X)[p^\infty]$) are connected. Thus, all summands of $\Jac(X)[p^\infty]$ are also connected. Let $\chi$ be a multiplicative character of $\Z/(p-1)$. We use Theorem \ref{diffblocks}, which tells us that $H^0(X, \Omega^1_X) \simeq \bigoplus_{j \in \Z/(p-1)} H^0(X, \Omega^1_X)^{\chi^j}$ breaks up, where the components acted on by $\chi^j$ are $j$-dimensional. 

Further, by Serre duality, $$H^0(X, \Omega^1_X) \simeq H^1(X, \mc{O}_X)^{\wedge}.$$

Because of the dual in this isomorphism, the summand acted on by $\chi^j$ is sent to the summand $\chi^{-j}$, this is crucial for what follows. In particular, the summand of $H^1(X, \mc{O}_X)^{\chi^j}$ is $(p-1)-j$-dimensional. 

\begin{align*}
H^0(X, \Omega^1_X)^{\chi^j} \to H^1(X, \mc{O}_X)^{\chi^{-j}} \\
H^0(X, \Omega^1_X)^{\chi^{-j}} \to H^1(X, \mc{O}_X)^{\chi^{j}}
\end{align*}

Since $H^1_{dR}(X)$ is an extension of $H^1(X,\mc{O}_X)$ by $H^0(X,\Omega^1_X)$ as a $\Z/{p-1}$-representation, the summand $H^1_{dR}(X)^{\chi^j}$ has dimension $j+p-1-j=p-1$ for every $1 \leq j \leq p-2$. We use the following short exact sequence.

$$H^0(X, \Omega^1_X) \to H^1_{dR}(X) \to H^1(X, \mc{O}_X)^\vee.$$

It follows that in the decomposition $H^1_{cris}=\bigoplus (H^1_{cris})^{\chi^j}$ every summand has rank $p-1$ as a W(k)-module. By Manin's slope formulas (Lemma \ref{slope}), slope $1/(p-1)$ appears at least once in $H^1_{cris}$. Therefore it appears in some $(H^1_{cris})^{\chi^j}$, which is to say that  $(H^1_{cris})^{\chi^j}[1/p]$ has $D_{1/(p-1)}$ as a direct summand. 

Since the dimension of $(H^1_{cris})^{\chi^j}[1/p]$ is $(p-1)$ and not larger, we get that $D_{1/(p-1)}$ is actually all of $(H^1_{cris})^{\chi^j}[1/p]$. So $(H^1_{cris})^{\chi^j}$ is the integral Dieudonn\'e module of some p-divisible group, and the sum of all slopes appearing is $(p-1)*1/(p-1)=1$ so the dimension of this p-divisible group is 1. The same argument works for all $(p-1)*j/(p-1)=j$ where $0 < j \leq p-1$, giving a $p$-divisible group of dimension $j$ and height $p-1$.
\end{proof}

We explicitly draw attention to the dimension one formal group, the spoils of our labor. 

\begin{manualcorollary}{C.1'} \label{hsplitting} {\stmt{The formal group $A^{\chi^1} =: F$ has dimension $1$ and height $p-1$, with Lie algebra dual to the Hodge line $k\omega_1$}}
\end{manualcorollary}

In Dieudonne terms: Given the $C_{p-1}$-decomposition $$H^1_{cris}(X, \Omega^1_X) \simeq \bigoplus_{j \in \Hom(C_{p-1}, \G_m)} H^1_{cris}(X, \Omega^1_X)^{\chi^j},$$ the Dieudonn\'e module of the formal group of dimension $1$ and height $h$ is the summand of $H^1_{cris}(X, \Omega_X^1)$ corresponding to $j=1$.

\begin{proof} We can conclude that $j=1$, from the proof of Lemma \ref{diffblocks}. This is because a 1-dimensional p-divisible group has to have a 1-dimensional Lie algebra, and its Lie algebra is dual to $F^1_{Hodge}$ on $(H^1_{cris})^{\chi^j}/p=(H^1_{dR})^{\chi^j}$ but by Lemma \ref{diffblocks} this $F^1_{Hodge}$ has dimension $j$. \end{proof}

\begin{remark}[splitting by Frobenius alone] The slopes above come from the eigenvalues of Frobenius (Stickelberger). At level $k=2$, \cite[Thms.\ C, D]{higherlevels} locate the Lubin--Tate module inside $H^1(X_2, \OO)$ using only the Cartier operator $\mc{C}$ on $H^0(\Omega^1)$: on the $\theta$-isotypic part of $H^0(\Omega^1)$ for the center $H$, $\mc{C}$ has a unique long Jordan block, ending at the extremal differential. Dually, $\bar F^{\,p-1}$ of the dual isotypic part of $H^1(\OO)$ is the Lubin--Tate module, twisted by the character of $x^{p-3}dx$. The same method at $k=1$, where $X$ has $p$-rank $0$ (Theorem \ref{thm:minimal}(3)) so that $\mc{C}$ is nilpotent on $H^0(\Omega^1)$, should give a second proof of Corollary \ref{hsplitting} from the Jordan type of $\mc{C}$ on the $\chi$-isotypic part, with the twist again the character $\psi = \chi$ of $\omega_1 = x^{p-3}dx$; we have not written this out.\end{remark}

\begin{defn} \label{trace} Let us consider a finite $G$ cover $f: X \to Y := X/G$ of curves over a field $k$. This corresponds to a field theory extension $K(X)/K(X)^G$. The field theory trace is

$$Tr_{X/Y}: K(Y) \to K(X)$$ defined by $$Tr_{X/Y}(y) = \sum_{\sigma \in G} \sigma(y).$$
\end{defn}
\section{Wild Action on the Jacobian}

The tame part split the Jacobian for us; now the wild part has to be caught in the act. Wild things are easiest to catch when they think no one is watching, so we watch the tangent space. To show injectivity of the functor we will later define, we need to understand what $C_p$ does to the summand it cannot split.
\begin{proposition}[the one-dimensional block]\label{prop:wild}
Let $F:=e_1\,\Jac(X)[p^\infty]$, the $\psi$-eigenpart of the $p$-divisible
group of the Jacobian. Let $\zeta_p$ be a pth root of unity. Then, the action of $\Z_p[C_p]$ on $F$ factors through
$\Z_p[C_p]/(N)\cong\Z_p[\zeta_p]=:\OO_L$, $\sigma\mapsto\zeta_p$, and the
Dieudonn\'e module of $F$ is free of rank one over $\OO_{\breve L}:=
\OO_L\otimes W(\bar k)$.
\end{proposition}
\begin{proof}
We begin by considering the norm element $N := \Sigma_i \sigma^i$ in the group ring $\Z_p[C_p]$. The norm $N$ is the trace map of the cover $X \to X/C_p$ as defined in Definition \ref{trace}, since $K(X)/K(X)^{C_p} \simeq K(X/C_p)$ in the sense of an algebraic quotient defined in \cite[Def. 7.26]{ray}. Thus, the action of $N$ on $\Jac(X)$ factors through the quotient $\Jac(X/C_p) \simeq \Jac(\mathbb{P}^1),$ 

$$\Jac(X) \to \Jac(X/C_p) \to \Jac(X).$$

The abelian variety $\Jac(\mathbb{P}^1)$ has trivial Jacobian, as genus 0 curves have no nontrivial line bundles. Thus, $N$ acts as zero on $\Jac(X)$, and thus $N$ acts as zero on $\Jac(X)[p^\infty]$. Hence, the action of $\Z_p[C_p]$ on $\Jac(X)[p^\infty]$ factors through the quotient $\Z_p[C_p]/(N).$ This quotient ring, $\Z_p[C_p]/(N) \simeq \Z_p[\zeta_p] =: \OO_{L}$. This isomorphism sends $\sigma \mapsto \zeta_p.$ The action of $\Z_p[C_p]$ on $\Jac(X)[p^\infty]$ commutes with the idempotent $e_1$ because $G$ is central in $G'$, thus $G$ acts trivially on $\Z_p[G]$. Thus, $\OO_L$ acts on $F = e_1\Jac(X)[p^\infty].$ We have shown that the height of $F$ is $p-1$ in Corollary \ref{hsplitting}. Since the height of a formal group $F'$ over a field $k$ of characteristic $p$ is the dimension of the Dieudonn\'e module $D(F')$ over $W(k)$, this means that the Dieudonn\'e module $\bb{D}(F)$, a free $W(k)$-module of rank $p-1$, is a module over $\OO_{\breve{L}} := \OO_L \otimes_{\Z_p} W(k)$, a complete discrete valuation ring (totally ramified of degree $p-1$ over $W(k)$). Since $\bb{D}(F)$ is torsion-free over $W(k)$, it is torsion-free over $\OO_{\breve L}$, hence free over the DVR $\OO_{\breve L}$; comparing $W(k)$-ranks, $(p-1)\cdot \mr{rank}_{\OO_{\breve L}} \bb{D}(F) = p-1$, so $\bb{D}(F)$ is free of rank one over $\OO_{\breve L}$. (Here $L = \Q_p(\zeta_p)$ is totally ramified of degree $p-1$; its completion is unchanged.) 
\end{proof}

\part{Functors and Lubin--Tate Representations}
\label{sec:KSpart}
\section{Theorem C, Equivalent Exchange: the Proof via Kodaira--Spencer} 
\label{sec:KS}
This is the heart\footnote{An earlier draft got here the hard way, through an explicit universal curve, and ended with a confession: it wanted a proof with no coordinates at all. ``It would soothe my soul.'' This section is that proof, and the author found it. The soul, for the record, is soothed.}. The proof of Theorem \ref{equivalentexchange} is a coordinate-free Kodaira--Spencer argument: it needs no equation for the universal curve, and it is the proof that generalizes. It has exactly three parts, and every subsection of this section serves one of them.

\emph{(i) The functor.} The central $G=C_{p-1}$ acts on the Jacobian of
every deformation of $(X,G)$, and since $p-1$ is a unit its idempotents
$e_j$ are integral. Set
$$\widehat{\mc F}^1(\mc X,\rho,\iota):=e_1\Pic^0(\mc X)[p^\infty],$$
the $\psi$-eigenpart of the $p$-divisible group, for $\psi$ the character
of $G$ on the extremal differential $\omega_1=dy/x$. Its special fiber is
a one-dimensional formal group $F$ of height $p-1$ (Corollary
\ref{hsplitting}): the wild $C_p$ acts on $\Jac(X)$ with
$N=\sum\sigma^i=0$ because $X/C_p=\PP^1$, so it acts on the block through
$\Z_p[C_p]/(N)=\Z_p[\zeta_p]$, and the Dieudonn\'e module of $F$ is free
of rank one over $W(\bar k)[\zeta_p]$ (Proposition \ref{prop:wild}); its height $p-1=[\Q_p(\zeta_p):\Q_p]$ is the exponent of $\pi^{p-1}=-p$.
The functor is $G'$-equivariant because $G$ is central (Lemma
\ref{lem:equiv}), and $G'$ embeds in $\Aut(F)$ with surjective
conjugation character (Lemma \ref{lem:faithful}).

\emph{(ii) The derivative.} By Grothendieck--Messing, a first-order
deformation of $F$ is a displacement of its Hodge line inside the
Dieudonn\'e module, and by Katz--Oda that displacement is cup product
with the Kodaira--Spencer class. Hence the derivative of
$\widehat{\mc F}^1$ at the closed point is
$$\mu\colon H^1(X,T_X)^G\longrightarrow\Hom\big(k\omega_1,\
H^1(X,\OO_X)^{\psi}\big),\qquad \vartheta\longmapsto(\omega_1\mapsto
\vartheta\cup\omega_1)$$
(Lemma \ref{lem:derivative}): cup product with the Hodge line. This is a
map between $(p-2)$-dimensional spaces, diagonal in the characters of
$G$, and the sequence $0\to T_X\xrightarrow{\omega_1}\OO_X\to\OO_D\to0$
for $D=\mr{div}(\omega_1)$ shows it is surjective, with kernel
$H^0(\OO_D)^{\psi}$ (Lemma \ref{lem:injection}). The whole proof is
therefore one character count on the divisor of $dy/x$: that divisor is
$(p-3)\sum_{i\in\F_p}R_i$ on the tame fixed points, the characters
occurring in $H^0(\OO_D)$ are $\chi^0,\chi^{-1},\dots,\chi^{-(p-4)}$, and
the multiplicity $p-3$ is small enough, by exactly two, to miss
$\psi$ (Lemmas \ref{lem:divisor}, \ref{lem:count}). So $\mu$ is an
isomorphism (Proposition \ref{prop:KS}). The three facts consumed (the
support of $\mr{div}(\omega_1)$, its multiplicity $(2g-2)/p=p-3$, and the
extremality of $\psi$) are each consequences of minimality (Remark
\ref{rem:threeinputs}): \emph{injectivity of the Kodaira--Spencer map is
the minimal-genus theorem seen from the formal group.}

\emph{(iii) From tangent isomorphism to equivalence.} Both sides are
pro-represented by formally smooth $W(k)$-algebras of dimension $p-2$
(Lemma \ref{lem:smooth}: unobstructedness of tame equivariant
deformations of curves; Lubin--Tate), so a tangent isomorphism is an
isomorphism of representing rings and $\widehat{\mc F}^1$ is an
equivalence. The transported $G'$-action is \emph{the} Lubin--Tate action
because $G'$ sits in $\G_{p-1}$ in a single conjugacy class (Bujard,
\cite[Lemma 0.1]{ray}; see Subsection \ref{sec:identify} and \cite[\S 6.3]{ray}, where
$\pi^{p-1}=-p$ is the equation of the uniformizer $\pi=\omega_0\Pi$ of
the maximal totally ramified subfield). The graded refinement, Theorem
\ref{gradedequivalentexchange}, is a torsor statement over this one: the
$\chi$-line of the Hodge bundle of the family is canonically the module
of invariant differentials of the block (Lemma \ref{lem:hodgeline}), so
the grading is carried along with nothing further to compute, and the
action of $G'$ on the periodicity generator $u$ is read off the Hodge
line (Corollary \ref{cor:uaction}).

\begin{conv}[characters]\label{conv:characters} $G=C_{p-1}=\langle\tau^{p-1}\rangle$ is the central tame subgroup, acting by $x\mapsto\eta x$, $y\mapsto y$, $\eta=\zeta^{p-1}$. Let $\chi\colon G\to k^\times$ be the character by which $G$ acts on the differential\footnote{Not the conjugation character $\chic\colon H\to\F_p^\times$ of Part \ref{curveprops}, which is trivial on the central $G$; we write $\chic$ for that one throughout, and $\chi$ only for the character on differentials.} $\omega_1:=dy/x=x^{p-3}dx$, i.e.\ $\chi(\eta)=\eta^{-1}$. This is the convention of Part \ref{splittydooda}'s splitting (Lemma \ref{diffblocks}): the $\chi^j$-block of $H^0(\Omega^1)$ has dimension $j$. Against Part \ref{curveprops}: $\chi=\theta^{-1}$ for the tangent character $\theta$ of Lemma \ref{lem:char}, and $\chi^{p-2}=\chi^{-1}=\theta$. We write $e_j:=\frac1{p-1}\sum_{g\in G}\chi^{-j}(g)\,g$ for the idempotents of $\Z_p[G]$ (integral since $p-1$ is a unit), and $\psi:=\chi$ for the character of the one-dimensional block.
\end{conv}

\subsection{Part (i): the functor, its target, and its equivariance}

We begin by making notations and the group actions as crystal clear as possible for the reader by recounting notation.

We use the upper-star/lower-star conventions of \cite[\S 8.1]{ray}.

Let $\Def^\star_F$ be the rigidified deformation problem \cite[Def. 2.1]{ray} of the $p$-divisible group $F$ in the prestack of $p$-divisible groups: objects $(F', \alpha \colon F'_s \simeq F)$, morphisms isomorphisms compatible with $\alpha$.

\begin{remark} Note that no group action is part of the datum in this prestack (as opposed to the $G$ carried around in $\Def_{(X, G)}$). \end{remark}

\begin{conv}[objects, and who acts on what]\label{conv:threeactions}
\begin{enumerate}
\item[(a)] \emph{The lift $\rho$,} a datum, and an action of $G$ only (as in \cite[Def. 2.1]{ray}). An element of $G'\smallsetminus G$ acts on $X$ but in general not on $\mc X$: it would have to preserve the branch locus of $\mc X\to\mc X/G$, which a general deformation moves off the $G'/G$-orbit \cite[Thm.\ 7.41]{ray}. So there is no ``$g'_*$'' on $\Pic^0(\mc X)$.
\item[(b)] \emph{The action of $G'$ on the source} $(\mc X,\rho,\iota) \in \Def_{X, G}^\star(\Lambda)$ \cite[Def. 2.2]{ray}: $g'\cdot(\mc X,\rho,\iota):=(\mc X,\rho,g'\circ\iota)$, which moves the marking, not the curve. Since $G$ is central, $g'\circ\iota$ is again $G$-equivariant for the same $\rho$; so this is a well-defined left action, and no conjugation of $\rho$ enters.
\item[(c)] \emph{The action of $\Aut(F)$ on the target} $(F', \alpha) \in \Def^\star_F(\Lambda)$: $a\cdot(F',\alpha):=(F',a\circ\alpha)$, and $G'$ acts through the homomorphism $\iota_{\mr{aut}}\colon G'\to\Aut(F)$ of Lemma \ref{lem:iotaaut}, which lives on the special fiber.
\end{enumerate}
\end{conv}

\begin{construction} \label{fall}
Let $\Jac$ be the functor that takes a curve $\mf{X}$ to $\Jac(\mf{X})$ (in $\mr{AbVar}$), and consider its formal group $\Jac(\mf{X})^\wedge_e$. We define the following functor by taking the Jacobian and then its $p$-divisible group, the $p$-power torsion points of the Jacobian:
\begin{align*}
\mc{F}^{all} \colon \mr{Curve}_{G} & \to  \mr{AbVar}_G \to \mr{FG}_G \\
(\mf{X}, \rho) & \mapsto (\Jac(\mf{X}), \Jac(\rho)) \mapsto (\Jac(\mf{X})^\wedge_e, \Jac(\rho)[p^\infty]).
\end{align*}
\end{construction}

\begin{defn}[the functor]\label{def:functor}
For $\xi=(\mc X,\rho,\iota)\in\Def^\star_{(X,G)}(\Lambda)$ set
$$\widehat{\mc F}^1(\mc X,\rho,\iota):=\big(\,e_1\Pic^0(\mc X)[p^\infty],\ e_1(\iota_*)\,\big)\ \in\ \Def^\star_F(\Lambda),$$
where:
\begin{itemize}
\item $G$ acts on $\Pic^0(\mc X)$ through $\rho$ by pushforward, $g\mapsto\rho(g)_*=(\rho(g)^{-1})^*$ \cite[\S 8.1]{ray}, and $e_1\in\Z_p[G]$ is the idempotent of $\psi$ (Convention \ref{conv:characters}), integral because $p-1$ is a unit;
\item $\iota_*\colon\Pic^0(\mc X_s)\xrightarrow{\sim}\Pic^0(X)$ is induced by $\iota$; it commutes with $G$ because $\iota$ is $G$-equivariant, so it restricts to $e_1(\iota_*)\colon e_1\Pic^0(\mc X_s)[p^\infty]\xrightarrow{\sim}e_1\Pic^0(X)[p^\infty]=F$.
\end{itemize}
A morphism $\phi$ is sent to $e_1(\phi_*)$. The special fiber is $F$ by Theorem \ref{splittydooda}; the construction commutes with base change in $\Lambda$; and the value is a connected $p$-divisible group of dimension one because connectedness and dimension are insensitive to nilpotent thickenings. In the language of Construction \ref{fall}, $\widehat{\mc F}^1$ is the functor induced on deformation problems \cite[Def. 2.1]{ray} by $\mc F^{all}$, followed by the projection $\mc I^1$ of \cite[Lemma 8.1]{ray}.
\end{defn}

\begin{lemma}[the action on the special fiber]\label{lem:iotaaut}
For $g'\in G' \simeq \Aut(X)$ the pushforward $g'_*=(g'^{-1})^*$ on $\Pic^0(X)$ \cite[\S 8.1]{ray} commutes with $e_1$, and
$$\iota_{\mr{aut}}\colon G'\longrightarrow\Aut(F),\qquad g'\longmapsto e_1(g'_*)=g'_*|_F,$$
is a homomorphism. On the central subgroup it is the Teichm\"uller scalar,
$$\iota_{\mr{aut}}(g)=[\psi(g)]\ \in\ \mu_{p-1}(\Z_p)\subset\Z_p^\times=Z(\Aut(F))\qquad(g\in G).$$
\end{lemma}
\begin{proof}
$G'$ acts on $\Pic^0(X)$ by functoriality, and $g'_*h'_*=(g'h')_*$. Since $G$ is central in $G'$, $g'_*$ commutes with every $\rho_0(g)_*$ and hence with $e_1\in\Z_p[G]$ \cite[Lemma 8.1]{ray}; so $g'_*$ preserves $F=e_1\Pic^0(X)[p^\infty]$. For $g\in G$, $g\,e_1=\psi(g)\,e_1$ in $\Z_p[G]$ with $\psi(g)$ read in $\mu_{p-1}(\Z_p)$, so $g_*$ acts on $F$ as multiplication by that scalar; and the centre of $\Aut(F)=\OO_D^\times$ is $\Z_p^\times$.
\end{proof}

To rephrase using nomenclature of \cite[Lemma 7.3]{ray}, action (b) is the group of deck transformations of the rigidified problem over the unrigidified one, and so is action (c):
$$G'\ \xrightarrow{\ \sim\ }\ \Aut_{\Def_{(X,G)}}\big(\Def^\star_{(X,G)}\big),\quad g'\mapsto f_{g'},$$
$$\Aut(F)\ \xrightarrow{\ \sim\ }\ \Aut_{\Def_F}\big(\Def^\star_F\big),\quad a\mapsto f_a,$$
with $f_{g'}(\mc X,\rho,\iota)=(\mc X,\rho,g'\circ\iota)$ and $f_a(F',\alpha)=(F',a\circ\alpha)$. 

\begin{lemma}[equivariance: $\widehat{\mc F}^1$ lies over $\iota_{\mr{aut}}$]\label{lem:equiv}
The functor $\widehat{\mc F}^1$ carries deck transformations to deck transformations, and the homomorphism it induces between these two groups is the $\iota_{\mr{aut}}$ of Lemma \ref{lem:iotaaut}: for every $g'\in G'$ there is an equality of functors
$$\widehat{\mc F}^1\circ f_{g'}\;=\;f_{\iota_{\mr{aut}}(g')}\circ\widehat{\mc F}^1,
\qquad\qquad
\begin{tikzcd}
\Def^\star_{(X,G)}\arrow[r,"\widehat{\mc F}^1"]\arrow[d,"f_{g'}"'] & \Def^\star_F\arrow[d,"f_{\iota_{\mr{aut}}(g')}"]\\
\Def^\star_{(X,G)}\arrow[r,"\widehat{\mc F}^1"] & \Def^\star_F,
\end{tikzcd}$$
compatibly with base change in $\Lambda$ and covering the unrigidified functor $\Def_{(X,G)}\to\Def_F$.
\end{lemma}

\begin{remark} This is $G'$-equivariance in the sense of \cite[Def. 3.4]{ray}, once $\iota_{\mr{aut}}$ is known to be injective (Lemma \ref{lem:faithful}); on objects it reads $\widehat{\mc F}^1(g'\cdot\xi)=\iota_{\mr{aut}}(g')\cdot\widehat{\mc F}^1(\xi)$. \end{remark}

\begin{proof}
\emph{On objects.} The objects $\xi$ and $f_{g'}(\xi)$ have the same $(\mc X,\rho)$, hence the same idempotent and the same $p$-divisible group $e_1\Pic^0(\mc X)[p^\infty]$; only the rigidification differs. On the special fiber,
$$e_1\big((g'\circ\iota)_*\big)=e_1\big(g'_*\circ\iota_*\big)=e_1(g'_*)\circ e_1(\iota_*)=\iota_{\mr{aut}}(g')\circ e_1(\iota_*),$$
the first equality by functoriality of $\Pic^0$, the second because both $g'_*$ and $\iota_*$ commute with $G$ (Lemma \ref{lem:iotaaut}; $\iota$ is equivariant). That is $f_{\iota_{\mr{aut}}(g')}$ applied to $\widehat{\mc F}^1(\xi)$. Every map in this computation is a map of special fibers: nothing has been asked of $g'$ on $\mc X$.

\emph{On morphisms, and over the unrigidified problems.} Both $f_{g'}$ and $f_a$ are the identity on morphisms \cite[Def. 2.2]{ray}, and $\widehat{\mc F}^1(\phi)=e_1(\phi_*)$ does not involve the rigidifications, so the two composites agree on morphisms as well; the equality is one of functors, not merely of isomorphism classes. For the same reason $\widehat{\mc F}^1$ commutes strictly with forgetting rigidifications, so in the notation of \cite[Def. 7.2]{ray} it sends the pair $(f_{g'},\mathrm{id})$ to the pair $(f_{\iota_{\mr{aut}}(g')},\mathrm{id})$. That $g'\mapsto\iota_{\mr{aut}}(g')$ is a homomorphism is Lemma \ref{lem:iotaaut}.
\end{proof}

This lemma will only be used later in the cohomology section, but we will prove it here where the notation is best established.
\begin{lemma}[the central $G$ acts trivially on both sides]\label{lem:Gtrivial}

For $g\in G$ and $\xi:=(\mc X,\rho,\iota)$, the automorphism $\rho(g)^{-1}$ of $\mc X$ is an isomorphism $\xi\xrightarrow{\sim}g\cdot\xi$ in $\Def^\star_{(X,G)}(\Lambda)$, and $\widehat{\mc F}^1$ carries it to the isomorphism $[\psi(g)]^{-1}\colon\widehat{\mc F}^1(\xi)\xrightarrow{\sim}\iota_{\mr{aut}}(g)\cdot\widehat{\mc F}^1(\xi)$. 
\end{lemma}

\begin{remark} In other words, $G$ acts trivially on isomorphism classes on both sides, the $G'$-actions on the two pro-representing rings factor through $G'/G\cong\F_p\rtimes\F_p^\times$, and $\widehat{\mc F}^1$ is compatible with these trivializations.
\end{remark}

\begin{proof}
$\rho(g)^{-1}$ is $G$-equivariant ($G$ is abelian) and $(\rho_0(g)\circ\iota)\circ\rho(g)^{-1}_s=\rho_0(g)\circ\rho_0(g)^{-1}\circ\iota=\iota$, which is the condition to be a morphism $\xi\to g\cdot\xi$. On the target, $[\psi(g)]\in\Z_p^\times$ is an automorphism of every $p$-divisible group $F'$, and $(\iota_{\mr{aut}}(g)\circ\alpha)\circ[\psi(g)]^{-1}_s=\alpha$ by Lemma \ref{lem:iotaaut}. Finally $e_1\big((\rho(g)^{-1})_*\big)=[\psi(g)]^{-1}$ on $e_1\Pic^0(\mc X)[p^\infty]$, by the definition of $e_1$.
\end{proof}

\subsubsection{Faithfulness on isotropy}

Intuitively, $\widehat{\mc{F}}^1$ extracts the formal group that sees the wild $C_p$-action, and we now check that it sees \textit{all} of $G'$, not a quotient. We do not need this for equivariance of $\widehat{\mc{F}}$. We have yet to show that $\iota_{\mr{aut}}$ is injective. A morphism of torsors lies over a homomorphism of groups whether or not that homomorphism is injective, and the proof of Lemma \ref{lem:equiv} used nothing about $\iota_{\mr{aut}}$ beyond its definition: the square commutes for any homomorphism. But we need this injectivity afterward for two reasons.

First, the injectivity is what makes the $G'$ on the Lubin--Tate side a \emph{subgroup} of $\G_{p-1}$. Forgetting rigidifications, the unrigidified functor
$$\Def_{(X,G)}\simeq\Def^\star_{(X,G)}\stackyq G'\longrightarrow\Def_F\simeq\Def^\star_F\stackyq\Aut(F)$$ induces $\iota_{\mr{aut}}$ on the automorphism group of the closed point, so $\iota_{\mr{aut}}$ is injective exactly when this functor is faithful there. If $\iota_{\mr{aut}}$ had a kernel $K$, \cite[Lemma 3.5]{ray} would still produce an equivalence of quotient stacks, but the $G'$ acting on $\Def^\star_F$ would act through $G'/K$: the identification with the maximal finite subgroup \cite[\S 6.3.1]{ray} would fail, and Theorem \ref{theoremd}, with its Tate companion, would compute the cohomology of the wrong group.

Second, it is not already known. The automorphism group of a curve of genus at least two acts faithfully on its Jacobian, but $\iota_{\mr{aut}}$ is that action followed by projection to the one summand $e_1$, and projection to a summand can kill elements: on the block $e_j$ the central $G$ acts through $\chi^j$, which has a kernel as soon as $\gcd(j,p-1)>1$. Nor can injectivity be read off the deformation space: by Lemma \ref{lem:Gtrivial} the map $G'\to\Aut(R)$ has kernel $G$ (exactly $G$, by Lemma \ref{lem:branchfaithful} and Theorem \ref{actup}). The central $G$ survives only as a deck transformation --- the pair $(f_g,\mathrm{id})$ of \cite[Lemma 7.3]{ray} is not equivalent to the identity, because the comparison $2$-morphism $\rho(g)^{-1}$ is a nontrivial automorphism once the rigidification is forgotten --- and as the character $\psi$ on $\omega_F$. So injectivity on $G$ has to come from the Hodge line, and injectivity on $\sigma$ from the Dieudonn\'e module; these are the two steps of the proof.

\begin{lemma}[faithfulness on isotropy: injectivity of $\iota_{\mr{aut}}$]\label{lem:faithful}
The homomorphism $\iota_{\mr{aut}}\colon G' \to \Aut(F)$ is injective; equivalently, the unrigidified functor $\Def_{(X,G)}\to\Def_F$ is injective on the automorphism group of the closed point. The conjugation character of its tame part on its Sylow $p$-subgroup is surjective onto $\F_p^\times$.
\end{lemma}

\begin{proof}

Write an element of the kernel of $\iota_{\mr{aut}}$ as $\sigma^a\tau^b$.

We will show first that $\tau^b$ acting trivially implies that $(p-1)^2 \mid b$. Consider the action of the tame part on the Hodge line. An automorphism of $F$ acts on the invariant differentials $\omega_F=k\omega_1$. By \cite[\S 8.1]{ray}, $\iota_{\mr{aut}}(g') = e_1(g'_*)$ with $g'_* = (g'^{-1})^*$, and $g'_*$ acts on $\omega_{\Pic^0} = H^0(\Omega^1)$ as $g'^*$; so $\iota_{\mr{aut}}(\tau)$ acts on $\omega_F$ by the pullback action of $\tau$ on $\omega_1$, which is $\zeta^{-1}$ (Proposition \ref{diffblocks}(3)). A kernel element acts trivially on $\omega_F$, so $\zeta^{-b}=1$, and $(p-1)^2 \mid b$ and $\tau^b=1$.

Next, we show that $\sigma^a$ acting trivially implies that $p \mid a$. We show this by contradiction, by showing that $\sigma^a \neq 1$ in $\Aut(F)$ for $p \nmid a$. Consider the action of the wild part on the Dieudonne module $D(F)$, Dieudonn\'e theory over the perfect field $k$ is fully faithful, so it suffices to see that $\iota_{\mr{aut}}(\sigma^a)$ acts nontrivially on its Dieudonne module.

By Proposition \ref{prop:wild} the action of $\Z_p[C_p]$ on $F$ factors through $\Z_p[C_p]/(N)\cong\Z_p[\zeta_p]=\OO_L$, the norm $N=\sum_i\sigma^i$ vanishes and $\bb D(F)$ is a free $\OO_{\breve L}$-module of rank one. Under any identification $\Z_p[C_p]/(N)\cong\OO_L$, $\sigma^a-1$ maps to $\zeta_p^{\,a}-1$, which is a nonzero element of the domain $\OO_{\breve L}$ when $p\nmid a$, and a nonzero element of a domain acts injectively on a nonzero free module. Hence $\sigma^a$ is nontrivial on $\bb D(F)$, hence on $F$.

We have shown that a kernel element $\sigma^a\tau^b$ has $(p-1)^2 \mid b$ by the
tame step and then $p \mid a$ by the wild step; so $\iota_{\mr{aut}}$ is injective. Surjectivity of the conjugation character $\chic$ onto $\F_p^\times$ is a property of the abstract group $G'$ (Lemma \ref{autgroup}) and transports
along any injection.
\end{proof}

\begin{remark}
The proof locates the image in $\Aut(F)=\OO_D^\times$: the tame part maps isomorphically onto the subgroup of order $(p-1)^2$ of $\F_{p^{p-1}}^\times$ under the reduction $\OO_D^\times\to\F_{p^{p-1}}^\times$ (note $(p-1)^2\mid p^{p-1}-1$), its subgroup $G$ landing in the center $\Z_p^\times$ (Lemma \ref{lem:iotaaut}), and the wild part lies in the kernel $1+\mf m_D=S_{p-1}$; this is the subgroup of order $p$ of Ravenel's Theorem 14 \cite{Rav}, read off the curve.
\end{remark}

\subsection{Part (ii): the derivative is an isomorphism}

Everything now comes down to one linear map between two vector spaces of dimension $p-2$. Write $\mu:=d\widehat{\mc F}^1$ for the derivative at the closed point, a
map of tangent spaces $T\Def^\star_{(X,G)}\to T\Def^\star_F$.

\begin{lemma}[the derivative is cup product with the Hodge line]
\label{lem:derivative}
Under the identifications $T\Def^\star_{(X,G)}=H^1(X,T_X)^G$ and
$T\Def^\star_F=\Hom\big(k\omega_1,\ H^1(X,\OO_X)^{\psi}\big)$,
$$\mu(\vartheta)\;=\;\big(\omega_1\mapsto\vartheta\cup\omega_1\big).$$
\end{lemma}
\begin{proof}
For $\vartheta\in H^1(T_X)^G$ with first-order deformation
$\mc X_\vartheta/k[\epsilon]$, the crystalline nature of $H^1$ trivializes
$H^1_{dR}(\mc X_\vartheta/k[\epsilon])\cong H^1_{dR}(X)\otimes k[\epsilon]$;
by Grothendieck--Messing \cite{Mes72} the deformation of $F$ is the
displacement of its Hodge line inside this trivialization, and by
Katz--Oda \cite{KO} that displacement is cup product with the
Kodaira--Spencer class $\vartheta$. Apply $e_1$: the Hodge line of $F$ is
$k\omega_1$ and the Hodge quotient $\bD(F)/p\big/\omega_F$ is $H^1(\OO)^\psi$
(Corollary \ref{hsplitting}; \cite[\S 8.1]{ray}).
\end{proof}

\begin{lemma}[clean baby]\label{lem:injection}
$\mu$ is surjective. Let $D:=\mr{div}(\omega_1)$; then $\mu$ is
injective if and only if $H^0(\OO_D)^{\psi}=0$.  
\end{lemma}
\begin{proof}
The differential $\omega_1$ gives an exact sequence
$0\to T_X\xrightarrow{\ \omega_1\ }\OO_X\to\OO_D\to0$ of $G$-sheaves, whose
long exact sequence reads
$$H^0(\OO_X)\to H^0(\OO_D)\to H^1(T_X)\xrightarrow{\ \mu\ }H^1(\OO_X)\to
H^1(\OO_D)=0,$$
the last term vanishing because $\OO_D$ is a skyscraper. The map
$H^1(T_X)\to H^1(\OO_X)$ is cup product with $\omega_1$, i.e.\ $\mu$
after passing to $G$-eigenspaces; taking $G$-invariants of a sequence of
$G$-modules is exact since $p\nmid|G|$. So $\mu$ is surjective, and
$$H^0(\OO_X)^{\psi}\to H^0(\OO_D)^{\psi}\to H^1(T_X)^G\xrightarrow{\mu}
H^1(\OO_X)^{\psi}\to0$$
with $H^0(\OO_X)=k$ in the trivial character and $\psi\ne1$: injectivity of
$\mu$ is exactly $H^0(\OO_D)^\psi=0$.
\end{proof}

Now we will spend some time to understand our new object of interest to make sure it dies. 

\begin{lemma}[the divisor of $\omega_1$]\label{lem:divisor}
$D=\mr{div}(dy/x)=(p-3)\sum_{i\in\F_p}R_i$, where $R_i=(0,i)$ are the
tame ramification points of $p_y$ (Lemma \ref{rammymammy}); in particular
$D$ is supported on the $G$-fixed points, avoids the wild point $Q$, and
$\deg D=p(p-3)=2g-2$. At each $R_i$ the function $x$ is a uniformizer and a
$G$-eigenvector, with $\tau^{p-1}$ acting by $\eta$.
\end{lemma}
\begin{proof}
At $Q$: $x$ and $y$ have poles only at $Q$, of orders $p$ and $p-1$ (Lemma
\ref{rammymammy}), so $v_Q(x)=-p$ and, as $p\nmid p-1$, $v_Q(dy)=-p$; hence
$v_Q(dy/x)=0$. At $x=0$ the fiber is $\{R_i\}$, since $y^p-y=\prod(y-j)$
and only the factor $y-i$ vanishes at $R_i$; $p_y$ is totally ramified at
$R_i$ with index $p-1$, so $v_{R_i}(y-i)=p-1$ and $v_{R_i}(x)=1$. Writing
$y-i=ux^{p-1}$ with $u(R_i)=\prod_{j\ne i}(i-j)^{-1}\ne0$,
$dy=x^{p-2}\big((p-1)u+xu'\big)dx$ has $v_{R_i}(dy)=p-2$, hence
$v_{R_i}(dy/x)=p-3$. The degree check $p(p-3)=2g-2$ confirms no other
zeros. $G$ fixes each $R_i$ and $x\mapsto\eta x$.
\end{proof}

\begin{lemma}[the character count]\label{lem:count}
As a $G$-module,
$$H^0(\OO_D)\;=\;\bigoplus_{i\in\F_p}\OO_{X,R_i}/\mm_{R_i}^{\,p-3}
\;=\;\bigoplus_{i\in\F_p}\bigoplus_{a=0}^{p-4}k\cdot x^a,\qquad
(\tau^{p-1})^*x^a=\eta^ax^a,$$
so the characters occurring are $\chi^{-a}$ for $0\le a\le p-4$, each with
multiplicity $p$, and the two characters $\chi^{-(p-3)}=\chi^{2}$ and
$\chi^{-(p-2)}=\chi^{1}=\psi$ do not occur. In particular
$H^0(\OO_D)^{\psi}=0$.
\end{lemma}
\begin{proof}
By Lemma \ref{lem:divisor}, $\OO_D=\bigoplus_i\OO_{X,R_i}/\mm^{p-3}$ with
$x$ simultaneously a uniformizer and an eigenvector at each $R_i$, so no
averaging is needed and the action is diagonal on the nose. The $p-1$
characters of $G$ are $\chi^0,\dots,\chi^{p-2}$; the powers $x^a$,
$0\le a\le p-4$, realize $\chi^{0},\chi^{-1},\dots,\chi^{-(p-4)}$, which
misses exactly $\chi^{-(p-3)}$ and $\chi^{-(p-2)}$, and
$\chi^{-(p-2)}=\chi$ since $\chi^{p-1}=1$.
\end{proof}

\begin{proposition}[Kodaira--Spencer]\label{prop:KS}
$\mu\colon T\Def^\star_{(X,G)}\to T\Def^\star_F$ is an isomorphism, and
both sides have dimension $p-2$.
\end{proposition}
\begin{proof}
Surjectivity is Lemma \ref{lem:injection}; injectivity is Lemma
\ref{lem:injection} combined with Lemma \ref{lem:count}. The dimension of
the source is $p-2$ by Theorem \ref{actup} (or directly:
$h^1(T_X)^G=h^0(\Omega^{\otimes2})^{\chi^{-1}\text{-twist}}$, the
configuration count), and of the target by Lubin--Tate, $h-1=p-2$.
\end{proof}

\begin{remark}[three inputs, all from minimality]\label{rem:threeinputs}
The count in Lemma \ref{lem:count} used three facts, and none is a free
choice: (a) $\mr{div}(\omega_1)$ is supported on the $G$-fixed points of the
tame cover; (b) its multiplicity is $p-3=(2g-2)/p$, small enough to miss
two characters; (c) $\psi$ is \emph{extremal}: $\omega_1$ lives in the
block of the last character rather than a middle one. Each is a consequence
of Theorem \ref{thm:minimal}. Change the curve and any of them breaks: for
$p=5$ and the non-minimal $y^5-y=x^{28}+cx^{12}$ of genus $54$, one still
gets a surjection, but $2g-2=106$ is not divisible in the same way, the
eigendifferential's divisor is not concentrated with small multiplicity on
a single $G$-fixed orbit, and the vanishing $H^0(\OO_D)^\psi=0$ fails. For
injectivity to fail here, $\omega_1$ would have to vanish to order at least
$p-1$ at some $R_i$, i.e.\ $\deg D\ge p(p-1)$; but $\deg D=2g-2=p(p-3)$ is
forced. Injectivity of $\mu$ \emph{is} the minimal-genus theorem seen from
the formal group.
\end{remark}

\begin{remark}[what $\ker\mu$ would have been]
Deformation-theoretically, $\ker\mu$ consists of first-order equivariant
deformations of $X$ that the height-$(p-1)$ formal group cannot see:
directions in which the curve moves but the $\psi$-part of the Hodge
filtration does not. Its vanishing says the block sees every equivariant
direction: the statement that Serre--Tate, which we lost at height
$>2$, has been recovered on the nose for this one summand.
\end{remark}

\subsection{Part (iii): from tangent isomorphism to equivalence}

\begin{lemma}[both sides are smooth of the same dimension]\label{lem:smooth}
$\Def^\star_{(X,G)}$ and $\Def^\star_F$ are pro-represented by formally
smooth complete local $W(k)$-algebras of relative dimension $p-2$:
$R\simeq W(k)[[e_1,\dots,e_{p-2}]]$ (Theorem \ref{actup}) and
$A_{\mr{LT}}\simeq W(k)[[u_1,\dots,u_{p-2}]]$ (Lubin--Tate).
\end{lemma}
\begin{proof}
The source is Theorem \ref{actup}, whose input is the chain: $\Ext^2_G(\Omega_X, \OO_X) = 0$ for the tame action of $G$ on $X$ (the theorem preceding Theorem \ref{theoremb1}, via the local-to-global sequence and \cite[Lemma 7.23]{ray}(i)), so $\Def_{(X,G)}$ is formally smooth of dimension $\dim \Ext^1_G(\Omega_X, \OO_X) = \dim H^1(T_X)^G = 3g_{\PP^1} - 3 + (p+1) = p-2$ \cite[Thm.\ 7.24(iii)]{ray}, \cite[Cor.\ 7.40]{ray}; \cite[Thm.\ 7.41]{ray} then identifies it with $\Def_{(\A^1; p)}$ and \cite[Thm.\ 7.45]{ray} gives the coordinates $e_i$. The target is \cite{LT66}.
\end{proof}

Before stating the theorem we set up the graded problem, so that the ungraded and
graded forms can be stated and proved together: all of the content is in the
ungraded form, and the graded one is a torsor statement over it.

\subsection{The grading: the Hodge line is the invariant differential}\label{sec:gradedKS}

Topology keeps the grading, so we must too: it is the cohomology of the graded Lubin--Tate ring $R_*$ that converges to the local homotopy groups of spheres. We remind the reader of the graded formal group setting. An excellent discussion of invariant differentials of formal groups is given Section 4 in Goerss \cite{goerss}. 

We use the graded Lubin--Tate ring $R_*$ of \cite[Lemma 5.29]{ray}. 

The curve counterpart of the graded formal group moduli problem is \cite[Def. 7.13]{ray}; we first make its trivialization datum explicit, and identify it with the invariant differential of the block (Lemma \ref{lem:hodgeline}), so that no coordinate and no explicit universal curve is needed.

\begin{defn}[the graded curve problem, made explicit]\label{def:gradedcurve}
For $(\mc X,\rho,\iota)\in\Def^\star_{(X,G)}(\Lambda)$ let
\[
\mc L(\mc X):=\big(\pi_*\Omega^1_{\mc X/\Lambda}\big)^{\chi},
\]
the $\chi$-isotypic part for the pullback action of $G$ through $\rho$.
By Proposition \ref{diffblocks}(2) and cohomology-and-base-change it is a line
bundle on $\Spec\Lambda$, and $\iota$ identifies its special fibre with
$k\omega_1$. Set
\[
\Def^{\mr{Lie}}_{(X,G)}(\Lambda):=\{(\mc X,\rho,\iota,\omega):\ \omega
\text{ a generator of }\mc L(\mc X)\},
\]
so that the forgetful map $\Def^{\mr{Lie}}_{(X,G)}\to\Def^\star_{(X,G)}$ is
a $\G_m$-torsor (scaling $\omega$); the fibre over $\omega\equiv\omega_1$
is a torsor under $1+\mf m_\Lambda$, and the grading of \cite[Lemma 5.29]{ray} is the $\G_m$-weight. This is \cite[Def. 7.13]{ray} at $i=1$, with $\mf g^{\chi^1}:=\omega$. (For $i\ge2$ the
$\chi^i$-part has rank $i$ by Proposition \ref{diffblocks}(2), so
``trivialization'' would have to mean of its determinant; only $i=1$ is
used in this paper.)
\end{defn}

At height $\varphi(p^k)$, $k \geq 2$, the analogous statement, that the Hodge line of the block is the extremal differential $x^{p-3}dx$, is an Expectation in \cite{higherlevels}; the lemma below is its level-one case, where it is a theorem because the block is an integral direct summand.

\begin{lemma}[the Hodge line is the invariant differential, in families]
\label{lem:hodgeline}
There is a canonical isomorphism of line bundles on $\Def^\star_{(X,G)}$,
natural in $(\mc X,\rho,\iota)$ and $G'$-equivariant,
\[
\omega_{F_{\mc X}}\;\xrightarrow{\ \sim\ }\;\mc L(\mc X),\qquad
F_{\mc X}:=e_1\Pic^0(\mc X)[p^\infty]=\widehat{\mc F}^1(\mc X,\rho,\iota),
\]
where $\omega_{F}$ is the module of invariant differentials of the formal
group $F_{\mc X}$ over $\Lambda$. On the special fibre it is the
identification $\omega_F=k\omega_1$ of Corollary \ref{hsplitting}.
\end{lemma}
\begin{proof}
The Lie algebra of the Picard scheme is $\mr{Lie}\,\Pic^0_{\mc X/\Lambda}
=R^1\pi_*\OO_{\mc X}$, compatibly with base change (a deformation of the
trivial line bundle over $\Lambda[\epsilon]$ is a class in
$H^1(\mc X,\OO)\otimes\epsilon\Lambda$), so
$\omega_{\Pic^0}=(R^1\pi_*\OO_{\mc X})^\vee\cong\pi_*\Omega^1_{\mc X/\Lambda}$
by Grothendieck--Serre duality for the smooth proper curve $\mc X/\Lambda$,
again compatibly with base change and with automorphisms of $\mc X$. An
automorphism $g$ of $\mc X$ acts on $\Pic^0$ by $(g^{-1})^*$, hence on
$\mr{Lie}$ by the contragredient of $g^*$, hence on $\omega_{\Pic^0}$ by
$g^*$: this is why the $\chi$-part for the \emph{pullback} action appears,
matching Convention \ref{conv:characters}, and why the block's cotangent
line is $k\omega_1$ (dimension $1$) and not the $\chi^{-1}$-block
(dimension $p-2$). Passing to the $p$-divisible group and to its connected
part changes neither $\mr{Lie}$ nor $\omega$, and the idempotent $e_1$
commutes with everything since $G$ acts by automorphisms of $\mc X$; so
$\omega_{F_{\mc X}}=e_1\omega_{\Pic^0}=(\pi_*\Omega^1_{\mc X/\Lambda})^{\chi}$.
Naturality in $(\mc X,\rho,\iota)$ and $G'$-equivariance are the
functoriality of $\Pic^0$ and of duality; on the special fibre the
statement is Corollary \ref{hsplitting}.
\end{proof}

\begin{defn}[the graded functor]\label{def:gradedfunctor}
$\widehat{\mc F}^1_{\mf{Lie}}(\mc X,\rho,\iota,\omega):=
\big(F_{\mc X},\,e_1(\iota_*),\,\omega\big)$, with $\omega$ read as a
generator of $\omega_{F_{\mc X}}$ through Lemma \ref{lem:hodgeline}.
It covers $\widehat{\mc F}^1$ and is $\G_m$-equivariant: the square
\[
\begin{tikzcd}
\Def^{\mr{Lie}}_{(X,G)}\ar[r,"\widehat{\mc F}^1_{\mf{Lie}}"]\ar[d] &
\Def^{\mr{Lie}}_F\ar[d]\\
\Def^\star_{(X,G)}\ar[r,"\widehat{\mc F}^1"] & \Def^\star_F
\end{tikzcd}
\]
commutes, and the vertical maps are $\G_m$-torsors.
\end{defn}

\subsection{Statement and Proof of Equivalent Exchange}\label{sec:theoremC}

\begin{manualtheorem}{C}\label{equivalentexchange}\label{thm:LT}\stmt{(Equivalent Exchange)
\begin{enumerate}
\item[(a)] The functor
$$\widehat{\mc F}^1\colon\Def^\star_{(X,G)}\longrightarrow
\big(\Def^{FG^{\text{ dim }1}}_{\widehat{\mc F}^1(X,G)}\big)^\star$$
is a $G'$-equivariant equivalence of formal moduli problems, where
$\widehat{\mc F}^1(X,G)=F$ is a one-dimensional formal group of height
$p-1$, and the induced $G'$-action on the target is the canonical
Lubin--Tate action.
\item[(b)]\makeatletter\def\@currentlabel{C(b)}\makeatother\label{gradedequivalentexchange}
(Graded Equivalent Exchange) The functor
$$\widehat{\mc{F}}^1_{\mf{Lie}} \colon \Def_{(X, G)}^{\mr{Curve}_{G}^{\mr{Lie }\chi^1}} \longrightarrow \Def_{F}^{\mc{M}_{\mr{fg}_1}^{\mr{Lie} \simeq \mr{triv}}}$$
of Definition \ref{def:gradedfunctor} is a $G'$-equivariant equivalence of formal moduli problems, a map of $\G_m$-torsors over \emph{(a)}, where the $G'$-action on the target is the graded Lubin--Tate action.
\end{enumerate}}
\end{manualtheorem}

\begin{proof}
\emph{(a): equivalence.} Let $\varphi^\#\colon A_{\mr{LT}}\to R$ classify the
$p$-divisible group $\widehat{\mc F}^1(\mc U)$ of the universal curve. By
Proposition \ref{prop:KS}, $\varphi^\#$ induces an isomorphism
$\mm_A/(\mm_A^2+pA)\to\mm_R/(\mm_R^2+pR)$; since $p\in\varphi^\#(\mm_A)$,
Nakayama and completeness give surjectivity; a surjection between regular
local domains of the same dimension $p-1$ (Lemma \ref{lem:smooth}) is an
isomorphism. By Yoneda, $\widehat{\mc F}^1$ is an equivalence. \emph{(a : equivariance.} Lemma \ref{lem:equiv}: $\widehat{\mc F}^1$ lies over
the homomorphism $\iota_{\mr{aut}}\colon G'\to\Aut(F)$.

\emph{(a): identification of the action.} By Lemma \ref{lem:faithful},
$\iota_{\mr{aut}}$ embeds $G'$ into $\Aut(F)=\G_{p-1}$ with surjective
conjugation character. By Bujard \cite{bouj} \cite[Lemma 0.1]{ray}, the
maximal finite subgroups of $\G_{p-1}$ with nontrivial $p$-torsion form a
single conjugacy class, and any embedded copy of $G'$ with surjective $\chi$
is maximal; so $\iota_{\mr{aut}}(G')$ is the canonical subgroup up to
conjugation, and conjugation by an element of $\G_{p-1}$ acts on
$\Def^\star_F$ by an automorphism intertwining the two actions. Hence the
transported action is the Lubin--Tate action.

\emph{(a $\Rightarrow$ b): equivalence.} By Lemma \ref{lem:hodgeline} and Definition
\ref{def:gradedfunctor}, $\widehat{\mc F}^1_{\mf{Lie}}$ is well defined, covers
$\widehat{\mc F}^1$, and is a $\G_m$-equivariant morphism of $\G_m$-torsors.
Such a morphism covering an equivalence is an equivalence: both problems are
discrete, and over each $\xi\in\Def^\star_{(X,G)}(\Lambda)$ it is a
$\Lambda^\times$-equivariant map from the $\Lambda^\times$-torsor of generators
of $\mc L(\mc X)$ to the $\Lambda^\times$-torsor of generators of
$\omega_{F_{\mc X}}$, hence a bijection. No tangent space, smoothness or
dimension count is needed beyond those of (a).

\emph{(a $\Rightarrow$ b): equivariance and the action.} As in Convention
\ref{conv:threeactions}, $G'$ acts on the graded problems by changing the
rigidification and nothing else,
$$g'\cdot(\mc X,\rho,\iota,\omega)=(\mc X,\rho,g'\circ\iota,\omega),\qquad
a\cdot(F',\alpha,\phi)=(F',a\circ\alpha,\phi),$$
the second being the graded Lubin--Tate action \cite{ray}, \cite[Lemma 5.29]{ray}. Neither $\mc L(\mc X)$, nor $\omega_{F_{\mc X}}$, nor the
isomorphism between them in Lemma \ref{lem:hodgeline} involves $\iota$. So by
Lemma \ref{lem:equiv},
$\widehat{\mc F}^1_{\mf{Lie}}\big(g'\cdot(\xi,\omega)\big)=
\iota_{\mr{aut}}(g')\cdot\widehat{\mc F}^1_{\mf{Lie}}(\xi,\omega)$, with the same
$\iota_{\mr{aut}}$, and the identification of the action is the conjugacy
argument of (a). The action of $G'$ on the grading generator $u$ is not visible
at this level: it appears on the universal family, through the comparison
isomorphisms, and is computed in Corollary \ref{cor:uaction}.
\end{proof}

\begin{conv}
Let $u$ be the generator of the invariant differential over the universal deformation $R_*=R[u^{\pm}]$ corresponding to the universal differential $\omega\in\mc L(\mc X^{\mr{univ}})$. For $g'\in G'$ let $\gamma_{g'}$ be the automorphism of the base by which $g'$ acts, and $\theta_{g'}\colon\mc X^{\mr{univ}}\to\gamma_{g'}^*\mc X^{\mr{univ}}$ the unique isomorphism of deformations from $g'\cdot(\mc X^{\mr{univ}},\rho,\iota)$ to
$\gamma_{g'}^*(\mc X^{\mr{univ}},\rho,\iota)$ \cite[Lemma 7.3]{ray}; on the special fiber $\theta_{g'}$ is $g'$ itself. This, and not the lift $\rho$ (which exists
only for $G$), is the sense in which $G'$ acts on the universal curve.
\end{conv}
\begin{cor}[the action on the grading, canonically]\label{cor:uaction}
The action of $G'$ on the universal curve induces an action of $G'$ the generator $u$ of the invariant differential of the universal formal group law:
$g'\cdot u=c(g')\,u$ with
$c(g')\in R^\times$ the unit defined by $\theta_{g'}^*\omega=c(g')\,\omega$,
by the naturality in Lemma \ref{lem:hodgeline}, and its reduction is the
character of $G'$ on the Hodge line:
\[
c(\sigma)\equiv1,\qquad c(\tau)\equiv\zeta^{-1}\pmod{\mf m_R}
\]. 
\end{cor}

\begin{proof} This is a restatement of (Proposition \ref{diffblocks}(3)).
\end{proof}

\begin{remark} This is the $g'(0)$ of Gorbounov--Mahowald
\cite{gm}, obtained without a coordinate: it is the ratio of the pulled-back invariant differential of the universal curve to itself.
\end{remark}

\subsection{Corollary C.2, Golden Fish: Identification of the Action}\label{sec:identify}
By \cite[\S 6.3]{ray}, any subgroup of $\OO_D^\times$ abstractly isomorphic to $G' = C_p \rtimes_{\chic} C_{(p-1)^2}$ with $\chic$ surjective is conjugate to the standard one $j(G') = \langle \zeta_p \rangle \rtimes \mu_{(p-1)^2}$, where $\pi = \omega_0\Pi$ satisfies $\pi^{p-1} = -p$ \cite[Lemma 6.3]{ray}, \cite[\S 6.3.1]{ray}. Applied to $\iota_{\mr{aut}}$:

\begin{corollary}\label{cor:identify}
There exist $u\in\OO_D^\times$ and $\alpha\in\Aut(G')$ with
$u^{-1}\iota_{\mathrm{aut}}(g')u=j(\alpha(g'))$; consequently, by \cite[Lemma 6.2]{ray}, the twist action of
$\iota_{\mathrm{aut}}(G')$ on $\Def^\star_F$ is isomorphic, via an automorphism of
$\Def^\star_F$, to the canonical Lubin--Tate action of $G'$ restricted along $\alpha$.
\end{corollary}

\begin{manualcorollary}{C.2}\label{goldenfish}\stmt{(golden fish)
Let $\mc U$ be the universal object over $\Def^\star_{(X,G)}$ and $R$ the
ring representing $\Def^\star_{(X,G)}$. Then the formal group
$\widehat{\mc F}^1(\mc U)$ over $R$ is a universal deformation of $F$; in
particular it is a universal one-dimensional formal group of height $p-1$
over $W(k)$.}
\end{manualcorollary}
\begin{proof}
Theorem \ref{thm:LT}(a) identifies $R$ with $A_{\mr{LT}}$ so that
$\widehat{\mc F}^1(\mc U)$ is the universal deformation.
\end{proof}

\begin{remark} Other grading conventions in the literature (Hopkins--Hill, Gorbounov--Mahowald, Hovey--Strickland) are compared in \cite[\S 5.1]{ray}; Corollary \ref{cor:uaction} is Gorbounov--Mahowald's $g'(0)$ without a coordinate. \end{remark}

\subsection{The Method in the Context of the Series}\label{rem:generalizes}

\subsubsection*{What the proof used}
The proof of Theorem \ref{equivalentexchange} touches the equation $y^p-y=x^{p-1}$ in only three places:
\begin{enumerate}
\item the decomposition of $H^0(X,\Omega^1)$ into isotypic pieces for the central $G$ (Proposition \ref{diffblocks});
\item the divisor of the extremal differential $\omega_1$ (Lemma \ref{lem:divisor});
\item the vanishing $N=0$ on the Jacobian (Proposition \ref{prop:wild}).
\end{enumerate}
All three follow from minimality (Theorem \ref{thm:minimal}) and from $X/G\cong\PP^1$. So the proof needs the minimal curve, not its equation.

\subsubsection*{What happens for $k\ge2$}
The third paper of this series \cite{higherlevels} takes the minimal curve at height $\varphi(p^k)=p^{k-1}(p-1)$, $k\ge2$. One ingredient carries over and one does not.

\begin{itemize}
\item \emph{The derivative carries over.} The derivative of the comparison map is still Kodaira--Spencer cupped with the Hodge line of the block \cite[Thm.\ A$'$]{higherlevels}, and the tangent space of Lubin--Tate space is still the induced module $\Ind_{\Z/(p-1)^2}^{G'}\theta$ of the branch orbit, modulo its gauge \cite[Thm.\ B]{higherlevels}.
\item \emph{The moduli do not.} At height $p-1$ the branch points move, and $\Def_{(X,G)}$ is a configuration space. For $k\ge2$ the branch orbit $W_k(\F_p)$ cannot move. Deformations then come from surgeries at the fixed branch points, and there is no tower over $\Def_{(X,G)}$ \cite[Prop.\ 2.10]{higherlevels}. The comparison map is defined only on the \emph{horizontal} surgeries, those along which the Hodge line of the block stays inside the block, and its derivative has a kernel, the \emph{ghosts}. Conjecture I of \cite{higherlevels} says that Lubin--Tate space is the horizontal surgeries modulo the ghosts; it is checked on tangents at $(p,k)=(3,2)$ and $(5,2)$.
\end{itemize}

\subsubsection*{Reading this back at $k=1$}
At height $p-1$ every deformation is horizontal and there are no ghosts, so the conjecture becomes a theorem: it is Theorem \ref{equivalentexchange}. In the language of periods, the map
$$\Def_{(X,G)}\longrightarrow\PP\big(\bb{D}(F)[1/p]\big),\qquad (\mc X,\rho,\iota)\longmapsto[\omega_{F_{\mc X}}],$$
sending a deformation to the Hodge line of its block, is a local isomorphism onto the Gross--Hopkins period domain, and $\widehat{\mc F}^1$ is its lift through the Gross--Hopkins map.

\part{Putting the Model to Work}\label{sec:cohomology}

\section*{Setup and Theorem D, Up Shot}
Everything so far has gone into building the model and proving it correct; now we put the model to work --- the hour when the work is done, the light goes gold, and everything we spent the day setting up starts paying us back in cohomology. First, the notation for the calculation to come.

\begin{notation} \label{lilguy} $\bb{W} := W(\F_{p^{p-1}})$, and $\lambda$ and $\overline{\rho}$ are the standard representation of $\Sigma_p$ and its restriction to $C_p \rtimes C_{p-1} \simeq \F_p \rtimes \F_p^\times$, with generators in degree $-2$ \cite[Defs.\ 7.44, 7.46]{ray}. \end{notation}

\begin{manualtheorem}{D} \label{theoremd} \stmt{Let $p$ be any prime and let $\bb{W} := \bb{W}(\F_{p^{p-1}})$.
Let $\bar{\rho}$ be the graded $C_p \rtimes C_{p-1}$-representation of the Notation above, where the generators of $\bb{W}[C_p]$ have degree $-2$.

Then there is a $C_p \rtimes C_{(p-1)^2}$-equivariant $\bb{W}[C_p]$-module equivariant isomorphism
$$ \pi_*(E_{(p-1)}) \simeq (\Sym(\bar{\rho})[d^{-1}])^\wedge_I.$$

If $x_0$ denotes $[1] \in \bar{\rho}$ and $\sigma$ generates $C_p$, then $d = \prod_{i=0}^{p-1}\sigma^i(x_0) = x_0 \cdots x_{p-1}$. The ideal $I$ is the $(C_p \rtimes C_{(p-1)^2})$-equivariant ideal $I := (p, \delta_i)$ generated by $p$ and the displacements $\delta_i := (1-\sigma^i)(x_0)$. It is $G'$-stable: $I = (p) + \mf{a}\cdot x_0$ for $\mf{a} = \ker(\bb{W}[C_p] \to \bb{W})$ the augmentation ideal, $\tau$ normalizes $C_p$ hence preserves $\mf{a}$, and $\tau(x_0) = \eta x_0$ with $\eta$ a unit, so $\tau(I) = I$.}
\end{manualtheorem} 

\begin{remark} One would like to complete with respect to $(p, x_i)$, but that isn't $C_p$-equivariant, which is why we instead complete with respect to this odd-looking ideal $I$. In other words, we complete with respect to the ideal $(p, x_0, \cdots, x_{p-1})^{C_p}.$ \end{remark}

\begin{proof} The equivalence as $C_p \rtimes C_{p-1}$ representations follows immediately by combining Theorem \ref{equivalentexchange} which states that $\widetilde{\mc{F}}^1$ is an $\Aut(X)$-equivariant equivalence, where $\Aut(X) \simeq C_p \rtimes C_{(p-1)^2}$, with Theorem \ref{actup} which describes the $\Aut(X)/C_{p-1}$ action on $\Def_{(X, G)}$. The action of the full $C_{(p-1)^2}$ on the graded ring (which Theorem \ref{actup} sees only through its quotient $C_{p-1}$, because the central $G$ acts trivially on $\Def_{(X,G)}$) is pinned down, without a coordinate, by Lemma \ref{lem:gradedaction} below: it is the configuration action on the degree-$0$ part together with the character $c(\tau)\equiv\zeta^{-1}$ on the grading generator (Corollary \ref{cor:uaction}).
\end{proof}

\section{Theorem D.1: $C_p$-Cohomology}
We start with the wild part alone, since it is the only part that can taunt us. We use the notation $\rho = \bb{W}[C_p]$, $\overline{\rho}$, $\Lambda = \Sym(\overline{\rho})$, $A = \Sym(\rho)$, $s_1$ and $d$ of \cite[\S 8.3]{ray}, and the Tate cohomology computations $\hat H^*(C_p, A) \simeq \Z/p[b^{\pm}, d]$ and $\hat H^*(C_p, \Lambda)[d^{-1}] \simeq \Z/p[b^{\pm1}, c, d^{\pm 1}]/(c^2)$ \cite[Lemma 8.2]{ray}, \cite[Lemma 8.3]{ray}.

\begin{manualtheorem}{D.1} \label{theoremd1} Let $R$ be the graded ring 
$$R := \bb{W}/p[b^{\pm}, d^{\pm}, c]/(c^2).$$
The Tate cohomology of the $C_p$-module $(E_{p-1})_*$ is: 
$$R \simeq \hat{H}^*(C_p, (E_{p-1})_*).$$
where $|b| = (2, 0)$, $|c| = (1, -2)$, and $|d| = (0, -2p)$.
\end{manualtheorem}

\begin{proof} 
This follows from \cite[Lemma 8.3]{ray} combined with Lemma \ref{idealhop}.
\end{proof}

\subsection{Analyzing the $C_{(p-1)^2}$ Action}
\label{Cp-1section}
The tame group is the bouncer at this party: it does not create any classes, it only decides which of the classes $C_p$ produced are allowed to stay.

\begin{lemma}[the $C_{(p-1)^2}$-action on $R_*$, without coordinates]
\label{lem:gradedaction}
Write $R_*=R[u^{\pm1}]$ for the graded Lubin--Tate ring, $R$ the degree-$0$
part, as in \cite[Lemma 5.29]{ray}. Then:
\begin{enumerate}
\item The $G'$-action on $R$ factors through $G'/G=C_p\rtimes C_{p-1}$ and is
the configuration action of Theorem \ref{actup}: $\sigma$ permutes the
branch coordinates cyclically and $\tau$ acts on the affine coordinate of
the branch line $X/G\cong\PP^1$ by the scalar $\vartheta(\tau)$, the
Teichm\"uller lift of $\bar\chi(\tau)\in\F_p^\times$.
\item The grading generator can be chosen so that
$$\sigma\cdot u=u,\qquad \tau\cdot u=\zeta^{-1}u,$$
with $\zeta$ the primitive $(p-1)^2$-th root of unity of Proposition
\ref{diffblocks}(3) (the eigenvalue of $\tau$ on the Hodge line
$k\omega_1$).
\item Consequently $G'$ acts on $R_*$ by $g'\cdot(f\,u^n)=(g'\cdot f)\,c(g')^nu^n$
with $c(\sigma)=1$, $c(\tau)=\zeta^{-1}$; in particular the scalar $\eta$ by
which $\tau$ acts on a degree-$(-2)$ generator $y_i=u\cdot e_i$ of
$\overline{\rho}$ is
$$\eta\;=\;\zeta^{-1}\,\vartheta(\tau),$$
a primitive $(p-1)^2$-th root of unity. (This $\eta$ is the scalar of
this Part, not the generator $\eta=\zeta^{p-1}$ of $G$ in Convention
\ref{conv:characters}.)
\item Explicitly: the branch-line coordinate is the $y$ of Lemma
\ref{autgroup}, on which $\tau^*y=\zeta^{p-1}y$, so
$\vartheta(\tau)=\zeta^{p-1}$, $a=\bar\zeta^{\,p-1}$ (as
$\sigma\colon y\mapsto y+1$ gives $\tau\sigma\tau^{-1}=\sigma^{\bar\vartheta(\tau)}$), and
$$\eta=\zeta^{p-2},\qquad \eta^{p-1}=\zeta^{-(p-1)}\equiv a^{-1}\pmod p,$$
which is the scalar by which $\tau$ acts on the periodicity class $b$
(Lemma \ref{C_p-1actiononA}). This is the consistency check between the
grading character and the cohomological action: $\tau(b)=\eta^{p-1}b$.
\end{enumerate}
\end{lemma}
\begin{proof}
(1) is Theorem \ref{actup} together with the observation that the central
$G$ acts trivially on $\Def_{(X,G)}$ (Lemma \ref{lem:Gtrivial}). (2) By
Corollary \ref{cor:uaction}, for any generator $u$ of $\omega_F$ over $R$
one has $g'\cdot u=c(g')u$ with $c(g')\in R^\times$, $c(\tau)\equiv\zeta^{-1}$
and $c(\sigma)\equiv1$ mod $\mf m_R$; $g'\mapsto c(g')$ is a $1$-cocycle of
$G'$ with values in $R^\times$ for the action of $G'$ on $R$. On the
prime-to-$p$ part it is a coboundary by averaging:
$u':=\frac1{(p-1)^2}\sum_{j}\zeta^{\,j}\,\tau^j\cdot u$ satisfies
$\tau\cdot u'=\zeta^{-1}u'$ and $u'\equiv u$ mod $\mf m_R$, hence is a
generator. On $\langle\sigma\rangle$ the cocycle has values in
$1+\mf m_R$ and is a coboundary layer by layer: on each
$\mf m_R^n/\mf m_R^{n+1}$, a $k$-vector space on which $\sigma$ acts through
the configuration action, a $1$-cocycle of $C_p$ that is the reduction of
a cocycle trivial on the previous layer is killed by modifying $u'$ by a
unit $\equiv1$ mod $\mf m_R^n$, and the modification can be taken
$\tau$-equivariant by averaging again. The normalizations are compatible
because $\tau\sigma\tau^{-1}=\sigma^a$ preserves the $\zeta^{-1}$-eigenline.
(3) is (1) and (2) combined: $\tau$ acts on $y_i=u\,e_i$ by $\zeta^{-1}$ on
$u$ and by $\vartheta(\tau)$ on $e_i$, up to the permutation of indices.
Since $\zeta^{-1}$ has exact order $(p-1)^2$ and $\vartheta(\tau)^{p-1}=1$,
the product has exact order $(p-1)^2$. (4) The branch line is $X/G=\PP^1_y$
with $\sigma\colon y\mapsto y+1$ and $\tau^*y=\zeta^{p-1}y$ (Lemma
\ref{autgroup}); then $\tau\sigma\tau^{-1}\colon y\mapsto y+\zeta^{p-1}$,
so $a\equiv\zeta^{p-1}$. Hence $\eta=\zeta^{-1}\zeta^{p-1}=\zeta^{p-2}$,
which is primitive as $\gcd(p-2,(p-1)^2)=1$, and
$\eta^{p-1}=\zeta^{(p-2)(p-1)}=\zeta^{(p-1)^2}\zeta^{-(p-1)}=\zeta^{-(p-1)}\equiv a^{-1}$.
The action of $\tau$ on the periodicity class $b\in\hat H^2(C_p,A)$ is
the contragredient of its action on $H_1(C_p,\Z)=C_p$, which is
multiplication by $a$; so $\tau(b)=a^{-1}b=\eta^{p-1}b$.
\end{proof}

\begin{remark}
The identity $\eta=\zeta^{-1}\vartheta(\tau)$ is the one place in this Part where the geometry of the curve enters. Both factors are characters of $\tau$ on cotangent lines, of the block ($\zeta^{-1}$, Proposition \ref{diffblocks}(3)) and of the branch line ($\vartheta(\tau)$, Lemma \ref{lem:char}), and the relation between them is the ramification row of the table in Subsection \ref{oneidentity}, $\chi=\theta^{-m}$ with $m=p-1$. At $p=3$ the normalization $\eta^{p-1}\equiv a^{-1}$ cannot be told apart from $\eta^{p-1}\equiv a$; at $p=5$ it can, and it is the relation $\tau(b)=\eta^{p-1}b$ that makes $\beta=d^{-(p-1)}b$ invariant in Corollary \ref{badbitch_meat}.
\end{remark}

In the basis $x_i=\sigma^i(x_0)$ of $\overline\rho$, Lemma \ref{lem:gradedaction}(3) says $\tau(x_i)=\eta\,x_{ia}$ for $a$ the primitive root with $\tau\sigma\tau^{-1}=\sigma^a$; the action on $\rho$ and $A$ is the same formula, extended along $\overline\rho\subset\rho$. For the full $C_{(p-1)^2}$-action in other coordinates see \cite[Prop.\ 3.2]{gorbounovsymonds} and \cite[Prop.\ 5.2]{nave}. Readers who enjoy watching every basis vector get moved by hand, including the extension of the defining sequence $s_1A\to A\to\Lambda$ to $C_p\rtimes C_{(p-1)^2}$ and the check that it generates, will find all of it in the earlier version of this paper, arXiv:2509.23428v2. Here we let the grading do the moving.

\section{Theorem E: $C_p \rtimes C_{(p-1)^2}$-Cohomology and Known Classes}
Now the two groups dance together. Let us examine the invariants of the $C_{(p-1)^2}$-action on $\hat{H}^*(C_p, A)$, starting with a reduction of complexity.

We use \cite[Lemma 8.4]{ray}: since $(p-1)^2$ is a unit, $H^*(G',M)\simeq H^*(C_p,M)^{C_{(p-1)^2}}$, and likewise for Tate cohomology. So everything comes down to how $\tau$ acts on three classes, and each gives up its secret in a line or two.

\begin{lemma} \label{C_p-1actiononA}\label{action} The generator $\tau\in C_{(p-1)^2}$ acts on
$$\hat H^*(C_p,\Lambda)[d^{-1}]\simeq\Z/p[b^{\pm1},c,d^{\pm1}]/(c^2)$$
by
$$\tau(b)=\eta^{p-1}b,\qquad \tau(d)=\eta^{p}d,\qquad \tau(c)=\eta^{p}c,$$
with $|b|=(2,0)$, $|c|=(1,-2)$, $|d|=(0,-2p)$, and $\eta$ as in Lemma \ref{lem:gradedaction}.
\end{lemma}

\begin{proof}
\emph{The class $b$.} It generates $\hat H^2(C_p,\Z)\simeq H^1(C_p,\Z/p)=\Hom(H_1(C_p,\Z),\Z/p)$ via the Bockstein. Since $\tau$ acts on $H_1(C_p,\Z)=C_p$ by $a$, it acts on the dual by the contragredient $a^{-1}$, so $\tau(b)=a^{-1}b=\eta^{p-1}b$ by Lemma \ref{lem:gradedaction}(4).

\emph{The class $d$.} It is the norm $d=x_0x_1\cdots x_{p-1}$, and $\tau(x_i)=\eta\,x_{ia}$ permutes the factors up to $\eta$, so $\tau(d)=\eta^pd$.

\emph{The class $c$.} By \cite[Lemma 8.2]{ray}, $c$ maps equivariantly to $\widetilde b=s_1b\in\hat H^2(C_p,s_1A)$, and $\tau(s_1)=\eta s_1$, so $\tau(\widetilde b)=\eta\cdot\eta^{p-1}\widetilde b=\eta^p\widetilde b$. Hence $\tau(c)=\eta^pc$.
\end{proof}

\begin{cor} \stmt{(meat) \label{badbitch_meat} Let $R$ be the graded ring 
$$R := \bb{W}/p[\alpha, \beta^{\pm}, \Delta^{\pm}]/(\alpha^2),$$
\noindent where $\alpha := d^{-1}c$, $\beta := d^{-(p-1)}b$, and $\Delta := d^{-((p-1)^2)}$. The Tate cohomology of the $C_p \rtimes C_{(p-1)^2}$-module $\Lambda[d^{-1}] = \mr{Sym}(\overline{\rho}[d^{-1}])$ is: 
$$R \simeq \hat{H}^*(C_p \rtimes C_{(p-1)^2}, \Lambda[d^{-1}]).$$
where $|\alpha| = (1, 2(p-1))$, $|\beta| = (2, 2p(p-1))$, and $|\Delta| = (0, 2p(p-1)^2).$}
\end{cor}

\begin{proof} 
The $C_{(p-1)^2}$-action on $\hat H^*(C_p, \Lambda)[d^{-1}]$ is described by Lemma \ref{action}. The fixed points of this action are precisely $\alpha := d^{-1}c$, $\beta := d^{-(p-1)}b$, and $\Delta := d^{-(p-1)^2}$. The fact that this is the Tate cohomology of $\Lambda$ follows from part (ii) \cite[Lemma 8.4]{ray}. The answer agrees with Symonds' \cite[Thm.\ 1.1]{symonds} and with the Hopkins--Miller computation recorded by Nave \cite{nave}, both reached by group-cohomological means; what the model adds is the names of the classes in its coordinates, which the next remark uses.
\end{proof}

Here the computation shakes hands with history.

\begin{remark}[Ravenel's detection, in these coordinates]\label{rem:ravenel}
Ravenel \cite{Rav} proved that the odd-primary Arf invariant classes
$b_i\in\Ext^{2,2(p-1)p^{i+1}}_{BP_*BP}(BP_*,BP_*)$ are nonzero by restricting
to an order-$p$ subgroup of $\G_{p-1}$ and showing $b_i\mapsto-c^{p^{i+1}}b$
in $H^*(\Z/p;\F_{p^{p-1}})=E(h)\otimes P(b)$, $c\in\F_{p^{p-1}}^\times$; his
$b$ is the class $b$ here and his $h$ is our $c$. The restriction factors
through $E_{p-1}^{h\mathbb G}\to E_{p-1}^{hG'}\to E_{p-1}^{hC_p}$, and
Corollary \ref{badbitch_meat} identifies the middle term: in cohomological
degree $2$ the ring $R$ is spanned by the classes $\beta\Delta^m$, of
internal degree $2p(p-1)+2p(p-1)^2m$, and
\[
2p(p-1)\big(1+(p-1)m\big)=2(p-1)p^{\,i+1}\quad\Longleftrightarrow\quad
m=\frac{p^i-1}{p-1}.
\]
Hence
\[
b_i\;\longmapsto\;u_i\,\beta\,\Delta^{(p^i-1)/(p-1)},\qquad
u_i\in(\bb W/p)^\times,
\]
the unit being nonzero because the further restriction to $C_p$ is
Ravenel's $-c^{p^{i+1}}b\ne0$. The exponent $(p^i-1)/(p-1)$ is $a_i/p$ for
the modulus $\ker b_0^{a_i}$, $a_i=p(p^i-1)/(p-1)$, of his inductive
differential $d_{2p-1}b_{i+1}\equiv h_0b_i^p$: the power of $\Delta$
carrying $b_i$ is the power of $b_0$ modulo which the differential is
computed. The scalar $c$, which he shows generates $\F_{p^{p-1}}$ and
which is why the detecting map needs that field, is on our side the
residue of the tame character $\zeta$ of Proposition \ref{diffblocks}(3), which
lives in $\F_{p^{p-1}}$ for the same reason. 
\end{remark}

One last piece of business: the completion at $I$. Completions are the fine print of chromatic homotopy theory, and fine print is where trouble usually hides. We must fastidiously show that it hides nothing here, and the next two lemmas do.

\begin{lemma}[completion is base change along the invariants]\label{lem:basechange}
Let $G$ be a finite group acting by $\bb{W}$-algebra automorphisms on a ring $\Lambda$ that is finitely generated as an algebra over a noetherian ring $\bb{W}$, let $R := \Lambda^G$, and let $I \subseteq \Lambda$ be a finitely generated $G$-stable ideal with $I_R := I \cap R$. Then $R$ is noetherian, $\Lambda$ is a finitely generated $R$-module, the $I$-adic and $I_R\Lambda$-adic topologies on $\Lambda$ agree, and for every $n$
$$\Lambda^\wedge_I \;\simeq\; \Lambda \otimes_R R^\wedge_{I_R}, \qquad
\hat H^n(G, \Lambda^\wedge_I) \;\simeq\; \hat H^n(G,\Lambda)\otimes_R R^\wedge_{I_R} \;=\; \hat H^n(G,\Lambda)^\wedge_{I_R}.$$
In particular, if $I_R$ acts by zero on $\hat H^*(G,\Lambda)$, then $\hat H^*(G,\Lambda^\wedge_I) = \hat H^*(G,\Lambda)$.
\end{lemma}
\begin{proof}
Every $a \in \Lambda$ is a root of the monic polynomial $\prod_{g \in G}(T - ga)$, whose coefficients are $G$-invariant, so $\Lambda$ is integral over $R$; being a finitely generated $\bb{W}$-algebra it is then a finitely generated $R$-module, and $R$ is a finitely generated $\bb{W}$-algebra (Noether), hence noetherian. If $a \in I$, the non-leading coefficients of that polynomial lie in $I \cap R = I_R$, so $a^{|G|} \in I_R\Lambda$; as $I$ is finitely generated, $I^N \subseteq I_R\Lambda \subseteq I$ for some $N$, so the two topologies agree. The $I_R$-adic completion of the finitely generated $R$-module $\Lambda$ is $\Lambda \otimes_R R^\wedge_{I_R}$, and $R^\wedge_{I_R}$ is flat over $R$. Finally, $\hat H^*(G,-)$ is computed by $\Hom_{\bb{W}[G]}(P_\bullet,-)$ for a complete resolution $P_\bullet$ by finitely generated free $\bb{W}[G]$-modules. On a $\Lambda$-module with $R$-linear $G$-action this is a complex of finitely generated $R$-modules with $R$-linear differentials, so $\hat H^n(G,\Lambda)$ is a finitely generated $R$-module; and since each $P_i$ is finitely generated free, $\Hom_{\bb{W}[G]}(P_i, \Lambda \otimes_R R^\wedge) = \Hom_{\bb{W}[G]}(P_i,\Lambda)\otimes_R R^\wedge$, and flatness carries this through cohomology. If $I_R$ kills $\hat H^n$, then $\hat H^n \otimes_R R^\wedge = \hat H^n \otimes_{R/I_R} R^\wedge/I_RR^\wedge = \hat H^n$, because $R^\wedge/I_R R^\wedge = R/I_R$.
\end{proof}

\begin{lemma} \label{idealhop} \stmt{(ideal hop)  
$$\hat{H}^*(C_p \rtimes C_{(p-1)^2}, (E_{p-1})_*) \simeq \hat{H}^*(C_p \rtimes C_{(p-1)^2}, \Sym(\bar{\rho}))[\Delta^{-1}].$$}
\end{lemma} 

\begin{proof} Put $\Lambda := \Sym(\bar\rho)[d^{-1}]$. By Theorem \ref{theoremd}, $(E_{p-1})_* = \Lambda^\wedge_I$ with $I = (p, \delta_i)$, and $\Delta$ is a unit multiple of a power of $d$, so the right-hand side is $\hat H^*(G',\Lambda)$ by \cite[Lemma 8.5(i)]{ray}. The claim is therefore that completing at $I$ does nothing to Tate cohomology, and by Lemma \ref{lem:basechange} it suffices to show that $I_R := I \cap \Lambda^{G'}$ acts by zero on $\hat H^*(G',\Lambda)$. Since $\hat H^*(G',\Lambda)$ is a module over the ring $\hat H^0(G',\Lambda)$ and an invariant element acts through its class there, it suffices that $I_R$ maps to zero in $\hat H^0(G',\Lambda)$.

Write $J := (\delta_1,\dots,\delta_{p-1})$, so $I = (p) + J$. The element $p$ is already zero in $\hat H^0$: Tate cohomology is killed by $|G'| = p(p-1)^2$, and $(p-1)^2$ is a unit. For $J \cap R$ we pass to $\Lambda/J$. There $x_0 = x_1 = \cdots = x_{p-1}$, so $d = x_0^{\,p}$, and the relation $x_0 + \cdots + x_{p-1} = 0$ of $\bar\rho$ reads $p\,x_0 = 0$; since $d$ is inverted, $x_0$ is a unit and $p = 0$. So
$$\Lambda/J \;=\; \bb{F}[x_0^{\pm1}], \qquad \bb{F} := \bb{W}/p,$$
with $C_p$ acting trivially and $\tau$ by $x_0 \mapsto \eta x_0$. By the collapse \cite[Lemma 8.4]{ray}, $\hat H^0(G',\Lambda/J) = \hat H^0(C_p,\Lambda/J)^{C_{(p-1)^2}} = \bb{F}[x_0^{\pm1}]^{\langle\tau\rangle}$, a Laurent polynomial ring in a power of $x_0$. On the other side, Corollary \ref{badbitch_meat} gives $\hat H^0(G',\Lambda) = \bb{F}[\Delta^{\pm1}]$, and the ring map $\hat H^0(G',\Lambda) \to \hat H^0(G',\Lambda/J)$ sends $\Delta$, a unit times a power of $d$, to a unit times a nonzero power of $x_0$; a map out of a Laurent polynomial ring that sends the variable to a nonzero monomial is injective. Now let $r \in J \cap R$. It acts by zero on $\Lambda/J$, hence on $\hat H^0(G',\Lambda/J)$; by naturality, its class in $\hat H^0(G',\Lambda)$ maps to zero under an injective map, so it is zero. Thus $I_R$ acts by zero on $\hat H^*(G',\Lambda)$, and Lemma \ref{lem:basechange} finishes the proof.
\end{proof}

\begin{manualtheorem}{E} \label{theoreme} \stmt{Let $R$ be the graded ring 
$$R := \bb{W}/p[\alpha, \beta^{\pm}, \Delta^{\pm}]/(\alpha^2).$$
The Tate cohomology of the $C_p \rtimes C_{(p-1)^2}$-module $(E_{p-1})_*$ is: 
$$R \simeq \hat{H}^*(C_p \rtimes C_{(p-1)^2}, (E_{p-1})_*).$$
where $|\alpha| = (1, 2(p-1))$, $|\beta| = (2, 2p(p-1))$, and $|\Delta| = (0, 2p(p-1)^2).$}
\end{manualtheorem}

\begin{proof} 
This follows from Theorem \ref{badbitch_meat} combined with Lemma \ref{idealhop}. The class $\beta$ is a unit because $b$ is, and $\beta^{-1}=d^{p-1}b^{-1}$ is again invariant. The ring agrees with Symonds' computation \cite[Thm.\ 1.1]{symonds}, reached there by group-cohomological means.
\end{proof}

\begin{remark} The classes $c$ and $bd$ represent $\alpha$ and $\beta$ in the homotopy groups of spheres, up to a unit. That is why we notate them as such in Theorem \ref{theoreme}. \end{remark}

\bibliographystyle{alpha}
\bibliography{biblio}

@article{goerss,
  title={Quasi-coherent sheaves on the moduli stack of formal groups},
  author={Goerss, Paul G},
  journal={arXiv preprint arXiv:0802.0996},
  year={2008}
}

@article{coleman,
  title={On the Frobenius endomorphisms of Fermat and Artin-Schreier curves},
  author={Coleman, Robert F},
  journal={Proceedings of the American Mathematical Society},
  volume={102},
  number={3},
  pages={463--466},
  year={1988}
}

@misc{ray,
  author = {Ray, Rin},
  title={Modeling group actions on stacks (especially the {L}ubin--{T}ate action)},
  year   = {2025},
  howpublished = {arXiv:2507.00309v2},
}

@article{st,
title = {Correspondance Serre–Tate},
journal = {SMF 2015},
year = {1964}, 
volume = {2},
author = {Pierre Colmez and Jean-Pierre Serre}
}

@article {gm,
    AUTHOR = {Gorbounov, V. and Mahowald, M.},
     TITLE = {Formal completion of the {J}acobians of plane curves and
              higher real {$K$}-theories},
   JOURNAL = {J. Pure Appl. Algebra},
  FJOURNAL = {Journal of Pure and Applied Algebra},
    VOLUME = {145},
      YEAR = {2000},
    NUMBER = {3},
     PAGES = {293--308},
      ISSN = {0022-4049},
   MRCLASS = {14L05 (14H05 55N22)},
  MRNUMBER = {1733427},
MRREVIEWER = {Neil P. Strickland},
       DOI = {10.1016/S0022-4049(98)00134-0},
       URL = {https://doi.org/10.1016/S0022-4049(98)00134-0},
}

@article{champs,
  author  = {Bertin, Jos\'e and Romagny, Matthieu},
  title   = {Champs de {H}urwitz},
  journal = {M\'em. Soc. Math. Fr. (N.S.)},
  volume  = {125--126},
  year    = {2011},
}

@article{wewers,
  author  = {Wewers, Stefan},
  title   = {Formal deformation of curves with group scheme action},
  journal = {Ann. Inst. Fourier (Grenoble)},
  volume  = {55},
  year    = {2005},
  number  = {4},
  pages   = {1105--1143},
}

@article{hillphd,
title = {Computational Methods for Higher Real K-Theory with Applications to Tmf}, 
journal = {Ph.D. Thesis},
author = {Michael Hill},
year = {2006}
}

@article{Hewett,
title = {Finite Subgroups of Division Algebras over Local Fields},
journal = {Journal of Algebra},
volume = {173},
number = {3},
pages = {518-548},
year = {1995},
issn = {0021-8693},
doi = {https://doi.org/10.1006/jabr.1995.1101},
url = {https://www.sciencedirect.com/science/article/pii/S0021869385711015},
author = {T. Hewett}
}

@article {devhop,
    AUTHOR = {Devinatz, Ethan S. and Hopkins, Michael J.},
     TITLE = {Homotopy fixed point spectra for closed subgroups of the
              {M}orava stabilizer groups},
   JOURNAL = {Topology},
  FJOURNAL = {Topology. An International Journal of Mathematics},
    VOLUME = {43},
      YEAR = {2004},
    NUMBER = {1},
     PAGES = {1--47},
      ISSN = {0040-9383},
   MRCLASS = {55P43 (55N22 55T15)},
  MRNUMBER = {2030586},
MRREVIEWER = {R. E. Stong},
       DOI = {10.1016/S0040-9383(03)00029-6},
       URL = {https://doi.org/10.1016/S0040-9383(03)00029-6},
}

@article {nave,
    AUTHOR = {Nave, Lee S.},
     TITLE = {The {S}mith-{T}oda complex {$V((p+1)/2)$} does not exist},
   JOURNAL = {Ann. of Math. (2)},
  FJOURNAL = {Annals of Mathematics. Second Series},
    VOLUME = {171},
      YEAR = {2010},
    NUMBER = {1},
     PAGES = {491--509},
      ISSN = {0003-486X},
   MRCLASS = {55Q10 (55P42 55Q51)},
  MRNUMBER = {2630045},
MRREVIEWER = {Paul G. Goerss},
       DOI = {10.4007/annals.2010.171.491},
       URL = {https://doi.org/10.4007/annals.2010.171.491},
}

@article {taf,
    AUTHOR = {Behrens, Mark and Lawson, Tyler},
     TITLE = {Topological automorphic forms},
   JOURNAL = {Mem. Amer. Math. Soc.},
  FJOURNAL = {Memoirs of the American Mathematical Society},
    VOLUME = {204},
      YEAR = {2010},
    NUMBER = {958},
     PAGES = {xxiv+141},
      ISSN = {0065-9266},
      ISBN = {978-0-8218-4539-4},
   MRCLASS = {55N34 (11F23 14G35 55N22 55P43)},
  MRNUMBER = {2640996},
MRREVIEWER = {Paul G. Goerss},
       DOI = {10.1090/S0065-9266-09-00573-0},
       URL = {https://doi.org/10.1090/S0065-9266-09-00573-0},
}

@incollection {taf2,
    AUTHOR = {Behrens, Mark},
     TITLE = {Topological modular and automorphic forms},
 BOOKTITLE = {Handbook of homotopy theory},
    SERIES = {CRC Press/Chapman Hall Handb. Math. Ser.},
     PAGES = {221--261},
 PUBLISHER = {CRC Press, Boca Raton, FL},
      YEAR = {[2020] \copyright 2020},
   MRCLASS = {55N34 (55N22 55P42)},
  MRNUMBER = {4197986},
MRREVIEWER = {Lennart Meier},
}

@incollection {Ravenel,
    AUTHOR = {Ravenel, Douglas C.},
     TITLE = {Toward higher chromatic analogs of elliptic cohomology},
 BOOKTITLE = {Elliptic cohomology},
    SERIES = {London Math. Soc. Lecture Note Ser.},
    VOLUME = {342},
     PAGES = {286--305},
 PUBLISHER = {Cambridge Univ. Press, Cambridge},
      YEAR = {2007},
   MRCLASS = {55N34 (14H40 55N22)},
  MRNUMBER = {2330518},
MRREVIEWER = {Mark J. Behrens},
       DOI = {10.1017/CBO9780511721489.015},
       URL = {https://doi.org/10.1017/CBO9780511721489.015},
}

@article {Ravenelii,
    AUTHOR = {Ravenel, Douglas C.},
     TITLE = {Toward higher chromatic analogs of elliptic cohomology. {II}},
   JOURNAL = {Homology Homotopy Appl.},
  FJOURNAL = {Homology, Homotopy and Applications},
    VOLUME = {10},
      YEAR = {2008},
    NUMBER = {3},
     PAGES = {335--368},
      ISSN = {1532-0073},
   MRCLASS = {55N34 (14G15 14H40 55N22)},
  MRNUMBER = {2475628},
MRREVIEWER = {Markus Szymik},
       URL = {http://projecteuclid.org/euclid.hha/1251832478},
}

@article{li2018newton,
  title={Newton polygon stratification of the Torelli locus in PEL-type Shimura varieties},
  author={Li, Wanlin and Mantovan, Elena and Pries, Rachel and Tang, Yunqing},
  journal={arXiv preprint arXiv:1811.00604},
  year={2018}
}

@article{deligne1986monodromy,
  title={Monodromy of hypergeometric functions and non-lattice integral monodromy},
  author={Deligne, Pierre and Mostow, George Daniel},
  journal={Publications Math{\'e}matiques de l'IH{\'E}S},
  volume={63},
  pages={5--89},
  year={1986}
}

@book{rappoportzink,
 ISBN = {9780691027814},
 URL = {http://www.jstor.org/stable/j.ctt1b7x72g},
 author = {Rapoport, M. and Zink, Th.},
 publisher = {Princeton University Press},
 title = {Period Spaces for "p"-divisible Groups (AM-141)},
 urldate = {2025-06-30},
 year = {1996}
}

@article {carayol90,
author = {Henri Carayol},
title = {Nonabelian Lubin-Tate theory, in Automorphic forms, Shimura
varieties, and L-functions, vol II }, 
journal = {Perspect. Math., vol. 11}, 
pages = {p. 15-39.},
year = 1990, 
}

@article {Rav,
    AUTHOR = {Ravenel, Douglas C.},
     TITLE = {The non-existence of odd primary {A}rf invariant elements in
              stable homotopy},
   JOURNAL = {Math. Proc. Cambridge Philos. Soc.},
  FJOURNAL = {Mathematical Proceedings of the Cambridge Philosophical
              Society},
    VOLUME = {83},
      YEAR = {1978},
    NUMBER = {3},
     PAGES = {429--443},
      ISSN = {0305-0041},
   MRCLASS = {55E45},
  MRNUMBER = {474291},
MRREVIEWER = {Donald W. Kahn},
       DOI = {10.1017/S0305004100054712},
       URL = {https://doi.org/10.1017/S0305004100054712},
}

@article {HHR,
    AUTHOR = {Hill, M. A. and Hopkins, M. J. and Ravenel, D. C.},
     TITLE = {On the nonexistence of elements of {K}ervaire invariant one},
   JOURNAL = {Ann. of Math. (2)},
  FJOURNAL = {Annals of Mathematics. Second Series},
    VOLUME = {184},
      YEAR = {2016},
    NUMBER = {1},
     PAGES = {1--262},
      ISSN = {0003-486X},
   MRCLASS = {55P91 (55N22 55P42 55Q45 55T15 55U35 57R15)},
  MRNUMBER = {3505179},
MRREVIEWER = {Paul G. Goerss},
       DOI = {10.4007/annals.2016.184.1.1},
       URL = {https://doi.org/10.4007/annals.2016.184.1.1},
}

@article {henn,
    AUTHOR = {Henn, Hans-Werner},
     TITLE = {Centralizers of elementary abelian {$p$}-subgroups and
              mod-{$p$} cohomology of profinite groups},
   JOURNAL = {Duke Math. J.},
  FJOURNAL = {Duke Mathematical Journal},
    VOLUME = {91},
      YEAR = {1998},
    NUMBER = {3},
     PAGES = {561--585},
      ISSN = {0012-7094},
   MRCLASS = {20J05 (20E18)},
  MRNUMBER = {1604171},
MRREVIEWER = {Norberto Gavioli},
       DOI = {10.1215/S0012-7094-98-09121-9},
       URL = {https://doi.org/10.1215/S0012-7094-98-09121-9},
}

@book {moduli,
    AUTHOR = {Harris, Joe and Morrison, Ian},
     TITLE = {Moduli of curves},
    SERIES = {Graduate Texts in Mathematics},
    VOLUME = {187},
 PUBLISHER = {Springer-Verlag, New York},
      YEAR = {1998},
     PAGES = {xiv+366},
      ISBN = {0-387-98438-0; 0-387-98429-1},
   MRCLASS = {14H10 (14-02 14D20 14D22)},
  MRNUMBER = {1631825},
MRREVIEWER = {R. F. Lax},
}

@article {bm,
    AUTHOR = {Bertin, Jos\'{e} and M\'{e}zard, Ariane},
     TITLE = {D\'{e}formations formelles des rev\^{e}tements sauvagement ramifi\'{e}s de
              courbes alg\'{e}briques},
   JOURNAL = {Invent. Math.},
  FJOURNAL = {Inventiones Mathematicae},
    VOLUME = {141},
      YEAR = {2000},
    NUMBER = {1},
     PAGES = {195--238},
      ISSN = {0020-9910},
   MRCLASS = {14D15 (14D20 14H10 14H30 14H37)},
  MRNUMBER = {1767273},
MRREVIEWER = {Akio Tamagawa},
       DOI = {10.1007/s002220000071},
       URL = {https://doi.org/10.1007/s002220000071},
}

@article {pstra,
    AUTHOR = {Pstragowski, Piotr},
     TITLE = {Chromatic homotopy theory is algebraic when {$p>n^2+n+1$}},
   JOURNAL = {Adv. Math.},
  FJOURNAL = {Advances in Mathematics},
    VOLUME = {391},
      YEAR = {2021},
     PAGES = {Paper No. 107958, 37},
      ISSN = {0001-8708},
   MRCLASS = {55P43 (18G80 18N60)},
  MRNUMBER = {4300917},
MRREVIEWER = {Steffen Sagave},
       DOI = {10.1016/j.aim.2021.107958},
       URL = {https://doi.org/10.1016/j.aim.2021.107958},
}

@article{gorbounovsymonds,
  title={Toward the homotopy groups of the higher real K-theory EO2},
  author={Gorbounov, Vassily and Symonds, Peter},
  journal={Contemporary Mathematics},
  volume={220},
  pages={103--116},
  year={1998},
  publisher={AMERICAN MATHEMATICAL SOCIETY}
}

@incollection {katz,
    AUTHOR = {Katz, Nicholas M.},
     TITLE = {Crystalline cohomology, {D}ieudonn\'{e} modules, and {J}acobi
              sums},
 BOOKTITLE = {Automorphic forms, representation theory and arithmetic
              ({B}ombay, 1979)},
    SERIES = {Tata Inst. Fund. Res. Studies in Math.},
    VOLUME = {10},
     PAGES = {165--246},
 PUBLISHER = {Tata Institute of Fundamental Research, Bombay},
      YEAR = {1981},
   MRCLASS = {14G10 (10G20 12G05 14L05)},
  MRNUMBER = {633662},
MRREVIEWER = {Thomas Zink},
}

@article {ghmr,
    AUTHOR = {Goerss, P. and Henn, H.-W. and Mahowald, M. and Rezk, C.},
     TITLE = {A resolution of the {$K(2)$}-local sphere at the prime 3},
   JOURNAL = {Ann. of Math. (2)},
  FJOURNAL = {Annals of Mathematics. Second Series},
    VOLUME = {162},
      YEAR = {2005},
    NUMBER = {2},
     PAGES = {777--822},
      ISSN = {0003-486X},
   MRCLASS = {55Q45 (55P60)},
  MRNUMBER = {2183282},
MRREVIEWER = {Mark J. Behrens},
       DOI = {10.4007/annals.2005.162.777},
       URL = {https://doi.org/10.4007/annals.2005.162.777},
}

@article {vesna,
    AUTHOR = {Stojanoska, Vesna},
     TITLE = {Duality for topological modular forms},
   JOURNAL = {Doc. Math.},
  FJOURNAL = {Documenta Mathematica},
    VOLUME = {17},
      YEAR = {2012},
     PAGES = {271--311},
      ISSN = {1431-0635},
   MRCLASS = {55N34 (14D23 14H52 55N91 55P43)},
  MRNUMBER = {2946825},
MRREVIEWER = {Paul G. Goerss},
}

@article {chromsplit,
    AUTHOR = {Beaudry, Agn\`es},
     TITLE = {The chromatic splitting conjecture at {$n = p = 2$}},
   JOURNAL = {Geom. Topol.},
  FJOURNAL = {Geometry \& Topology},
    VOLUME = {21},
      YEAR = {2017},
    NUMBER = {6},
     PAGES = {3213--3230},
      ISSN = {1465-3060},
   MRCLASS = {55P60 (55Q45)},
  MRNUMBER = {3692966},
MRREVIEWER = {J\'{e}r\^{o}me Scherer},
       DOI = {10.2140/gt.2017.21.3213},
       URL = {https://doi.org/10.2140/gt.2017.21.3213},
}

@article {moore2,
    AUTHOR = {Ichigi, Ippei and Shimomura, Katsumi},
     TITLE = {Homotopy groups of generalized {$E(2)$}-local {M}oore spectra
              at the prime three},
   JOURNAL = {Hiroshima Math. J.},
  FJOURNAL = {Hiroshima Mathematical Journal},
    VOLUME = {35},
      YEAR = {2005},
    NUMBER = {1},
     PAGES = {125--142},
      ISSN = {0018-2079},
   MRCLASS = {55Q52 (55P42)},
  MRNUMBER = {2131379},
       URL = {http://projecteuclid.org/euclid.hmj/1150922489},
}

@book {cyc,
    AUTHOR = {Washington, Lawrence C.},
     TITLE = {Introduction to cyclotomic fields},
    SERIES = {Graduate Texts in Mathematics},
    VOLUME = {83},
 PUBLISHER = {Springer-Verlag, New York},
      YEAR = {1982},
     PAGES = {xi+389},
      ISBN = {0-387-90622-3},
   MRCLASS = {11-01 (11R18 11R23)},
  MRNUMBER = {718674},
MRREVIEWER = {T. Mets\"{a}nkyl\"{a}},
       DOI = {10.1007/978-1-4684-0133-2},
       URL = {https://doi.org/10.1007/978-1-4684-0133-2},
}

@article {forms,
    AUTHOR = {Morava, Jack},
     TITLE = {Forms of {$K$}-theory},
   JOURNAL = {Math. Z.},
  FJOURNAL = {Mathematische Zeitschrift},
    VOLUME = {201},
      YEAR = {1989},
    NUMBER = {3},
     PAGES = {401--428},
      ISSN = {0025-5874},
   MRCLASS = {55N22 (11S31)},
  MRNUMBER = {999737},
MRREVIEWER = {P. S. Landweber},
       DOI = {10.1007/BF01214905},
       URL = {https://doi.org/10.1007/BF01214905},
}

@article {manin,
    AUTHOR = {Manin, Ju. I.},
     TITLE = {Theory of commutative formal groups over fields of finite

              characteristic.},
   JOURNAL = {Uspehi Mat. Nauk},
  FJOURNAL = {Akademiya Nauk SSSR i Moskovskoe Matematicheskoe Obshchestvo.
              Uspekhi Matematicheskikh Nauk},
    VOLUME = {18},
      YEAR = {1963},
    NUMBER = {6 (114)},
     PAGES = {3--90},
      ISSN = {0042-1316},
   MRCLASS = {14.50 (14.49)},
  MRNUMBER = {157972},
MRREVIEWER = {E.\ C.\ Dade},
}

@article{LT,
     author = {Lubin, J. and Tate, J.},
     title = {Formal moduli for one-parameter formal {Lie} groups},
     journal = {Bulletin de la Soci\'et\'e Math\'ematique de France},
     pages = {49--59},
     publisher = {Soci\'et\'e math\'ematique de France},
     volume = {94},
     year = {1966},
     doi = {10.24033/bsmf.1633},
     zbl = {0156.04105},
     mrnumber = {39 #214},
     language = {en},
     url = {http://www.numdam.org/articles/10.24033/bsmf.1633/}
}

@article {tafhill,
    AUTHOR = {Hill, Michael and Lawson, Tyler},
     TITLE = {Automorphic forms and cohomology theories on {S}himura curves
              of small discriminant},
   JOURNAL = {Adv. Math.},
  FJOURNAL = {Advances in Mathematics},
    VOLUME = {225},
      YEAR = {2010},
    NUMBER = {2},
     PAGES = {1013--1045},
      ISSN = {0001-8708},
   MRCLASS = {55P43 (11F23 11G18 14G35 55P42)},
  MRNUMBER = {2671186},
MRREVIEWER = {Paul G. Goerss},
       DOI = {10.1016/j.aim.2010.03.009},
       URL = {https://doi.org/10.1016/j.aim.2010.03.009},
}

@article {bouj,
  title={Finite subgroups of the extended Morava stabilizer groups},
  author={C{\'e}dric Bujard},
  journal={arXiv: Algebraic Topology},
  year={2012}
}

@article{BSS20,
  author  = {Barthel, Tobias and Schlank, Tomer M. and Stapleton, Nathaniel},
  title   = {Chromatic homotopy theory is asymptotically algebraic},
  journal = {Invent. Math.},
  volume  = {220},
  year    = {2020},
  pages   = {737--845}
}

@incollection{Rez98,
  author    = {Rezk, Charles},
  title     = {Notes on the {H}opkins--{M}iller theorem},
  booktitle = {Homotopy theory via algebraic geometry and group representations},
  series    = {Contemp. Math.},
  volume    = {220},
  publisher = {Amer. Math. Soc.},
  year      = {1998},
  pages     = {313--366}
}

@incollection{GH04,
  author    = {Goerss, Paul G. and Hopkins, Michael J.},
  title     = {Moduli spaces of commutative ring spectra},
  booktitle = {Structured ring spectra},
  series    = {London Math. Soc. Lecture Note Ser.},
  volume    = {315},
  publisher = {Cambridge Univ. Press},
  year      = {2004},
  pages     = {151--200}
}

@book{Mes72,
  author    = {Messing, William},
  title     = {The crystals associated to {B}arsotti--{T}ate groups: with applications to abelian schemes},
  series    = {Lecture Notes in Mathematics},
  volume    = {264},
  publisher = {Springer-Verlag},
  year      = {1972}
}

@article{KO,
  author  = {Katz, Nicholas M. and Oda, Tadao},
  title   = {On the differentiation of de {R}ham cohomology classes with respect to parameters},
  journal = {J. Math. Kyoto Univ.},
  volume  = {8},
  year    = {1968},
  pages   = {199--213}
}

@article{rognes,
  author  = {Rognes, John},
  title   = {Galois extensions of structured ring spectra. {S}tably dualizable groups},
  journal = {Mem. Amer. Math. Soc.},
  volume  = {192},
  number  = {898},
  year    = {2008}
}

@article{mrt,
  author  = {M\'ezard, Ariane and Romagny, Matthieu and Tossici, Dajano},
  title   = {Sekiguchi--{S}uwa theory revisited},
  journal = {J. Th\'eor. Nombres Bordeaux},
  volume  = {26},
  year    = {2014},
  pages   = {163--200}
}

@article{ss,
  author  = {Sekiguchi, Tsutomu and Suwa, Noriyuki},
  title   = {A note on extensions of algebraic and formal groups. {IV}. {K}ummer--{A}rtin--{S}chreier--{W}itt theory of degree {$p^2$}},
  journal = {Tohoku Math. J. (2)},
  volume  = {53},
  year    = {2001},
  pages   = {203--240},
}

@article{subrao,
  author  = {Subrao, Doré},
  title   = {The {$p$}-rank of {A}rtin--{S}chreier curves},
  journal = {Manuscripta Math.},
  volume  = {16},
  year    = {1975},
  pages   = {169--193}
}

@article{wiman,
  author  = {Wiman, Anders},
  title   = {\"{U}ber die hyperelliptischen {C}urven und diejenigen vom {G}eschlechte {$p=3$}, welche eindeutige {T}ransformationen in sich zulassen},
  journal = {Bihang Kongl. Svenska Vetenskaps-Akademiens Handlingar},
  volume  = {21},
  year    = {1895}
}

@article{harvey,
  author  = {Harvey, William J.},
  title   = {Cyclic groups of automorphisms of a compact {R}iemann surface},
  journal = {Quart. J. Math. Oxford Ser. (2)},
  volume  = {17},
  year    = {1966},
  pages   = {86--97}
}

@article{valmadan,
  author = {Valentini, Robert C. and Madan, Manohar L.},
  title  = {A {H}auptsatz of {L}.{E}. {D}ickson and {A}rtin--{S}chreier extensions},
  journal= {J. Reine Angew. Math.},
  volume = {318},
  year   = {1980},
  pages  = {156--177}
}

@article{LT66,
  author = {Lubin, Jonathan and Tate, John},
  title  = {Formal moduli for one-parameter formal {L}ie groups},
  journal= {Bull. Soc. Math. France},
  volume = {94},
  year   = {1966},
  pages  = {49--59}
}

@incollection{Faltings,
  author    = {Faltings, Gerd},
  title     = {A relation between two moduli spaces studied by {V}. {G}. {D}rinfeld},
  booktitle = {Algebraic number theory and algebraic geometry},
  series    = {Contemp. Math.},
  volume    = {300},
  publisher = {Amer. Math. Soc.},
  year      = {2002},
  pages     = {115--129},
}

@misc{higherlevels,
  author = {Ray, Rin},
  title={Only the gauge grows: the {L}ubin--{T}ate action at height {$p^{k-1}(p-1)$} via surgeries on curves},
  year = {2026},
  howpublished = {Preprint},
}

@book{hartshorne,
  author    = {Hartshorne, Robin},
  title     = {Algebraic geometry},
  series    = {Graduate Texts in Mathematics},
  volume    = {52},
  publisher = {Springer-Verlag},
  address   = {New York},
  year      = {1977},
}

@misc{BSSW,
  author = {Barthel, Tobias and Schlank, Tomer M. and Stapleton, Nathaniel and Weinstein, Jared},
  title = {On the rationalization of the {$K(n)$}-local sphere},
  year = {2024},
  howpublished = {arXiv:2402.00960}
}

@incollection{symonds,
  author = {Symonds, Peter},
  title = {The {T}ate--{F}arrell cohomology of the {M}orava stabilizer group {$S_{p-1}$} with coefficients in {$E_{p-1}$}},
  booktitle = {Homotopy theory: relations with algebraic geometry, group cohomology, and algebraic {$K$}-theory},
  series = {Contemp. Math.},
  volume = {346},
  publisher = {Amer. Math. Soc., Providence, RI},
  year = {2004},
  pages = {485--492}
}

@misc{BLLLSZ,
  author = {Bobkova, Irina and Lachmann, Andrea and Li, Ang and Lima, Alicia and Stojanoska, Vesna and Zhang, Adela Yiyu},
  title = {Bounding the {$K(p-1)$}-local exotic {P}icard group at {$p>3$}},
  year = {2024},
  howpublished = {arXiv:2403.15572}
}

@incollection{henn07,
  author = {Henn, Hans-Werner},
  title = {On finite resolutions of {$K(n)$}-local spheres},
  booktitle = {Elliptic cohomology},
  series = {London Math. Soc. Lecture Note Ser.},
  volume = {342},
  publisher = {Cambridge Univ. Press, Cambridge},
  year = {2007},
  pages = {122--169}
}

@article{bobkovagoerss,
  author = {Bobkova, Irina and Goerss, Paul G.},
  title = {Topological resolutions in {$K(2)$}-local homotopy theory at the prime 2},
  journal = {J. Topol.},
  volume = {11},
  year = {2018},
  number = {4},
  pages = {917--956}
}

@article{HMS,
  author = {Heard, Drew and Mathew, Akhil and Stojanoska, Vesna},
  title = {Picard groups of higher real {$K$}-theory spectra at height {$p-1$}},
  journal = {Compos. Math.},
  volume = {153},
  year = {2017},
  number = {9},
  pages = {1820--1854}
}

@article{BBGS,
  author = {Barthel, Tobias and Beaudry, Agn\`es and Goerss, Paul G. and Stojanoska, Vesna},
  title = {Gross--{H}opkins duals of higher real {$K$}-theory spectra},
  journal = {Trans. Amer. Math. Soc.},
  volume = {372},
  year = {2019},
  number = {5},
  pages = {3347--3368}
}

@misc{heardtate,
  author = {Heard, Drew},
  title = {The {T}ate spectrum of the higher real {$K$}-theories at height {$n=p-1$}},
  year = {2015},
  howpublished = {arXiv:1501.07759}
}

@article{hahnshi,
  author = {Hahn, Jeremy and Shi, XiaoLin Danny},
  title = {Real orientations of {L}ubin--{T}ate spectra},
  journal = {Invent. Math.},
  volume = {221},
  year = {2020},
  number = {3},
  pages = {731--776}
}

@misc{BMRSSWZ,
  author = {Barthel, Tobias and Mann, Lucas and Ray, Rin and Schlank, Tomer M. and Senger, Andrew and Weinstein, Jared and Zhou, Xinyu},
  title = {On the chromatic splitting conjecture in coheight 1},
  year = {2026},
  note = {Preprint}
}

@article{harbaterabh,
  author = {Harbater, David},
  title = {Abhyankar's conjecture on {G}alois groups over curves},
  journal = {Invent. Math.},
  volume = {117},
  year = {1994},
  number = {1},
  pages = {1--25}
}

@article{raynaudabh,
  author = {Raynaud, Michel},
  title = {Rev\^etements de la droite affine en caract\'eristique {$p>0$} et conjecture d'{A}bhyankar},
  journal = {Invent. Math.},
  volume = {116},
  year = {1994},
  number = {1-3},
  pages = {425--462}
}

@article{lazard,
  author = {Lazard, Michel},
  title = {Groupes analytiques {$p$}-adiques},
  journal = {Inst. Hautes \'Etudes Sci. Publ. Math.},
  number = {26},
  year = {1965},
  pages = {389--603}
}
\end{document}